\documentclass[reqno,11pt]{amsart}
\usepackage{amsmath, latexsym, amsfonts, amssymb, amsthm, amscd}
\usepackage[rightcaption]{sidecap}
\usepackage{multirow,stmaryrd,appendix}
\usepackage{color}
\usepackage{graphicx,epsf,psfrag,caption,subcaption}
\usepackage[hidelinks, pdfstartview=FitV, pagebackref=false]{hyperref}

\numberwithin{equation}{section}

\newtheorem{theorem}{Theorem}[section]
\newtheorem{lemma}[theorem]{Lemma}
\newtheorem{proposition}[theorem]{Proposition}

\newtheorem{rem}[theorem]{Remark}
\newtheorem{conjecture}[theorem]{Conjecture}

\newcommand{\ind}{\mathbf{1}}

\renewcommand{\tilde}{\widetilde}

\newcommand{\cA}{{\ensuremath{\mathcal A}} }

\newcommand{\cH}{{\ensuremath{\mathcal H}} }

\newcommand{\cZ}{{\ensuremath{\mathcal Z}} }

\newcommand{\bP}{{\ensuremath{\mathbf P}} }
\newcommand{\bE}{{\ensuremath{\mathbf E}} }

\DeclareMathSymbol{\leqslant}{\mathalpha}{AMSa}{"36} % nicer `smaller or equal'
\DeclareMathSymbol{\geqslant}{\mathalpha}{AMSa}{"3E} % nicer `larger or equal'
\DeclareMathSymbol{\eset}{\mathalpha}{AMSb}{"3F}     % nicer `emptyset'
\renewcommand{\leq}{\;\leqslant\;}                   % redef. of < or =
\renewcommand{\geq}{\;\geqslant\;}                   % redef. of > or =
\newcommand{\dd}{\,\text{\rm d}}             % a straight d for differentials
\newcommand{\maxtwo}[2]{\max_{\substack{#1 \\ #2}}} % max with 2 lines
\newcommand{\suptwo}[2]{\sup_{\substack{#1 \\ #2}}} % sup with 2 lines
\newcommand{\sumtwo}[2]{\sum_{\substack{#1 \\ #2}}} % sum with 2 lines
\newcommand{\limtwo}[2]{\lim_{\substack{#1 \\ #2}}}     % \lim with 2 lines
\newcommand{\bbE}{{\ensuremath{\mathbb E}} }

\newcommand{\bbN}{{\ensuremath{\mathbb N}} }

\newcommand{\bbP}{{\ensuremath{\mathbb P}} }
\newcommand{\bbQ}{{\ensuremath{\mathbb Q}} }
\newcommand{\bbR}{{\ensuremath{\mathbb R}} }

\newcommand{\bbZ}{{\ensuremath{\mathbb Z}} }

\newcommand{\ga}{\alpha}
\newcommand{\gb}{\beta}
\newcommand{\gd}{\delta}
\newcommand{\gep}{\varepsilon}       % \ge already exists...

\newcommand{\gD}{\Delta}

\newcommand{\go}{\omega}

\newcommand{\gl}{\lambda}

\makeatletter
\def\captionfont@{\footnotesize}
\def\captionheadfont@{\scshape}

\long\def\@makecaption#1#2{%
  \vspace{2mm}
  \setbox\@tempboxa\vbox{\color@setgroup
    \advance\hsize-6pc\noindent
    \captionfont@\captionheadfont@#1\@xp\@ifnotempty\@xp
        {\@cdr#2\@nil}{.\captionfont@\upshape\enspace#2}%
    \unskip\kern-6pc\par
    \global\setbox\@ne\lastbox\color@endgroup}%
  \ifhbox\@ne % the normal case
    \setbox\@ne\hbox{\unhbox\@ne\unskip\unskip\unpenalty\unkern}%
  \fi
  \ifdim\wd\@tempboxa=\z@ % this means caption will fit on one line
    \setbox\@ne\hbox to\columnwidth{\hss\kern-6pc\box\@ne\hss}%
  \else % tempboxa contained more than one line
    \setbox\@ne\vbox{\unvbox\@tempboxa\parskip\z@skip
        \noindent\unhbox\@ne\advance\hsize-6pc\par}%
\fi
  \ifnum\@tempcnta<64 % if the float IS a figure...
    \addvspace\abovecaptionskip
    \moveright 3pc\box\@ne
  \else % if the float IS NOT a figure...
    \moveright 3pc\box\@ne
    \nobreak
    \vskip\belowcaptionskip
  \fi
\relax
}
\makeatother
\def\writefig#1 #2 #3 {\rlap{\kern #1 truecm
\raise #2 truecm \hbox{#3}}}

\newcommand{\tf}{\textsc{f}}

\begin{document}

\title[The disordered generalized Poland-Scheraga model]{Disorder and denaturation transition in the generalized Poland-Scheraga model}

\author{Quentin Berger}
\address{
  Sorbonne Universit\'e, Laboratoire de Probabilit{\'e}s Statistique et Mod\'elisation, UMR 8001, 
            F- 75205 Paris, France
}

\author{Giambattista Giacomin}
\address{
  Universit\'e Paris Diderot, Sorbonne Paris Cit\'e,   Laboratoire de Probabilit{\'e}s Statistique et Mod\'elisation, UMR 8001,
            F- 75205 Paris, France
}

\author{Maha Khatib}
\address{
Universit\'e Libanaise, Laboratoire de Math\'ematiques-EDST, Beyrouth, Liban 
            }

\begin{abstract}
We investigate the generalized Poland-Scheraga model, which is used in the bio-physical literature to model the DNA denaturation transition, in the case where the two strands are allowed to be non-complementary (and to have different lengths). The homogeneous model was recently studied from a mathematical point of view in \cite{cf:GK,cf:BGK}, via a $2$-dimensional renewal approach, with a \emph{loop exponent} $2+\alpha$ (${\alpha>0}$): it was found to undergo a localization/delocalization phase transition of order $\nu = \min(1,\alpha)^{-1}$, together with -- in general -- other phase transitions. In this paper, we turn to the disordered model, and we address the question of the influence of disorder on the denaturation phase transition, that is %to know 
whether adding an arbitrarily small amount of disorder (\textit{i.e.}\ inhomogeneities) affects the critical properties of this transition. Our results are consistent with Harris' predictions for $d$-dimensional disordered systems (here $d=2$). First, we prove that when $\ga<1$ (\textit{i.e.}\ $\nu>d/2$), then disorder is \emph{irrelevant}: the quenched and annealed critical points are equal, and the disordered denaturation phase transition is also of order $\nu=\alpha^{-1}$. On the other hand, when $\alpha>1$, disorder is \textit{relevant}: we prove that the quenched and annealed critical points differ. 
 Moreover, we discuss a number of open problems, in  particular the  smoothing phenomenon that is expected to enter the game  when disorder is relevant. 
%However, we are not able to show the existence of a smoothing phenomenon which would tell that the critical exponents also differ. The marginal case $\alpha=1$ is left aside, essentially because of extra technical difficulties.
\end{abstract}


\maketitle

%\tableofcontents

\section{Introduction of the model and results}

 The analysis of the DNA denaturation phenomenon, \textit{i.e.}\ the unbinding  at high temperature of two strands of
 DNA, has lead to the proposal of a very elementary model, the Poland-Scheraga (PS) model \cite{cf:PSbook}, that turns out to be relevant not only at a conceptual and qualitative level \cite{cf:Fisher,cf:GB}, 
 but also at a quantitative level \cite{cf:Blake1,cf:Blake2}. This model can naturally embody 
 the inhomogeneous character of the DNA polymer, which  is a monomer sequence  of four different types  (A,T, G and C).
 The binding energy for A-T pairs is different from the binding energy for G-C pairs.
 The quantitative analysis is then based on finite length chains with a given sequence of pairs, but in order to analyse general properties of inhomogeneous chains   
 bio-physicists focused on the cases in which the base sequence is the realization of a sequence of random variables, that is often referred to
 as disorder in statistical mechanics. 
 The  PS  model is limited to the case in which  the two strands are of equal length and the $n^{\textrm{th}}$ base of one strand can only bind with the $n^{\textrm{th}}$ base of the other strand: it does not allow {\sl mismatches} or, more generally, {\sl asymmetric loops}, see Fig.~\ref{fig:PS}. 
A less elementary model, the generalized Poland-Scheraga model (gPS)  \cite{cf:GK} allows asymmetric loops, and different length strands are allowed too, see Fig.~\ref{fig:gPS}. 
   
\begin{figure}[hbtp]
	\centering
	\begin{subfigure}[t]{0.4\textwidth}
		\centering
		\includegraphics[width=0.9\textwidth]{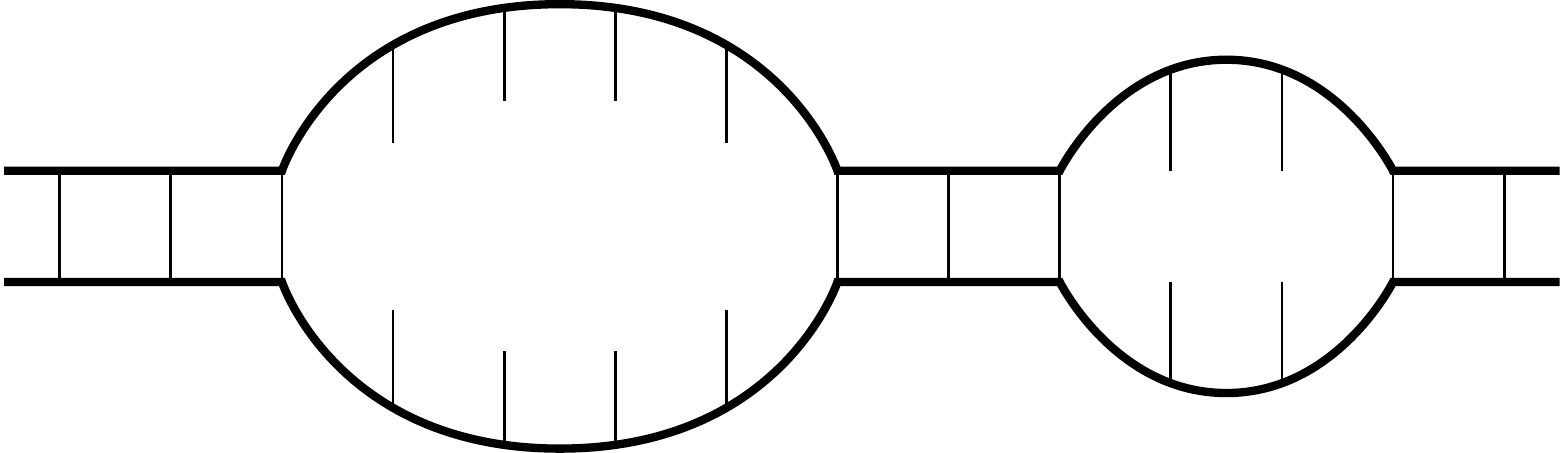}
		\caption{\footnotesize Standard PS model.}\label{fig:PS}		
	\end{subfigure}
	\qquad \qquad
	\begin{subfigure}[t]{0.4\textwidth}
		\centering
		\includegraphics[width=0.9\textwidth]{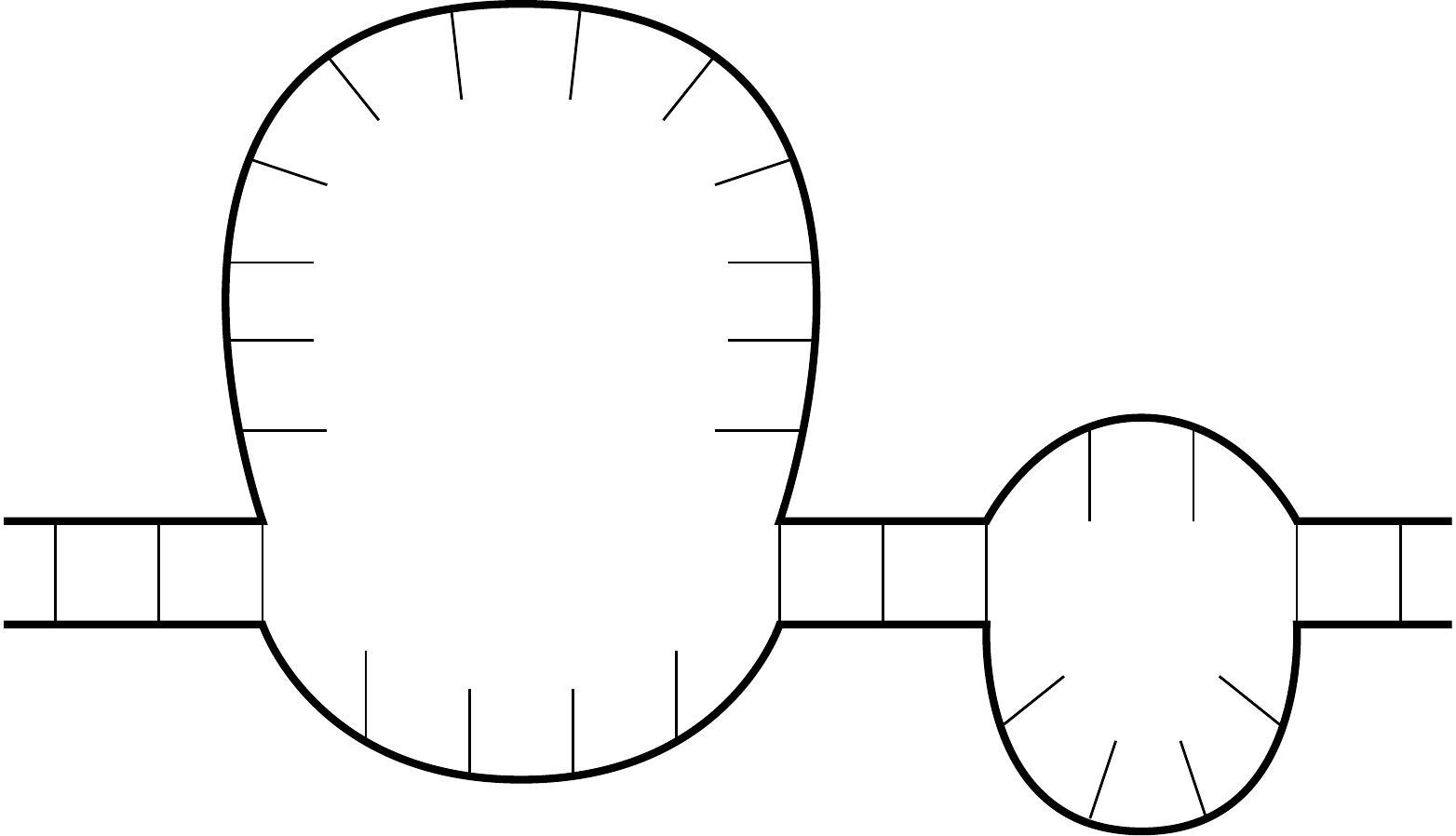}
		\caption{\footnotesize Generalized PS model.}\label{fig:gPS}
	\end{subfigure}
	\caption{Standard v.\ Generalized Poland-Scheraga models. The figure on the left represents the standard PS model: the two strands of DNA have the same length (there are $14$ base pairs in Fig.~\ref{fig:PS}), loops are symmetric (there are $5$ loops of lengths $1$, $1$ loop of length $3$ and $1$ loop of length $5$). The figure on the right represents the generalized PS model: the two strands may have a different number of bases ($22$ for the 'top' one, and $16$ for the 'bottom' one), and loops are allowed to be asymmetric and can be encoded by two numbers $(n,m)$ where $n$ is the length the 'top' strand and $m$ of the 'bottom' strand (the loops in Fig.~\ref{fig:gPS} are from left to right $(1,1),(1,1),(13,5),(1,1),(1,1),(3,5),$ $ (1,1)$).}\label{fig:1}
\end{figure}

   A remarkable feature of the non disordered PS model  (this corresponds to the case in which  all the bases are the same: for example a strand AAA... and a second strand TTT...)
is its solvable character. Notably, one can show that the model has a {\sl denaturation} transition in the limit of infinite strand length, and one can identify the critical point (the critical temperature)
and the critical behavior, \textit{i.e.}\ the nature of the singularity of the free energy at the critical value. %At the denaturation transition the two strands completely unbind. 
Somewhat surprisingly, also the gPS model is exactly solvable, in spite of the fact that it is considerably more complex than the PS model. 
This has been pointed out first in \cite{cf:GOO,cf:GO04,cf:NG06} and a mathematical treatment can be found in \cite{cf:GK}. Let us stress that the higher complexity level of the gPS model is however reflected in a richer behavior.
Notably,  in the gPS model, other phase transitions exist, beyond the denaturation transition. Another relevant remark is that PS and gPS models contain a parameter -- the {\sl loop exponent} -- that, in a mathematical or
theoretical physics perspective, can be chosen arbitrarily and on which depends the critical behavior. In fact in this class of models the critical exponent  depends on this parameter, and arbitrary critical exponents can be  observed by tuning the loop exponent.

Stepping to the disordered model is not (at all) straightforward.  One way to attack the problem is by looking at it as a stability issue: is the transition -- and we will focus 
on the denaturation one -- still present in the model if we introduce some disorder, for example a small amount? And, if it does, what is the new critical value  and is the critical behavior the same as without disorder?  
We refer to \cite[Ch.~5]{cf:GB} for an outline on this general very important issue in statistical mechanics and on the 
renormalization group ideas that lead to the so called \emph{Harris criterion} of disorder irrelevance.  We speak of disorder relevance when 
the disorder, irrespective of its strength, makes the critical behavior of the model different from the one of the non disordered model. Disorder is instead irrelevant if the two critical behaviors coincide for a small disorder strength. 
In the relevant (resp.\ irrelevant) case one can argue that applying a coarse graining procedure makes  the disorder stronger (resp.\ weaker). 
Harris' idea is that disorder (ir)relevance, can be read out of the critical exponent in the non disordered model.

More precisely, Harris criterion says that, if $\nu$ denotes the correlation length exponent of the non-disordered system and $d $ the dimension,  $\nu >2/d$ implies disorder irrelevance, at least if the disorder is not too strong. One also expects disorder relevance if $\nu< 2/d$. The case $\nu= 2/d$ is dubbed marginal and deciding whether disorder is relevant or not is usually a delicate issue, even leaving aside mathematical rigor.
The PS and gPS  models, with their wide spectra of critical behaviors,
therefore become an ideal framework for testing the validity of the physical predictions. 
In fact, the mathematical activity on the PS model (which is one-dimensional) has been very successful. Results include:

 \begin{itemize}
 \item Very complete understanding of the PS model when disorder is irrelevant \cite{cf:A08,cf:G,cf:Lac,cf:Ton08}; 
 \item Precise estimates on the disorder induced shift of the critical point (with respect to the annealed model)  in the relevant disorder case  \cite{cf:AZ09,cf:DGLT09}, and a proof of the fact that
 disorder does change the critical exponent  \cite{cf:GT06,cf:CH13} (without determining the new one: this is an open problem also in the physical literature, even if consensus is starting to emerge about the fact that pinning model in 
 the relevant disorder regime should display a very smooth localization transition,  see \cite{cf:DR,cf:BGL} and references therein);
 \item Determination of whether or not there is a disordered induced critical point shift in the marginal case, and precise estimates of this shift: this issue was controversial in the physical literature \cite{cf:BL16,cf:1GLT10}. In absence of critical point shift, the critical exponent has also been shown to be unchanged by the noise.
Showing that disorder does change the critical behavior when there is a critical point shift at marginality is an open issue, and determining the critical behavior in presence of disorder does not appear to be easier than attacking the same issue in the relevant case \cite{cf:DR}. 
 \end{itemize}

 %This subject has been the object of a lot of attention in theoretical physics \cite{cf:CH97,cf:KM03,cf:KL12} and mathematical physics\cite{cf:A08,cf:AZ09,cf:AZ10,cf:B14,cf:BCPSZ14,cf:BL12,cf:DGLT09,cf:GLT10,cf:1GLT10,cf:GLT11,cf:GT06,cf:GT09}. Allowing non complementarity of the strands and different length, it was shown that the non disordered generalized Poland-Scheraga model (gPS) is exactly solvable  . There are many ways to introduce disorder to this model: the natural way by assigning to each strand a sequence of (correlated or independent) potentials \cite{cf:EON11,cf:GOO,cf:GO04,cf:TN08}, introducing a family of random variables with two indices or adding disorder to a single strand. 

%This paper deals with the question of disorder relevance for the gPS model. The disorder is introduced via a IID family $ {\lbrace \omega _ {n,m} \rbrace}_{(n,m) \in \bbN^2}$ of random variables where the disorder is independent at each site of the two strands. Using the tools developed for the PS model, we will show that in some cases the quenched and annealed critical points differ proving the presence of a relevant disorder regime for the localization/delocalization transition. 

Our aim is to analyze the disordered gPS model and to understand the effect of disorder on the denaturation transition for this generalized, 2-dimensional, model.

\subsection{The generalized Poland-Scheraga model}
\label{sec:model}

Let $\tau = {\lbrace \tau_n \rbrace}_{n\geq 0}= {\lbrace (\tau_n^{(1)}, \tau_n^{(2)}) \rbrace}_{n\geq 0}$ to be a bivariate renewal process, \textit{i.e.}\ $\tau_0 = (0,0)$ and $\{ \tau_{n} - \tau_{n-1}\}_{n \geq 1}$ are identically distributed $\mathbb{N}^2$-valued random vectors. We denote by $\bP$ the law of $\tau$, and we assume that it has inter-arrival distribution $\bP(\tau_1= (n,m))=  K(n+m)$, where 
\begin{equation}
\label{def:Kn}
K(n) := \frac{L(n)}{n^{2+\ga}} \,,
\end{equation}
for some $\ga \ge 0$ and some slowly varying function $L(\cdot)$. Let $\mu := \bE [\tau_1^{(1)}]= \bE [\tau_1^{(2)}]\in (1, \infty]$.  Without loss of generality, we assume that $\tau$ is persistent, \textit{i.e.}\ $\sum_{n,m} K(n+m) =1$. 
A set $\tau={\lbrace \tau_n \rbrace}_{n\geq 0}$ is then interpreted as a two-strand DNA configuration: the $\tau_n^{(1)}$'s monomer of the first strand is attached to the $\tau_n^{(2)}$'s monomer of the second strand. Put differently, the $n^{\rm th}$ loop in the double strand is encoded by $(\tau_n^{(1)}-\tau_{n-1}^{(1)},\tau_n^{(2)}-\tau_{n-1}^{(2)})$, see Figure~\ref{fig:PS} and its caption. We refer to \cite{cf:GK} for further details.

Let $ \omega := { \lbrace \omega_{n,m} \rbrace }_{n,m \in \bbN}$ be a sequence of IID centered random variables (the disorder), taking values in $\bbR$, with law denoted $\bbP$. We assume that the variables $\omega_{n,m}$ are centered, have unit variance and exponential moments of all order, and we set for $\beta \in \bbR$
\begin{equation}
\label{eq:defQ}
Q( \beta) :=  \bbE [ \exp( \beta \omega) ] < \infty \,.
\end{equation}
This choice of disorder is discussed in detail in Section~\ref{sec:misc}.

Given $\beta >0$, $h \in \bbR$ (the pinning parameter) and $N,M \in \bbN$, we define $\bP_{N,M}^{\beta,h,\omega}$ a measure whose
Radon-Nikodym derivative w.r.t. $\bP$ is given by
\begin{equation}
\frac{\mathrm{d} \bP_{N,M,\omega}^{\beta,h}}{\mathrm{d} \bP} (\tau) := \frac{1}{Z_{N,M, \omega}^{\beta,h}} \exp \Big( \sum_{n=1}^N \sum_{m=1}^M (\beta \omega_{n,m} +h) \mathbf{1}_{(n,m) \in \tau} \Big) \mathbf{1}_{(N,M)\in \tau} \,,
\end{equation}
where $Z_{N,M,\omega}^{\beta,h}$ is the constrained partition function (the normalization constant)
\begin{equation}
\label{eq:pf}
Z_{N,M,\omega}^{\beta,h} := \bE \Big[ \exp \Big( \sum_{n=1}^N \sum_{m=1}^M (\beta \omega_{n,m} +h) \mathbf{1}_{(n,m) \in \tau} \Big) \mathbf{1}_{(N,M)\in \tau} \Big] \,.
\end{equation}
This corresponds to giving a reward $\gb \go_{n,m}+h$ (or a penalty if it is negative) if the $n^{\rm th}$ monomer of the first strand and the $m^{\rm th}$ monomer of the second strand meet.
Note that the presence of $\mathbf{1}_{(N,M)\in \tau}$ in the right-hand side means that we are considering trajectories that are pinned at the endpoint of the system (at a technical level it is more practical to work with the system pinned at the endpoint, see the proof of Theorem~\ref{th:fq}).

We also define the \emph{free} partition function, where the endpoints are free
\begin{equation}
\label{eq:free}
Z^{f,\beta,h}_{N,M,\omega}= \bE \Big[ \exp \Big( \sum_{n=1}^N \sum_{m=1}^M (\beta \omega_{n,m} +h) \mathbf{1}_{(n,m) \in \tau} \Big)  \Big] \,,
\end{equation}
that can be compared to the \emph{constrained} partition function \eqref{eq:pf}, see Lemma \ref{th:zfc}.
For notational convenience, we will sometimes suppress the $\beta,h$ from the partition function.

\smallskip

One then defines the {\sl quenched}  free energy of the system. We prove the following theorem in Section \ref{sec:existence}.

\begin{theorem}
\label{th:fq}
For all $\gamma>0$, $h \in \bbR$, $\beta \ge 0$ and every choice of $\{M(N)\}_{N=1,2, \ldots}$ such that $\lim_{N\to\infty}  M(N)/N=\gamma$
we have
\begin{equation}
\label{eq:saf}
\lim_{N\to\infty} \frac 1N \log Z^{\gb,h}_{N,M(N),\omega} \, =\,  \lim_{N\to\infty} \frac 1N \bbE \log Z^{\gb,h}_{N,M(N),\omega}\, =:\, \tf_{\gamma}(\beta,h) \,,
\end{equation}
where the first limit
exists $\bbP(\dd \go)$-almost surely and in $L^1(\bbP)$. The same result holds for the free model, that is $\tf_{\gamma}(\beta,h) =  \lim_{N\to\infty}  \frac 1N \log Z^{f,\gb,h}_{N,M(N),\omega}$ 
  $\bbP(\dd \go)$-a.s.\  and in $L^1(\bbP)$.
  
The function $(\gb,h) \mapsto \tf_{\gamma}(\gb,h)$ is convex, $h\mapsto \tf_{\gamma}(\gb,h)$
 and  $\gb\mapsto \tf_{\gamma}(\gb,h)$ are nondecreasing, and $\gamma \mapsto \tf_{\gamma}(\gb,h)$ is nondecreasing and 
 continuous.

\end{theorem}

The homogeneous model corresponds to the case $\beta = 0$: let us drop the $\gb$ and $\go$ dependence in the partition function that will be simply denoted $Z_{N, M}^h$. 
%We will use for the partition function in absence of disorder the following notation
%\begin{equation}
%Z_{N,M}(x) = \bE \left[ \exp \left( x \sum_{n=1}^N \sum_{m=1}^M  \mathbf{1}_{(n,m) \in \tau} \right) \mathbf{1}_{(N,M)\in \tau} \right]\,.
%\end{equation}
 The homogeneous model is {\sl exactly solvable}  and sharp estimates of $\tf_\gamma (0,h)$ near criticality are given in~\cite{cf:GK}.  
%We present some facts we will need in the proofs:

\begin{theorem}[\cite{cf:GK}]
\label{th:beta0}
For every $\gamma \ge 1$, when $\beta=0$ we have $h_c(0) := \sup\{ h\, :\, \tf_{\gamma}(0,h)=0\}= 0$. Moreover there is a slowly varying function $L_{\ga, \gamma}(\cdot)$ such that as $h \searrow 0$ one has
\begin{equation}
\tf_\gamma (0,h) \sim  L_{\ga,\gamma} ( h) \, h^{1/ \min (1, \ga)} \,.
\end{equation}
Moreover, if $\sum_n n^2 K(n) < \infty$ (\textit{i.e.}\ $\mu<\infty$), then ${L_{\ga,\gamma}(h)}^{-1} = c^{-1} := \frac{1}{2} \sum_{n } n(n-1) K(n)$. 
\end{theorem}

\noindent
In fact, the disordered  system also presents this transition: we define the critical point 
\begin{equation}
h_c(\beta) := \sup \lbrace h \, : \,  \tf_\gamma (\beta,h) = 0 \rbrace = \min \lbrace h \, : \,  \tf_\gamma (\beta,h) >0 \rbrace \,.
\end{equation}
Let us note that $h_c(\gb)$ does not depend on $\gamma >0$, thanks to \eqref{eq:existence+.3} below.

On the other hand, we define the annealed free energy as 
\begin{equation}
\tf_\gamma^{a}(\beta,h) := \limtwo{N\to\infty,}{M(N) / N \to \gamma} \frac 1N \log \bbE Z_{N,M(N),\omega}^{\gb,h} = \tf_\gamma (0, h + \log Q(\beta)) \,.
\end{equation}
This link with the homogeneous model and the fact that $h_c(0)=0$ allow immediately to identify  the annealed critical point:
\begin{equation}
\label{eq:hann}
h_c^{a}(\beta) := \min \lbrace h \, : \,  \tf^{a}_\gamma (\beta,h) >0 \rbrace = - \log Q(\beta) \,.
\end{equation}
Now observe that by Jensen's inequality, we have that $\bbE \log Z_{N,M,\omega} \le \log \bbE Z_{N,M,\omega}$ and hence $\tf_\gamma^{q}(\beta,h) \le \tf_\gamma^{a}(\beta,h)$. Moreover, since $\beta \mapsto \tf_\gamma (\beta,h)$ is non-decreasing,  
we have that $\tf_\gamma (0,h) \le \tf^{q}_\gamma (\beta ,h) $. Therefore for every $\beta$ we have
\begin{equation}
\label{eq:inh}
h_c^{a}(\beta)\,  \le\,  h_c(\beta) \le h_c(0) \,.
\end{equation} 
One can show, by adapting the argument of proof of
\cite[Th.~5.2]{cf:G}, that the second inequality is strict for every $\gb \neq 0$.
The first inequality may or may not be strict and this is an important issue which is directly linked to disorder relevance and irrelevance.

\smallskip
Harris' criterion predicts that disorder is irrelevant if $\nu>2/d$. Here, Theorem~\ref{th:beta0} suggests that $\nu=1/\min(1,\alpha)$, if we admit that the correlation length of the non-disordered system can be given by the reciprocal of the free energy, as it is the case for the PS model, see \cite{cf:G08}. Since the model is $2$-dimensional (contrary to the PS model which is $1$-dimensional), it would mean that disorder is irrelevant when $\nu>1$, that is when $\ga<1$.

And in fact our first result states that the first inequality in \eqref{eq:inh} is an equality if $\ga<1$
and $\gb$ is not too large. For the same values of $\gb$ we can also show that the critical behavior is the same as for the
$\gb=0$ case (disorder irrelevance).
Our second result asserts that the inequality is strict for $\ga>1$. We interpret this critical point shift, with a certain abuse, as disorder relevance. We however refer to the discussion in Section~\ref{sec:misc} (in particular Conjecture~\ref{conj:smoothing}) regarding the change in the critical behavior.
We therefore prove that disorder is irrelevant if $\ga<1$, and relevant (in terms of critical points) if $\ga>1$,  confirming Harris' prediction.

\subsection{Relevance and irrelevance of disorder}

Let us define $\sigma :=\tau \cap \tau'$, where~$\tau$ and~$\tau'$ are two independent copies of~$\tau$: $\sigma$ is another bivariate renewal process, and Proposition~\ref{prop:transience} tells that $\sigma$ is terminating if $\ga\in(0,1)$ and persistent if $\ga>1$ -- the case $\ga=1$ is discussed in Remark~\ref{rem:alpha=1}.
% We start by showing disorder irrelevance for $\ga<1$.

\begin{theorem}
\label{thm:irrel}
Assume that $\sigma$ is terminating ( this includes $\ga<1$ and excludes $\ga >1$). Then there exists  $\beta_1>0$ (see \eqref{eq:beta1}), such that  for every $\beta \in (0,\beta_1)$ we have $h_c(\beta) = h_c^a(\beta)$, and moreover
\begin{equation}
\label{eq:s2}
\lim_{h \searrow h_c(\beta)} \frac{\log \tf_{\gamma}(\beta,h)}{\log (h-h_c(\beta))} = \frac{1}{\ga} \,.
\end{equation}
\end{theorem}
Hence, the order of the phase transition is unchanged when $\sigma$ is terminating (which is the case if $\ga<1$), at least when $\gb$ is small enough. We prove Theorem~\ref{thm:irrel} in Section~\ref{sec:irrel}.
We mention that when  the disorder distribution is infinitely divisible (for instance Gaussian), one can get  sharper bounds regarding the critical behavior of $\tf_{\gamma}(\beta,h)$, via a replica-coupling method, as done in \cite{cf:Ton08} or \cite{cf:Wa12}. For a statement and a detailed proof, we refer to \cite{cf:Khatib}.

%When the disorder distribution is Gaussian, we are able to obtain a sharper result regarding the critical behavior of $\tf_{\gamma}(\beta,h)$, via a replica-coupling method detailed in Appendix \ref{app:replica}.
%\begin{theorem}
%\label{thm:replica}
%Assume that $\go_1\sim \mathcal{N}(0,1)$ and that $\sigma$ is terminating.
%For every $\epsilon >0$, there exists $\beta_0(\epsilon)>0$ and $\Delta_0(\epsilon) >0$ such that, for every $\beta \le \beta_0(\epsilon)$ and $0 < \Delta < \Delta_0(\epsilon)$, 
%\begin{equation}
%(1- \epsilon) \tf_\gamma(0, \Delta) \le \tf_\gamma(\beta, h_c^a(\beta)+ \Delta) \le \tf_\gamma(0, \Delta) \,.
%\end{equation} 
%\end{theorem}
%

On the other hand, when $\ga>1$, we show that the quenched and annealed critical points differ, and we give a lower bound on the critical point shift.

\medskip

\begin{theorem}
\label{thm:rel}
For $\ga>1$ we have  $h_c(\gb) > h_c^{a}(\gb)$ for every $\beta>0$.
Moreover, for every $\gep>0$, there exists $\gb_{\gep}>0$ such that for any $\gb\leq \gb_\gep$ we have
\begin{equation}
\label{eq:rel}
h_c(\beta) - h_c^{a}(\beta)\ge \gD_{\gb}^{\gep}:=
\begin{cases}
\gb^{\frac{2\ga}{\ga-1} +\gep} & \quad \text{ if } \ga\in(1,2]\, , \\
\gb^4 |\log \gb |^{-6}  & \quad \text{ if } \ga >2\, .
\end{cases}
\end{equation}
Moreover, there is a slowly varying function  $\tilde{L}(\cdot)$ such that
\begin{equation}
\label{eq:upb}
h_c(\beta) - h_c^a (\beta) \le 
\tilde{L}(1/\gb)   \gb^{\frac{2\ga}{\ga-1} \vee 4} \, .
\end{equation}
\end{theorem}

\medskip

We add that $\gb\mapsto h_c(\gb)-h_c^{a}(\gb)$ is a non decreasing function of $\beta$: this result can be proven by the exact same procedure as the one used to prove Proposition 6.1 in \cite{cf:GLT11}. It is to be interpreted that disorder relevance is \emph{non-decreasing} in $\beta$.

\subsection{On the results, perspectives and related work}
\label{sec:misc}

\subsubsection{On the  main theorems}
A two replica computation plays a central role in the proof of Theorem~\ref{thm:irrel} and in  the proof of \eqref{eq:upb} of Theorem~\ref{thm:rel}: the intersection renewal $\sigma$ therefore  emerges naturally, like in the PS model. In the PS context, we now know that
disorder is irrelevant (for small values of $\gb$) if and only if the intersection renewal is terminating \cite{cf:BL16}. 
For the gPS model our results go in the same direction, but it is not sharp in the marginal case $\ga=1$: we only show  disorder irrelevance when the intersection renewal $\sigma$ is terminating. We refer to Remark \ref{rem:alpha=1} for further discussion on the case $\ga=1$, where more technicalities arise.
%{\color{blue} The difficulty is already at the renewal process level: more technicalities arise, and
%a necessary and sufficient condition for persistence  of $\sigma$ when  $\ga=1$ is more subtle.
%%-- it depends therefore on  the slowly varying function $L(\cdot)$ -- remains to be found, and seems out of reach without further understanding of bivariate renewal processes, 
%We refer to Remark \ref{rem:alpha=1} for further discussion. 
%}

The proof of \eqref{eq:rel} is based on coarse graining techniques and fractional moment method: we have chosen to adapt  the method proposed in 
\cite{cf:DGLT09} and the difficulties in its generalization come from dealing with the richness of a multidimensional path with respect to the one dimensional structure of the PS models.  A keyword for these difficulties is {\sl off-diagonal estimates}. It can certainly be improved in the direction of getting rid of the~$\gep$ in the exponent 
for $\ga \in (1, 2]$ and of the logarithmic term in the case $\ga >1$ by using more sophisticated coarse graining techniques 
(see \cite[Ch.~6]{cf:G} and references therein).
One could probably aim also for sharp estimates, like in \cite{cf:BL16}, but the estimates are technically rather demanding already to obtain \eqref{eq:rel}. We have chosen to stick to these simplified non-optimal (but almost optimal) bounds because sharper results would have required a substantially heavier argument of proof. 
%Such a sharp treatment would  make sense if we were to treat also the (marginal) case  $\ga=1$.
%, in which
%as we have explained, we still have some more basic difficulties
%technicalities arise already at the level of the intersection of bivariate renewals.
The techniques developed in \cite{cf:BL16,cf:BL16bis} should transfer to this model: at the expense of a high level of technicality, we expect that, in analogy with the PS model,  the necessary and sufficient condition for a critical point shift is the persistence of the intersection renewal $\sigma=\tau\cap\tau'$.
%: a more technically involved treatment should enable us to treat the case $\ga=1$.
%, and a more involved treatment should treat the $\ga=1$ case.

\subsubsection{Discussion on the presence of a \emph{smoothing} phenomenon}
Of course, a fully satisfactory result on disorder relevance would include showing that the critical exponent is modified by the disorder.
We do not have such a result, but let us make one observation and formulate a conjecture.

The observation is that Theorem~\ref{thm:irrel}
may appear at first surprising in view of the smoothing inequality \cite{cf:GT06,cf:CH13} for PS models that ensures 
that the free energy exponent cannot be smaller than $2$ in presence of disorder: for the gPS model the free energy exponent can go down to $1$, since in \eqref{eq:s2} we can choose $\ga$ arbitrarily close to $1$. 
The reason of the difference is that the PS model is $1$-dimensional whereas the gPS model is $2$-dimensional: Harris criterion tells that disorder should be irrelevant if $\nu>2$ for the PS model, and $\nu>1$ for the gPS model. In the gPS model, the irrelevant disorder regime therefore holds even if $\nu $ ($=\min(1,\ga)^{-1}$) is arbitrarily close to $1$: hence one should not hope for a general smoothing inequality valid whatever $\alpha$ is.

It is however worthwhile attempting to sketch the argument in \cite{cf:GT06}, in the simplified set-up of Gaussian charges \cite[Ch.~5, Sec.~4]{cf:GB}.
This is useful both to understand were the argument fails and because we can realize that a suitable generalization of the argument naturally leads to a conjecture that we state just below.

\begin{SCfigure}[50][htbp]
\centering
\caption{\label{fig:smoothing} 
\footnotesize
Schematic view of the coarse graining procedure proposed for a smoothing inequality. The environment 
is divided in blocks of size $\ell$, called $\ell$-boxes. A $\ell$-box is good (shadowed in the figure) if
the partition function in this block grows at 
an exponential rate that is larger than the free energy of the system. The good $\ell$-boxes will be rare, but we can choose 
$n$ such that in a system of linear size $n\ell$, with positive probability, there will be at least one good 
$\ell$-box. A lower bound on the partition function follows  by the limitation to trajectories that visit
only a given  good $\ell$-box (say, the closest).
}
\includegraphics[scale=0.34]{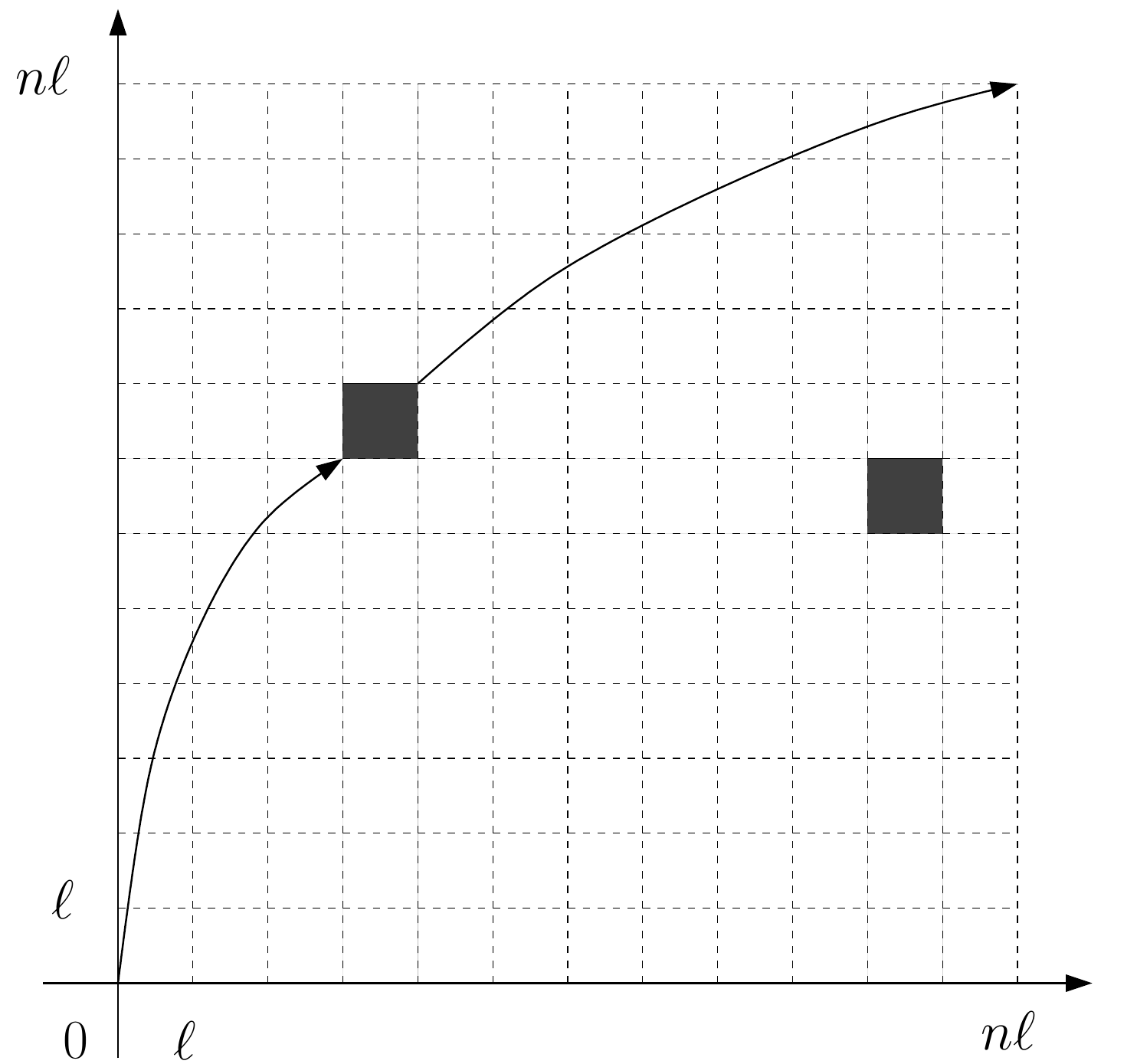}
\end{SCfigure}

The argument \cite{cf:GT06} is based on introducing a coarse graining scale $\ell \in \bbN$ and considering the environment in terms of $\ell$-boxes, see Figure~\ref{fig:smoothing}. We argue for the case $\gamma=1$ ($M=N$) and we consider the system at criticality, that is $h=h_c(\gb)$:

\begin{enumerate}
\item A good $\ell$-box is a box for which the pinned partition function (\textit{i.e.}\ pinned at the south-west and north-east corners of the $\ell$-box) is larger than 
$\exp\big( \tfrac12 \, \ell \, \tf_1(\gb, h+ \gd) \big)$, with $\gd>0$. For $\ell\to \infty$ this is a rare event. The probability of such an event 
can be estimated from below by shifting the environment of $\gd/\gb$, that is $\go_{i,j}$ is replaced with
$\go_{i,j}+ \gd/\gb$, and by performing
a relative entropy estimate \cite[Ch.~5]{cf:G}. This shows that the probability of such a rare event is at least
$\exp(-\gd^2 \ell^2 /(2 \gb^2))$:   note the $\ell^2$ term, with respect to $\ell$ in the PS case \cite{cf:GT06}. 
\item We then make a lower bound on the partition function of the system by discarding renewal trajectories that visit
$\ell$-boxes that are not good, and keeping only trajectories that enter good $\ell$-boxes through the south-west corner
and exit through the north-east corner. 
\end{enumerate}

The trajectories are therefore alternated jumps to a good box, {\sl visit} of the box, and then a new jump to another good box.
Jumps are long because good boxes are rare.
The analysis in  \cite{cf:GT06} is ultimately reduced to see what happens in one {\sl jump and visit}: by exploiting 
super-additivity one can even just choose $N=n\ell$ such that there is (say, with probability  at least $1/2$), 
at least one good box in the system (like it is done in \cite{cf:BL11}). We therefore see that we need  $n^2 \exp(-\gd^2 \ell^2 /(2 \gb^2))\approx 1$,
so that $n \approx   \exp(-\gd^2 \ell^2 / \gb^2)$: with this level of precision, jumping to enter such a box costs $K(n\ell) = (n\ell)^{-(2+\ga)}$ (let us consider the case in which $L(\cdot)$ is  a constant, but the computation goes through in the same way also in the general case). In the box there will be a contribution $\exp\big( \tfrac12\, \ell\, \tf_1(\gb, h+ \gd) \big)$. The net contribution to the logarithm of the partition function, divided by the size $n\ell$ of the system, is then
\begin{equation}
\frac 1{n \ell } \Big(  \log K\left(\exp(-\gd^2 \ell^2 / \gb^2) \ell \right) +  \frac\ell 2 (\tf_1(\gb, h+ \gd)) \Big) \, \ge \, 
\frac 1n \Big(
-\frac{c}{\gb^2} \gd ^2 \ell+ \frac 12 \tf_1(\gb, h+ \gd) \Big)\, ,
\end{equation}
 with $c$ a positive constant that we have left implicit (it depends on more accurate computations, and can be in principle just reduced to $2+ \ga$).
 
 Now let us choose $h=h_c(\gb)$. So the argument we just outlined goes in the direction  of saying that
 \begin{equation}
 0\, =\, \tf_1(\gb, h_c(\gb))\, \ge \,  \frac 1n \Big(
-\frac{c}{\gb^2} \gd ^2 \ell+ \frac 12 \tf_1(\gb, h_c(\gb)+ \gd) \Big)\, ,
 \end{equation}
so that 
\begin{equation}
\label{eq:tf1sm}
\tf_1(\gb, h_c(\gb)+ \gd) \, \le \, \frac{2c}{\gb^2} \gd ^2 \ell\, .
 \end{equation}
At this stage choosing $\ell$ arbitrarily large is of no help. The steps we have performed up to now require that $ \ell \gd$ is large (so that the good boxes we have chosen are really sparse). On the other hand we  need to have chosen the size of the boxes so that
$Z_{\ell, \ell, \go}^{\gb, h} \ge \exp( \ell (\tf_1(\gb, h_c(\gb)+ \gd))/2)$. This is a delicate issue, 
but it definitely appears that for this to hold, $\ell \tf_1(\gb, h_c(\gb)+ \gd)$ needs to be sufficiently large (say, larger than  a 
suitable constant): see for example the discussion on the notion of correlation length given in \cite[Ch.~2]{cf:G} and references therein, notably \cite{cf:GTalea}, where the correlation length is identified by the reciprocal of the free energy.
But if $\ell$ is (a constant times) $1/\tf_1(\gb, h_c(\gb)+ \gd)$ then from
\eqref{eq:tf1sm} we obtain
\begin{equation}
\label{eq:trivb}
\tf_1(\gb, h_c(\gb)+ \gd)\, \le \, C \gd\, ,
\end{equation}
for some $C>0$.  But such a bound is trivial: it holds with $C=1$ just because the contact density cannot exceed one! On the other hand, as we have already pointed out, we could not have hoped for a better bound valid for any $\alpha>0$. 

\smallskip
In spite of the fact that it leads to a trivial result, we insist 
that the argument we have just outlined can be made rigorous: the delicate step is the last one, where one has to use arguments 
developed in \cite{cf:GTalea}. 
It can therefore be taken as a starting point to push things further. Indeed, it appears useless to modify the environment in the whole $\ell$-box, at least if $\ga>1$.
In fact if $\ga>1$ one can show that for $q>1/ \max(\ga,2)$
\begin{equation}
\label{eq:evin}
\lim_{N \to \infty}\bP \left( \tau \cap [0, N]^2 \subset \{(i,j) \in  \bbZ^2:\, \vert i-j \vert \le  N^q \}\right) \, =\, 1\, .
\end{equation}
We can then consider modifying only the environment that is close to the diagonal, that is in a subset of the $\ell$-box with $|i-j|\leq \ell^q$. This would  improve the lower bound on the probability of a good $\ell$-box to $\exp( - c'\, \gd^2 \ell^{q+1})$, and \eqref{eq:tf1sm} would become
\[ \tf_1(\gb,h_c(\gb)+\gd) \leq \frac{c''}{\gb^2} \gd^2 \ell^{q}.\]
Taking $\ell$ a constant times $1/\tf_1(\gb,h_c(\gb)+\gd)$ as in the argument leading to \eqref{eq:trivb}, and then taking $q$ arbitrarily close to $1/\min(\alpha,2)$ supports the following:
\begin{conjecture}
For every $\ga>0$ and every $\gb >0$
\label{conj:smoothing}
\begin{equation}
\limsup_{\gd \searrow 0} \
\frac{\log \tf_1 (\gb, h_c(\gb)+\gd)
}{
\log  \gd  
} \, \ge 
\begin{cases}
\frac{2\ga}{\ga+1} & \text{ for } \ga \in (1,2)\, ,
\\
\frac 43 & \text{ for } \ga \ge 2\, .
\end{cases}
\end{equation}
\end{conjecture}
We stress that a natural concern arises from performing the change of measure  only in a subset of the environment, close to the diagonal.  One indeed needs to be sure that the trajectories contributing to (a fraction of) $ \frac{1}{\ell} \log Z_{\ell,\ell,\go} \approx \tf_1(\gb,h_c(\gb)+\gd)$ can be constrained to stay in the region $\{(i,j) \in  \bbZ^2:\, \vert i-j \vert \le  \ell^q \}$: if it is the case, one can ``force'' trajectories to visit sites where the environment  has indeed been shifted.

\subsubsection{An important  modeling issue: the choice of the disorder}
There is no doubt that the first disorder that comes to mind when thinking of DNA modeling is not the one we have used. One would rather choose 
$\go_{i,j}=f(\go_i, \go_j)$ for a suitable choice of a function $f$ and a sequence $\{\go_j\}_{j=1,2, \ldots}$ of random variables (let us say IID for simplicity, but if we want to stick to DNA problems very closely it appears that some sort of strongly correlated sequence may be more appropriate \cite{cf:Peng}). For example, we could choose $\go_j$ taking only two values $e_{AT}$ and $e_{GC}$ and then make a choice for $f$ that reflects the fact that AT bounds are weaker than GC bounds, and that all other possible bounds are even weaker. 
Even restricting to $\{\go_j\}_{j=1,2, \ldots}$ that is IID, this model is highly non trivial (gPS model with this type of disorder  has been considered at a numerical level in \cite{cf:GOO,cf:GO04},  see also
\cite{cf:EON11,cf:TN08} for related work). But one could also choose to consider the binding of two sequences that are not complementary (the case considered in \cite{cf:NG06} goes in this direction, even if only heuristics and numerics are presented): choose for example two independent sequences $\{\go^{(1)}_j\}_{j=1,2, \ldots}$ and
$\{\go^{(2)}_j\}_{j=1,2, \ldots}$ and use $\go_{i,j}=f(\go^{(1)}_i, \go^{(2)}_j)$. This is somewhat closer to what we are using (though it can be considered as a one-dimensional disorder), but it is still 
very difficult to deal with. The problem is in any case due to correlations in the disorder field $\go_{i,j}$, which can be dealt with  in some cases, see e.g.~\cite{cf:BL12,cf:BP15} or \cite{cf:AB18, cf:CCP17}. Our choice is in a sense a toy choice, but we stress that it is conceptually similar to the simplification  made for example in \cite{cf:BundHwa} in the RNA context. 
Moreover it recovers importance once we leave somewhat the DNA context and focus rather on moving toward understanding  
mathematically Harris' theory of disorder (ir)relevance---in particular for 2-dimensional systems, compared to the PS model, which is 1-dimensional.

We also point out  that this disordered version of the gPS model gives a bridge between
pinning model and directed polymers in random environment \cite{cf:Co07,cf:La10}, in particular, to the long range directed polymer \cite{cf:Co07,cf:We15}. Moreover a different class of two-strand polymer problems (the random walk pinning model) is treated in
\cite{cf:BT10,cf:BS1,cf:BS2}.

\subsubsection{Open questions and perspectives}
Several natural issues remain open:
 let us list some of them.

\begin{enumerate}
\item Prove a smoothing inequality,  thus showing disorder relevance in the original sense of Harris, for $\ga>1$  (see Conjecture~\ref{conj:smoothing}).
\item What is the effect of disorder on the other phase transitions? Here we have addressed only the denaturation transition, 
but in \cite{cf:GK} other transitions are shown to exist. Do they withstand the introduction of disorder? If so, does the corresponding critical behavior differ from the homogeneous case? This is the question (quickly) addressed  \cite{cf:NG06} where a rather bold conjecture is set forth.
\item We have dealt only with free energy estimates, but, like for the standard PS model, obtaining precise estimates on 
the gPS process (\textit{i.e.}\ establish properties of trajectories) is very challenging, see  \cite[Ch.8]{cf:G} and references therein.  The problem comes of course from the inhomogeneous nature of the disorder and the fact that on rare regions atypical disorder behaviors appear (this is ultimately also the problem we face at the free energy level, but it becomes particularly explicit when one analyses the trajectories).  A precise analysis of the trajectories of the non disordered gPS model can be found in \cite{cf:BGK}: this analysis is substantially more demanding  than the corresponding one for the PS model.
\item Dealing with the marginal case $\ga=1$ is open, mostly because of the additional technical difficulties (more complicated coarse-graining procedure, more technical estimates for bivariate renewals, etc.). This appears to be a problem at reach, but a very substantial amount of technical work is certainly needed.
\end{enumerate}

\subsubsection{Organization of the rest of the work}
The issues of existence and self-averaging of the free energy, \textit{i.e.}\ the proof of Theorem~\ref{th:fq}, are treated in Section~\ref{sec:existence}. In Section~\ref{sec:irrel} we prove Theorem~\ref{thm:irrel}, as well as the upper bound \eqref{eq:upb} of Theorem~\ref{thm:rel}.
The rest of the Theorem~\ref{thm:rel} is proven in Section~\ref{sec:fm}.
We collect in Appendix~\ref{app:bivrenewal} a number of statements and proofs about bivariate renewals. 
%Finally, in Appendix~\ref{app:replica}, Theorem~\ref{thm:replica} is proven.

\subsection{Some further notations}
\label{sec:defan}

We stress that $\tau$ is symmetric and in the domain of attraction of a $\min(\ga,2)$-stable distribution: we denote $(b_n)_{n\geq 1}$ be the recentering sequence and ${(a_n)}_{n \ge 0}$  the renormalizing sequence for $\tau_n$, that is such that $\tfrac{1}{a_n} (\tau_n - (b_n,b_n))$ converges to a $\min(\ga,2)$ stable distribution, whose density is denoted $g_\ga(\cdot,\cdot)$.
For $b_n$, we have $b_n=\mu n$ if $\ga>1$, $b_n= n\bE[\min(X_1, n)]$ if $\ga=1$, and $b_n=0$ if $\ga \in(0,1)$.
The asymptotic behavior of $a_n$ is characterized by
\begin{equation}
\label{def:an}
\begin{split}
L(a_n) (a_n)^{-\ga} \sim \,  1/n\, &\qquad \text{ if } \ga<2\\
\sigma(a_n) (a_n)^{-2} \sim \, 1/n \, &\qquad \text{ if } \ga\geq 2
\end{split}
\end{equation}
where $\sigma(n):=\bE[\min(X_1, n)^2]$. If $\ga = 2$ and $\bE[X^2]=+\infty$, then $\sigma(n)$ grows to infinity as a slowly varying function (and verifies $\sigma(n)/L(n) \to +\infty$), whereas if $\bE[X^2]<+\infty$ (in particular when $\ga>2$) $a_n$ is proportional to $\sqrt{n}$.

In any case, there exists some slowly varying function $\psi(\cdot)$ such that
\begin{equation}
a_n = \psi(n) n^{1/\min(\ga, 2)} \, .
\label{an}
\end{equation}
We provide some useful results on bivariate renewals in Appendix \ref{app:bivrenewal}, in particular on the renewal mass function $\bP((n,m)\in\tau)$.

\section{Free Energy: existence and properties}
\label{sec:existence}

In this section we often assume $\gamma\in \bbQ$: in this case we write it as $\gamma=p/q$ with $p$ and $q$ relatively prime positive integer numbers.

\begin{proposition}
\label{th:existence+}
For every $\gamma>0$ and every  $\{M(N)\}_{N=1,2, \ldots}$ such that 
$\lim_{N \to \infty} M(N)/N=\gamma$ we have that 
\begin{equation}
\label{eq:existence+.1}
\lim_{N \to \infty} \frac 1N \log Z_{N,M(N), \go}^{\gb,h} \,=\, \lim_{N \to \infty} 
\frac 1N \bbE \log Z_{N,M(N), \go}^{\gb,h}\, =:\, \tf_\gamma(\gb, h) \, ,
\end{equation}
where the first limit is meant
$\bbP(\dd \go)$-a.s.\ and in $L^1 (\bbP)$.
$\tf_\gamma (\cdot,\cdot)$ is convex and $\tf_\gamma (\gb,\cdot)$ is non-decreasing, and also $\tf_\gamma (\cdot,h)$
is non-decreasing on the positive semi-axis, non-increasing in the negative one.
Moreover if $\gamma=\frac pq\in \bbQ$
\begin{equation}
\label{eq:existence+.2}
\tf_\gamma (\gb, h) \, =\, \sup_{N:\, \frac N q \in \bbN} \frac 1N \bbE \log Z_{N, \gamma N, \go}^{\gb,h}\, .
\end{equation}
Finally we have the bound: for every $\gamma_2\ge \gamma_1 >0$
\begin{equation}
\label{eq:existence+.3}
\tf_{\gamma_1}(\gb, h) \, \le \, \tf_{\gamma_2}(\gb, h)\, \le \, \frac{\gamma_2}{\gamma_1} \tf_{\gamma_1}(\gb, h) \, ,
\end{equation}
which implies that $\gamma \mapsto \tf_{\gamma}(\gb, h)$ is locally Lipschitz (hence continuous).
\end{proposition}

\noindent
\emph{Proof.}
The proof   is divided into several steps:

\begin{enumerate}
\item We first show that for $\gamma \in \bbQ$, along a subsequence with $ N/q\in \bbN$,  $\log Z_{N, \gamma N, \go}$ 
 is super-additive in an ergodic sense, which implies the existence of the free energy limit \eqref{eq:existence+.1} along this subsequence.
 %and also \eqref{eq:existence+.1}. 
\item The restriction $\gamma N\in \bbN$ is then removed by a direct estimate, for what concerns the existence of the free energy limit, still with $\gamma \in \bbQ$.
\item We then prove a comparison estimate between $Z_{N, \gamma_1 N, \go}$ and $Z_{N, \gamma_2 N, \go}$
%recall the convention $Z_{N, \gamma N, \go} =Z_{N, \lfloor\gamma N\rfloor, \go}$ 
and use it to establish the existence of the free energy limit for $Z_{N, \gamma N, \go} $, every $\gamma>0$.
\item The same comparison estimate yields also  \eqref{eq:existence+.3}, and the fact that one can take the limit along 
an arbitrary sequence satisfying $M (N) \sim \gamma N$, for $N \to \infty$.
\item Finally, we prove  the convexity and monotonicity statements.
\end{enumerate}
\smallskip

\noindent
\emph{Step 1.} With $\gamma=p/q$ set 
$
\cZ_{j}(\go) := Z_{jq, jp, \go}$. Then one directly sees that
\begin{equation}
\label{eq:super-add-erg}
\cZ_{j_1+j_2}(\go)\,\ge \, \cZ_{j_1}(\go) \cZ_{j_2}\left(\Theta_{j_1 q,j_1 p} \go\right)\, ,
\end{equation}
where $(\Theta_{q,p} \go)_{n,m}= \go_{q+n,p+m}$. Since $\go$ is an IID sequence of $L^1$ random variables, it is straightforward to see that 
$\log \cZ_{j} \in L^1(\bbP)$. Also,   $ \vert \log \cZ_{j} \vert \leq h n+ \gb  \sup_{\gamma \in \Gamma } \sum_{(n,m)\in\gamma} |\go_{n,m}|$, where $\Gamma$ is the set of nearest-neighbors up-right paths: a Last Passage Percolation observation then tells us that a sufficient condition for having $\sup_n \frac 1n \bbE \vert \log \cZ_{j} \vert   <+\infty$ is $\bbE[\go_{1,1}^2]<+\infty$, see \cite{cf:Martin02}.
Hence we see that $\{ -\log \cZ_{j}(\go)\}_{j=1,2 , \ldots}$ satisfies the hypotheses of Kingman sub-additive ergodic theorem
(see for example \cite[Sec.~A.7]{cf:GB}), and we get that $\{ \frac 1j\log \cZ_{j}(\go)\}_{j=1,2 , \ldots}$ converges $\bbP( \dd \go)$-a.s. and in $L^1(\bbP)$.
Moreover \eqref{eq:super-add-erg} directly tells us that $\{ \bbE \log \cZ_{j}(\go)\}_{j=1,2 , \ldots}$ is super-additive,
so that $\lim _{j\to \infty} \frac 1j  \bbE \log \cZ_{j}(\go)= \sup_{j\in \bbN} \frac 1j  \bbE \log \cZ_{j}(\go)$. 
This establishes \eqref{eq:existence+.2}, and also
\eqref{eq:existence+.1}, but only for $M(N)= \gamma N$ with $\gamma \in \bbQ$ and along the subsequence satisfying 
$\gamma N \in \bbN$.

\smallskip

\noindent
\emph{Step 2.}  Still with $\gamma =p/q$, the restriction to $\gamma N \in \bbN$ can be removed by observing that we can 
write $N= jq+r$, with $r\in \{0, 1, \ldots , q-1\}$, and for $r\neq 0$ 
\begin{equation}
\label{eq:exist-step2.1}
\begin{split}
Z_{N, \lfloor \gamma N \rfloor , \go} \, &\ge\, Z_{jq, jp, \go} \exp(\gb \go_{N, \lfloor \gamma N \rfloor}+h) K \Big(r+ \Big \lfloor jp+ \frac pq r \Big \rfloor
-jp \Big) 
\\ & \ge\, c(p,q) \exp(\gb \go_{N, \lfloor \gamma N \rfloor}+h) Z_{jq, jp, \go} \, ,
\end{split}
\end{equation}
where $c(p,q)>0$.

In the same way
\begin{equation}
\label{eq:exist-step2.2}
\begin{split}
Z_{(j+1)q, (j+1)p, \go}\, &\ge \,  
Z_{N, \lfloor \gamma N \rfloor, \go} \exp(\gb \go_{(j+1)q, (j+1)p}+h)
K \Big(q-r+ (j+1)p- \Big \lfloor jp+ \frac pq r \Big \rfloor \Big)
\\
&\ge \, c(p,q) \exp(\gb \go_{(j+1)q, (j+1)p}+h)
Z_{N, \lfloor \gamma N \rfloor, \go}\, , 
\end{split}
\end{equation}
possibly redefining $c(p,q)>0$.
From \eqref{eq:exist-step2.1} and \eqref{eq:exist-step2.2} one easily removes the 
restriction to $\gamma N \in \bbN$ and establishes \eqref{eq:existence+.1} for $M(N) = \lfloor \gamma N \rfloor$ with 
$\gamma\in \bbQ$.

\smallskip

\noindent
\emph{Step 3.} We now establish \eqref{eq:existence+.1} for $M(N)= \lfloor \gamma N\rfloor$ for an arbitrary $\gamma>0$,
by proving the announced comparison bounds, upper and lower.

The upper bound is more general: if $M_2> M_1$  and if 
there exists $c>0$ such that $M_2 \le c N$
we see 
that
\begin{equation}
\label{eq:comparison-upper}
\begin{split}
Z_{N, M_1, \go}\, & = \, \sum_{n=0}^{N-1} \sum_{m=0}^{M_1-1} Z_{n, m, \go}
K(N-n+M_1-m) \exp\left(\gb \go_{N, M_1} +h\right)
\\
&\le \,  c_K N^{c_K} \exp\left(\gb (\go_{N, M_1}-\go_{N, M_2}) \right)
\\ & \quad \quad
\sum_{n=0}^{N-1} \sum_{m=0}^{M_1-1}  Z_{n, m, \go}
K(N-n+M_2-m) \exp\left(\gb \go_{N, M_2} +h\right)
\\ &\le\, c_K N^{c_K} \exp\left(\gb (\go_{N, M_1}-\go_{N, M_2}) \right) Z_{N, M_2, \go}\, ,
\end{split}
\end{equation}
where in the first inequality we have used that $K(\cdot)$ is regularly varying and that $M_2 \le cN$   to see that there exists $c_K>0$ such that
\begin{equation}
\frac{
K(N-n+M_1-m)}{
K(N-n+M_2-m)}\, \le \, c_K N^{c_K}\, ,
\end{equation}
for every $N$. For the second inequality we have relaxed the constrained $m< M_1$ to $m<M_2$. 

On the other hand, we prove a comparison  lower bound only for $M$ of the form $\lfloor \gamma N \rfloor$.
Let us choose $\gamma_2>\gamma_1>0$.
Note that for 
\begin{equation}
N'\, :=\,  \Big \lfloor \frac{\gamma_1}{\gamma_2}N \Big \rfloor - \Big \lceil \frac 2 {\gamma_2} \Big \rceil\, ,
\end{equation}
we have $\lfloor \gamma_2 N' \rfloor +1 \le  \lfloor \gamma_1 N \rfloor $ so that
\begin{equation}
\label{eq:comparison-lower}
\begin{split}
Z_{N,  \lfloor \gamma_1 N \rfloor , \go}\, &\ge \, K\left(N-N'+ \lfloor \gamma_1 N \rfloor -\lfloor \gamma_2 N' \rfloor\right)
\exp\left(\gb\go_{N,  \lfloor \gamma_1 N \rfloor }+h\right)
Z_{N',  \lfloor \gamma_2 N' \rfloor , \go}
\\
&\ge \,  \left(c_K N^{c_K}\right)^{-1} \exp\left(\gb\go_{N,  \lfloor \gamma_1 N \rfloor }+h\right)
Z_{N',  \lfloor \gamma_2 N' \rfloor , \go}\, ,
\end{split}
\end{equation}
possibly changing the value of $c_K>0$.

We now choose $0<\gamma_1<\gamma_2 \in  \bbQ$. Then
\eqref{eq:comparison-upper} implies that $\bbP(\dd \go)$-a.s.
\begin{equation}
\label{eq:fclb1}
\limsup_{N \to \infty} \frac 1N \log Z_{N,  \lfloor \gamma_1 N \rfloor , \go} \, \le \, \tf_{\gamma_2} (\gb, h)\, ,
\end{equation}
and \eqref{eq:comparison-lower} implies that $\bbP(\dd \go)$-a.s.
\begin{equation}
\label{eq:fclb2}
\liminf_{N \to \infty} \frac 1N \log Z_{N,  \lfloor \gamma_1 N \rfloor , \go} \, \ge \, 
\lim_{N \to \infty} \frac 1{N'}\log  Z_{N',  \lfloor \gamma_2 N' \rfloor , \go}\, =\, \frac{\gamma_1}{\gamma_2}
\tf_{\gamma_2}(\gb, h)
\, ,
\end{equation}
and the proof of \eqref{eq:existence+.1} is achieved in the $\bbP(\dd \go)$-a.s.\ sense
for $M(N)= \lfloor\gamma_1 N \rfloor$, by choosing a sequence of values for $\gamma_2$
converging to $\gamma_1$, defining thus $F_{\gamma_1}(\gb, h)$ also by this limit procedure.
Note that a byproduct is that \eqref{eq:existence+.3} holds, hence
%of this procedure and of the bounds we obtained is that $\gamma \mapsto 
$\gamma\mapsto \tf_\gamma(\gb, h)$ is non decreasing and (locally) Lipschitz continuous.
To upgrade \eqref{eq:existence+.1} to the $L^1(\bbP)$ sense one simply applies the expectation $\bbE[ \, \cdot\,  ]$ to the $\log$ of  \eqref{eq:comparison-lower} and \eqref{eq:comparison-upper}  so that one obtains 
$\lim_{N \to \infty} (1/N) \bbE\log Z_{N,  \lfloor \gamma N \rfloor , \go} = \tf_{\gamma}(\gb, h)$
for every $\gamma>0$, and  the first limit in  \eqref{eq:existence+.1} holds in the  $L^1(\bbP)$ sense by Scheff\'e's Lemma.

\smallskip

\noindent
\emph{Step 4.} The generalization to a sequence $M(N)\sim \gamma N$ is just made by observing that
given arbitrary $\gamma_1< \gamma_2$ with $\gamma \in (\gamma_1, \gamma_2)$
 for
$N_0$ sufficiently large we have $\lfloor \gamma_1 N \rfloor < M(N) < \lfloor \gamma_2 N \rfloor$ for every $N \ge N_0$.
At this point we can apply the comparison bounds like in the previous step and conclude by an approximation 
procedure.

\smallskip

\noindent
\emph{Step 5.}
The function  $(\beta,h) \mapsto \tf_\gamma(\beta,h)$ is  convex 
because  it is the limit of a sequence of convex functions. Monotonicity in $h$ for $\gb$ fixed is also evident from 
the finite $N$ expression. The fact that $\gb \mapsto \tf_\gamma(\beta,h)$ is non increasing for $\gb \le 0$ and 
non decreasing for $\gb \ge 0$ follows from convexity and the fact that $\partial _\gb \bbE \log Z_{N,M,\omega}^{\beta,h}=0$
(by direct computation, since the $\go$ variables are centered), so
$\partial_\gb \tf_\gamma(\beta,h) \vert_{\gb=0}=0$.
This completes the proof of Proposition~\ref{th:existence+}.
\qed

\smallskip

We  now compare  the constrained and the free partition function:
\begin{lemma}
\label{th:zfc}
For any $\ga_+ > \ga$,
there exists $C$ such that for every $N ,M \in \bbN$ and 
\begin{multline}
Z^{c}_{N,M,\omega} \le Z^{f}_{N,M,\omega} \le \\
 Z^{c}_{N,M,\omega}  \times \Big( 1 + C   (N+M)^{3+\ga_+} e^{-\beta \omega_{N,M}} \suptwo{1 \le n \le N}{1 \le m \le M } \lbrace e^{\beta \omega_{n,M} }, e^{ \beta \omega_{N,m} } \rbrace   \Big) \,.
\end{multline}
\end{lemma}

\begin{proof}
The lower bound is trivial: we have $Z_{N,M,\omega}^{f} \geq Z^{f}_{N,M,\omega} (	(N,M)\in\tau) =Z_{N,M,\omega}^{c}$, where we introduced the notation $Z^{f}_{N,M,\omega} (	A) = 
 \bE \big[ \exp \big( \sum_{n=1}^N \sum_{m=1}^M (\beta \omega_{n,m} +h) \mathbf{1}_{(n,m) \in \tau} \big) \ind_A \big]$.

On the other hand, for $N,M \ge 1$, we have
\begin{multline}
Z^{f}_{N,M,\omega} = \sum_{n=0}^N \sum_{m=0}^{M} Z^{f}_{N,M,\omega} \big( \tau \cap [n,N] \times [m,M] = \lbrace (n,m) \rbrace \big) \\
 \leq Z_{N,M,\omega}^c +\sum_{n=0}^{N-1} \sum_{m=0}^{M-1}  Z_{n,m,\omega}^{c} + \sum_{n=0}^{N-1} Z_{n,M,\omega}^{c} +\sum_{m=0}^{M-1}  Z_{N,m,\omega}^{c}  \, .
 \label{eq:zf4}
\end{multline}
%
%Then we write
%\begin{align}
%\label{eq:zf4}
%Z^{f}_{N,M,\omega} = Z^{c}_{N,M,\omega} + \sum_{n=0}^{N-1}  \sum_{m=0}^{M-1} Z^{c}_{n,m,\omega} \bP(\tau_1\cap(&0,N-n ]\times (0,M-m]=\emptyset ) \notag\\
%& + \sum_{n=0}^{N-1} Z^{c}_{n,M,\omega} + \sum_{m=0}^{M-1} Z^{c}_{N,m,\omega}   \, ,
%\end{align}
%from which we see that the lower bound holds.
%For the upper bound, 
Now, observe that for any $n\leq N-1$, $m\leq M-1$,
\begin{equation}
 Z^{c}_{n,m,\omega}  
\le C_1 (N+M)^{2+ \ga_+} e^{ -h - \beta \omega_{N,M} } Z^{c}_{n,m,\omega} K (M+N -n-m) e^{  h+\beta \omega_{N,M} }
\end{equation}
for any $\ga_+ > \ga$, so that
\begin{equation}
\sum_{n=0}^{N-1} \sum_{m=0}^{M-1} Z^{c}_{n,m,\omega} \leq  C_1 (N+M)^{2+ \ga_+}  Z^{c}_{N,M,\omega}e^{-h - \beta \omega_{N,M} } .
\end{equation}

For $n < N$ and $m=N$, there exists $C_2$ such that 
\begin{equation}
Z^{c}_{n,M,\omega} \le C_2 N^{2+\ga_+} \, Z^{c}_{N,M,\omega} \, \exp \left( \beta \omega_{n,M} - \beta \omega_{N,M} \right) \,,
\end{equation}
and we obtain
\begin{equation}
\sum_{n=0}^{N-1} Z^{c}_{n,M,\omega}  \le C_2 N ^{3 +\ga_+ }  \, \sup_{1 \le n \le N} \lbrace \exp( \beta \omega_{n,M} ) \rbrace \, e^{-\beta \omega_{N,M}} \, Z^{c}_{N,M,\omega} \,.
\end{equation}
The analogous holds for the last term in \eqref{eq:zf4},
%there exists 
%$C_3 >0$ 
%\begin{equation}
%\sum_{m=0}^{M-1} Z^{c}_{N,m,\omega} \le  C_3 M^{3 +\ga_+ }  \, \sup_{1 \le m \le M} \lbrace \exp(\beta \omega_{N,m} ) \rbrace \, e^{-\beta \omega_{N,M}} \, Z^{c}_{N,M,\omega} \,.
%\end{equation}
and the proof is therefore complete.
\end{proof}

From Lemma \ref{th:zfc}, and using also that $\lim_{N\to\infty} \frac1N \sup_{1\leq n\leq N} \go_{n,M} = 0$ $\bbP$-a.s.\ (note that $\bbP(|\go_1|>x) =o(1/x)$, since $\bbE[|\go_1|]<+\infty$), it follows that Theorem~\ref{th:fq} also holds for the free model, namely:
\begin{equation}
\label{eq:saf.2}
\tf_{\gamma}(\beta,h) \, =\,  \lim_{N\to\infty}  \frac 1N \log Z^{f,\gb,h}_{N,M(N),\omega} \qquad \bbP(\dd \go)\text{-a.s.  and in } L^1(\bbP)\, .
\end{equation}

\smallskip 

We now introduce some notation that is used later in the paper:
for positive integers  $a_1 < a_2$ and $b_1 < b_2$, we define the partition function of the system on $[a_1,a_2] \times [b_1,b_2]$ by
\begin{equation}
\label{eq:pp}
 Z_{(a_1,b_1),(a_2,b_2),\omega} :=   \bE \bigg[ \exp \bigg( \sum_{n=a_1+1}^{a_2} \sum_{m=b_1+1}^{b_2} (\beta \omega_{n,m} +h) \mathbf{1}_{(n,m) \in \tau} \bigg) \mathbf{1}_{(a_2,b_2)\in \tau} \,  \bigg| \, (a_1,b_1) \in \tau \bigg] \,,
\end{equation}
with the convention that $Z_{(a_1,b_1),(a_1,b_1),\omega} = 1$ and $ Z_{(a_1,b_1),(a_1,b_2),\omega}= Z_{(a_1,b_1),(a_2,b_1),\omega}=0$.

\section{Upper bound on the critical point shift}
\label{sec:irrel}

The arguments in this section follow the line of proof of H.~Lacoin in \cite{cf:Lac}, and is mainly based on a second moment computation. We start with some preliminary results.

\begin{proposition}
\label{th:unint1}
If ${ \lbrace {Z^{f,\gb,h_c^{a}(\gb)}_{N,M(N),\omega}} \rbrace}_N$ is \emph{uniformly integrable} ,
there exists $\zeta >0$ such that for every sequence of events $\{A_N\}_{N=1, 2, \ldots}$ satisfying $\lim_{N} \bP(A_N)=0$
there is $N_0\in \bbN $ such that 
\begin{equation}
\label{eq:unint1}
\inf_{N \ge N_0}
\bbP \left( \bP^{f,\gb,h_c^{a}(\gb)}_{N,M(N),\omega} (A_N)\le \frac 12
\text{ and } Z^{f,\gb,h_c^{a}(\gb)}_{N,M(N),\omega} > \frac 12  \right)  \, \ge\, \zeta\,.
\end{equation} 

\end{proposition}

%\begin{equation}
%\label{eq:AZ}
%\bbE \left[ \bP^{f,\gb,h_c^{a}(\gb)}_{N,M,\omega} (A_N) \mathbf{1}_{\{ Z^{f,\gb,h_c^{a}(\gb)}_{N,M,\omega} > 1/2 \}} \right] < \delta \,.
%\end{equation} 

\begin{proof}
We set $h=h_c^{a}(\gb)$. It is sufficient to prove that there exists $\zeta>0$ such that
\begin{equation}
\label{eq:infz}
\inf_N \bbP \Big( Z^{f,\gb,h_c^{a}(\gb)}_{N,M,\omega} > \frac 12 \Big) \, \ge \,  2\zeta \,.
\end{equation}
and   
\begin{equation}
\label{eq:newAZ}
\lim_{N \to \infty}
\bbP \Big( \bP^{f,\gb,h_c^{a}(\gb)}_{N,M,\omega} (A_N)>\frac 12
\text{ and } Z^{f,\gb,h_c^{a}(\gb)}_{N,M(N),\omega} > \frac 12  \Big)  \, =\, 0\,.
\end{equation} 

Since $ \bbE [ Z^{f,\gb,h_c^{a}(\gb)}_{N,M,\omega} ] = 1$, and because ${ \lbrace {Z^{f,\gb,h_c^{a}(\gb)}_{N,M,\omega}} \rbrace}_{N}$ is uniformly integrable, then \eqref{eq:infz} follows immediately from \cite[Lemma~4.6]{cf:G}.

For \eqref{eq:newAZ} we observe that the Fubini-Tonelli Theorem implies
\[
\bbE \left[ Z^{f,\gb,h_c^{a}(\gb)}_{N,M,\omega} \bP^{f,\gb,h_c^{a}(\gb)}_{N,M,\omega} (A_N) \right] = \bbE \bE  \Big[ \exp \Big( \sum_{n=1}^N \sum_{m=1}^M (\beta \omega_{n,m} - \log Q(\beta) )  \delta_{n,m} \Big) \mathbf{1}_{ A_{N}} \Big]
= \bP(A_N) \,,
\]
with $ \delta_{n,m} := \mathbf{1}_{(n,m) \in \tau}$. Hence $\lim_{N\to\infty} \bbE \big[ Z^{f,\gb,h_c^{a}(\gb)}_{N,M,\omega} \bP^{f,\gb,h_c^{a}(\gb)}_{N,M,\omega} (A_N) \big] =0$
and \eqref{eq:newAZ} follows because
\begin{multline}
\bbP \Big( Z^{f,\gb,h_c^{a}(\gb)}_{N,M,\omega}> \frac 12 \text{ and } \bP^{f,\gb,h_c^{a}(\gb)}_{N,M,\omega} (A_N)> \frac 12 \Big)\, 
=\, \bbP \Big( Z^{f,\gb,h_c^{a}(\gb)}_{N,M,\omega}\ind_{ \{\bP^{f,\gb,h_c^{a}(\gb)}_{N,M,\omega} (A_N)> \frac 12 \}}> \frac 12  \Big)\\
\,  \le\,  2\, \bbE \Big[ Z^{f,\gb,h_c^{a}(\gb)}_{N,M,\omega} \ind_{ \{\bP^{f,\gb,h_c^{a}(\gb)}_{N,M,\omega} (A_N)> \frac 12 \}} \Big] \, \le \, 4
\bbE \left[ Z^{f,\gb,h_c^{a}(\gb)}_{N,M,\omega} \bP^{f,\gb,h_c^{a}(\gb)}_{N,M,\omega} (A_N) \right]\, .
\end{multline}
\end{proof}

We now prove that ${ \lbrace {Z^{f,\gb,h_c^{a}(\gb)}_{N,M(N),\omega}} \rbrace}_{N}$ is uniformly integrable
(and this holds for an arbitrary choice of  $M(N)$) provided that the intersection renewal $\sigma = \tau \cap \tau'$ is terminating -- $\tau$ and $\tau'$ are two independent copies of $\tau$ -- and $\gb$ is small enough. Let us point out that, since $\sigma $ is a terminating renewal then the total number $\vert \sigma \vert$ of renewal points (except the origin), that is 
$\vert \sigma \vert= \sum_{(n,m)\in \bbN^2} \tilde \delta_{n,m}$ with $\tilde \delta_{n,m} = \mathbf{1}_{(n,m) \in \sigma}$, is a geometric random variable of parameter 
$\bP^{\otimes 2} ( \sigma_1 < \infty)$,  where $\sigma_1 < \infty$ simply means that both components of $\sigma_1$ are finite. 
This in particular implies that 
$\bP^{\otimes 2} ( \sigma_1 < \infty) =  1/\bE^{\otimes 2} [\vert \sigma \vert]$. Moreover it is straightforward to see that  
$\bE^{\otimes 2} [\vert \sigma \vert]= \sum_{n,m} \bP((n,m) \in \tau)^2$.
\medskip

\begin{lemma}
\label{th:tau'}
If $\sigma:=\tau\cap \tau'$ is terminating, then defining
\begin{equation}
\label{eq:beta1}
0< \beta_1 := \sup \Big\{ \, \beta : \, \log Q(2 \beta) - 2 \log Q(\beta)  < - \log \bP^{\otimes 2} ( \sigma_1 < \infty)  \Big\} \, ,
\end{equation}
we have that for every $\beta \in (0,\beta_1)$ the sequence ${ \lbrace {Z^{f,\gb,h_c^a(\gb)}_{N,M(N),\omega}} \rbrace}_{N}$ is bounded in $L^2(\bbP)$, and is therefore \emph{uniformly integrable}.
\end{lemma}

\begin{proof} 
We write $M=M(N)$ and we compute the second moment of the partition function:
\begin{equation}
\label{eq:Z2}
\begin{split}
\bbE \Big[ { \left( {Z_{N,M,\omega}^{f,\beta ,h_c^a(\beta)}} \right) }^2 \Big] & = 
 \bE^{\otimes 2} \Big[ \bbE \Big[ \exp \Big( \sum_{n=1}^N \sum_{m=1}^M (\beta \omega_{n,m} - \log Q(\beta)) (\delta_{n,m}+\delta'_{n,m}) \Big) \Big]  \Big] \\
 & = \bE^{\otimes 2} \Big[ \exp \Big( \sum_{n=1}^N \sum_{m=1}^M (\log Q(2 \beta) - 2 \log Q(\beta)) \tilde\delta_{n,m} \Big)  \Big] \,.
\end{split}
\end{equation}
The sequence ${ \lbrace Z_{N,M,\omega}^{f,\beta ,h_c^a(\beta)} \rbrace}_{N}$ is bounded in  $L^2(\bbP)$ if 
\begin{equation}
\label{eq:tau'}
\bE^{\otimes 2} \Big[ \exp \Big( \sum_{n,m=1}^{\infty}  (\log Q(2 \beta) - 2 \log Q(\beta)) \tilde\delta_{n,m} \Big) \Big] < \infty \,.
\end{equation} 
Since  $\vert \sigma \vert$  is a geometric random variable of parameter $\bP^{\otimes 2} ( \sigma_1 < \infty)$, \eqref{eq:tau'} holds if 
\begin{equation}
 \log Q(2 \beta) - 2 \log Q(\beta)  < - \log \bP^{\otimes 2} ( \sigma_1 < \infty) %= \log \bE^{\otimes 2} [ \vert \sigma \vert] 
 \,.
\end{equation} 
\end{proof}

%We now have all the ingredients to complete the proof of Theorem \ref{thm:irrel}.

\begin{proof}[Proof of Theorem \ref{thm:irrel}]
In view of what we want to prove and of
Proposition~\ref{th:existence+}, notably the explicit continuity estimate 
\eqref{eq:existence+.3}, it suffices to establish the result for $\gamma\in \bbQ$
and for $M=\lfloor \gamma N \rfloor$, which we shall assume till the end of the proof,
even if this explicit choice is used in full only at the very end.

%In view of \eqref{eq:inh}, to prove the first statement we have to show only that $h_c^q(\beta) \le h_c^a (\beta)$ for $\beta<\beta_1$. 

Because of Lemma \ref{th:tau'}, we have that the sequence ${ \lbrace Z_{N,M,\omega}^{f,\beta ,h_c^a(\beta)} \rbrace}_{N}$ is uniformly integrable for $\beta<\beta_1$.
Now for all $0<\eta < \ga$ (recall that if $\sigma$ is terminating, it implies that $\ga\leq 1$), we set
\begin{equation}
\label{eq:AN}
A_{N} \,:=\, \left\lbrace \vert \tau \cap \left((0,N] \times (0,M] \right) \vert \le  N^{\eta} \right\rbrace \,.
\end{equation}
 From Lemma~\ref{th:limp}, we have 
 that $\lim_N \bP(A_{N}) =0$. Observe also that
\begin{equation}
\label{eq:binfz}
\begin{split}
Z_{N,M,\omega}^{f,\beta ,h_c^a(\beta)+h} & = Z_{N,M,\omega}^{f,\beta ,h_c^a(\beta)} \bE_{N,M,\omega}^{f,\beta ,h_c^a(\beta)} \left[ \exp \left( h \vert \tau \cap \left((0,N]\times(0,M] \vert \right) \right)\right] \\
& \ge Z_{N,M,\omega}^{f,\beta ,h_c^a(\beta)} \bP_{N,M,\omega}^{f,\beta ,h_c^a(\beta)} \left( {A}^{c}_{N} \right) \exp \left(  h N^{\eta} \right) \,.
\end{split}
\end{equation}

Let us call $E_N$ the event whose probability is estimated from below in \eqref{eq:unint1}. Then on $E_N$, whose  
probability is at least $\zeta >0$, we have
\begin{equation}
\label{eq:zfb}
Z_{N,M,\omega}^{f,\beta ,h_c^a(\beta)+h} 
\ge \frac{1}{2} \Big( 1- \bP_{N,M,\omega}^{f,\beta ,h_c^a(\beta)} \left( A_{N}  \right) \Big) \exp \left( h  N^\eta\right) 
 \ge  \frac1 4 \exp \left(  h N^{\eta}\right) \,.
\end{equation}
Therefore we obtain
\begin{equation}
\label{eq:zfpb}
 \bbP \Big( Z_{N,M,\omega}^{f,\beta ,h_c^a(\beta)+h} \ge \frac 14 \exp \left(  h  N^{\eta}\right) \Big) \,\ge\,  \bbP\left(E_N\right) \,\ge\,  \zeta \, .
\end{equation}

Our aim is to prove that $\tf_\gamma(\beta, h+ h_c^{a}(\beta)) >0$ or more precisely give a lower bound for $\tf_\gamma(\beta, h+ h_c^{a}(\beta))$. We aim at using \eqref{eq:existence+.2}, this is why  we  have chosen $\gamma\in \bbQ$, and now we choose also $N$ such that
$\gamma N \in \bbN$, so $N=jq$, $j\in \bbN$ ($\gamma=p/q$). 
Since the first part of the proof exploits the free partition function, and not the constrained one for which \eqref{eq:existence+.2} holds, we use
Lemma~\ref{th:zfc} that guarantees that 
\begin{equation}
\label{eq:zfc1}
 \log Z^{f, \beta , h_c^a(\beta) +h}_{N,M,\omega} \le 
 \log Z^{c, \beta , h_c^a(\beta) +h}_{N,M,\omega}
+ c_1\Big( 1 +  \log  (N+M)  + \beta | \omega_{N,M}| \Big)  \,.
\end{equation}
Since there exists  $c_2 >1$ such that $\beta | \omega_{N,M}| < c_2 \log (N+M)$ with probability at least $1-\zeta/2$,  and recalling that $M\sim \gamma N$, we get that there exists $c_3>0$ such that
\[\bbP \Big(\log Z^{c, \beta , h_c^a(\beta) +h}_{N,M,\omega}  \le \log Z^{f, \beta , h_c^a(\beta) +h}_{N,M,\omega}  - c_3 \log N \Big) \leq \frac{\zeta}{2} \, .\]
Combining this with \eqref{eq:zfpb}, we get that
\begin{equation}
\label{eq:hA}
\bbP \Big(  \log Z^{c, \beta , h_c^a(\beta) +h}_{N,M,\omega}  \, \ge\,  \frac{1}{2} h N^{ \eta} - c_3\log N \Big) \, \ge\,  \frac \zeta 2
 \, .
\end{equation}
Now using the uniform bound $Z^{c, \beta , h_c^a(\beta)}_{N,M,\omega} \ge K(N+M) e^{\beta \omega_{N,M} - \log Q(\beta) }$ on the event $E_n^c$, we arrive at
\begin{equation}
\label{eq:irZ}
\begin{split}
 \bbE \log Z_{N,M,\omega}^{c,\beta ,h_c^a(\beta)+h} 
& \ge  \frac{\zeta}{4}   h  N^{\eta} - \frac{c_3\zeta}{2} \log N  + \log K(N+M) - \beta \bbE [ |\omega_{1,1}| ] - \log Q(\beta) \\
& \ge  c_4 h  N^{\eta} -  c_5\log N \, ,
\end{split}
\end{equation}
for suitably chosen $c_4,c_5 >0$.

At this point the choice $\gamma=p/q$ and $M=\gamma N \in \bbN$ enters the game. By 
\eqref{eq:existence+.2} we have
\begin{equation}
\label{eq:irf}
 \tf_{\gamma}(\gb,h_c^a(\gb) +h) \geq \sup_{N=jq : \, j =j_0, j_0+1, \ldots}  \left\{ c_4 h N^{\eta-1} - c_5  N^{-1} \log N \right\} \,,
 \end{equation} 
 and the fact that $j$ has to be chosen larger than a certain $j_0$ just reflects the fact that the estimates in this proof have been performed for a $N$  larger than a suitable $N_0$.
  We now estimate from below the right-hand side in \eqref{eq:irf} by choosing
  $N= h^{-\frac{1+\gep}{\eta}}$ (for some $\gep>0$ fixed): this means that we have chosen $h=(jq)^{-\eta/(1+\gep)}$. 
With this choice
\begin{equation}
\tf_{\gamma}(\gb,h_c^a(\gb) +h) \geq  c_4 h^{-\gep} N^{-1} - c_5\frac{1+\gep}{\eta}  N^{-1} \log \frac 1h \geq  N^{-1} = h^{\frac{1+\gep}{\eta}}\, ,
\end{equation}
where the last inequality holds provided that $h$ is small enough.
This is the estimate we were after since we can choose $\eta$ arbitrarily close to $\ga$ and $\gep$ close to $0$, but we have established it only for $h$ of the form $(jq)^{-\eta/(1+\gep)}$, $j=j_0, j_0+1, \cdots$.
However, we can use that $h\mapsto \tf_{\gamma}(\gb,h_c^a(\gb) +h)$ is non decreasing: having demonstrated that
$\tf_{\gamma}(\gb,h_c^a(\gb) +h)\ge h^{\frac{1+\gep}{\eta}}$ for $h=h_j:=(jq)^{-\eta/(1+\gep)}$
implies that $\tf_{\gamma}(\gb,h_c^a(\gb) +h)\ge \tfrac12 h^{\frac{1+\gep}{\eta}}$ for every sufficiently small $h$ (
this can be verified by checking that 
$\tfrac12 h_j^{\frac{1+\gep}{\eta}}$ is smaller than $h_{j+1}^{\frac{1+\gep}{\eta}}$).
This completes the proof of Theorem \ref{thm:irrel}.
\end{proof}

\medskip

The technique used to prove Theorem \ref{thm:irrel} could be adapted for $\ga >1$ to deduce the upper bound for the difference between quenched and annealed critical points. 

\begin{proposition}
Let $\ga >1$. There exists a slowly varying function $\tilde{L}(\cdot)$ such that
\begin{equation}
h_c^q(\beta) - h_c^a (\beta) \le 
\tilde{L}(1/\gb) \,  \gb^{\frac{2\ga}{\ga-1} \vee 4} \, , 
\end{equation}
for $\beta \le 1$.
\end{proposition}
%From Proposition~\ref{th:unint1} and the proof of Theorem~\ref{thm:irrel} (or rather \eqref{eq:AN})
\begin{proof}
As in the previous proof, it suffices to work with the case $\gamma = p/q\in \bbQ$ and $M=\lfloor \gamma N \rfloor$. We set 
\begin{equation}
N_\beta := \max \Big \lbrace N \in q\bbN \, : \,  \bbE \Big[ \Big( Z^{f,\beta,h_c^a(\beta)}_{N,M,\omega} \Big)^2 \Big]  \le 2 \Big \rbrace.
\end{equation}

Using Paley-Zygmund inequality, we therefore get that $\bbP(Z_{N,M,\go}^{f,\beta,h_c^a(\beta)} >1/2) \geq 1/8$ for any $N\leq N_{\gb}$, and we can then adapt the proof of Proposition \ref{th:unint1}.

Let us take $A_{N} \,:=\, \left\lbrace \vert \tau \cap \left((0,N] \times (0,M] \right) \vert \le  N/{2 \mu} \right\rbrace$. 
Since $\lim_{N \to \infty} \bP(A_N) =0$, and $\bbP(Z_{N,M,\go}^{f,\beta,h_c^a(\beta)} >1/2) >1/8$ for $N\leq N_{\gb}$, we find, exactly as in the proof of Proposition~\ref{th:unint1}, that there exists $N_0 \in \bbN$ such that for every $N_0 \le N \le N_\beta$ we have 
\begin{equation}
\bbP \Big( \bP^{f,\gb,h_c^{a}(\gb)}_{N,M,\omega} (A_N)\le \frac 12
\text{ and } Z^{f,\gb,h_c^{a}(\gb)}_{N,M,\omega} > \frac 12  \Big)  \, \ge\, \frac{1}{20} \,.
\end{equation}
Following the proof of Theorem \ref{thm:irrel} (see \eqref{eq:irf}), provided $N_{\gb} \geq N_0$, and since $N_{\gb} \in q\bbN$, we get that 
\begin{equation}
 \tf_{\gamma}(\gb,h) \geq   \big \{ c_6 (h- h_c^a(\beta))  - c_7  N_{\gb}^{-1} \log N_{\gb} \big \} \,.
 \end{equation} 
  
We therefore observe that if  $h- h_c^a(\beta) > c_7/c_6  N_{\gb}^{-1} \log N_{\gb}$ then  $\tf_\gamma(\beta,h) >0$. Hence we get that
\begin{equation}
\label{eq:hNbeta}
h_c^ q(\beta) - h_c ^a(\beta) \le \frac{c_7}{c_6}\, \cdot  \frac{\log N_\beta}{N_\beta} \,.
\end{equation}

It therefore boils down to estimating $N_\beta$, namely obtaining a lower bound.
Recall from \eqref{eq:Z2} that 
\begin{equation}
\bbE \Big[ \left( Z^{f,\beta,h_c^a(\beta)}_{N,M,\omega} \right)^2 \Big] = \bE^{\otimes 2} \Big[ \exp \Big(  (\log Q(2 \beta) - 2 \log Q(\beta)) \cH_{N,M}(\sigma) \Big)  \Big] \,,
\end{equation}
with $\cH_{N,M}(\sigma) = \sum_{n=1}^N \sum_{m=1}^M \mathrm{1}_{(n,m) \in \sigma}$, and $\sigma = \tau \cap \tau'$ the intersection renewal. Recall that   $\sigma$ is persisting for $\ga > 1$, see Proposition \ref{prop:transience}.

Note that for $\beta \le 1$, there exists $c_8$  such that 
$\log Q(2 \beta) - 2 \log Q(\beta) \le c_{8} \, \beta^ 2$, and that $\cH_{N,M} \leq \cH_{M,M}$ if $M\geq N$.
The question is therefore reduced to estimating $\bE^{\otimes 2} \left[ \exp (t \cH_{M,M}(\sigma)) \right]$, with $t=c_8 \beta^2$.
We have 
\begin{equation}
\begin{split}
\bE^{\otimes 2} \Big[ \exp (t \cH_{M,M}(\sigma)) \Big] &
= 1 + \sum_{k=1}^M \left( e^{tk} - e^{t (k-1)} \right) \bP^{\otimes 2} \left( \cH_{M,M}(\sigma) \ge k \right) \\
& \le 1+ (e^{t}-1) \sum_{k=1}^M e^{t k} \bP^{\otimes 2} \left( \cH_{M,M}(\sigma) \ge k \right) \,.
\end{split}
\end{equation}
In order to obtain an upper bound, we use the following fact
\begin{equation}
\bP^{\otimes 2} \left( \cH_{M,M}(\sigma) \ge k \right) = \bP^{\otimes 2} \left( \sigma_k \in (0,M]^2\right) \le 
 { \bP^{\otimes 2} \left( \sigma_1 \in (0,M]^2 \right) }^k \,.
\end{equation}
Then we get
\begin{equation}
\bE^{\otimes 2} \left[ \exp (t \cH_{M,M}(\sigma)) \right] \le 
 1 + (e^{t}-1) \sum_{k=1}^M \exp \left[ k \Big(t + \log \bP^{\otimes 2} \left( \sigma_1 \in (0,M]^2  \right) \Big)\right] \,.
\end{equation}
Let $\underline{\sigma}:= \sigma^{(1)} + \sigma^{(2)}$. An elementary observation is that $\bP^{\otimes 2} \left( \sigma_1 \notin (0,M]^2  \right) $ is of the same order as $\bP^{\otimes 2} \left( \underline{\sigma}_1 >M  \right) $: indeed, for every $M \in \bbN$ we have
\begin{equation}
\bP^{\otimes 2} \left( \underline{\sigma}_1 > 2 M \right) \, \le \, \bP^{\otimes 2} \left( \sigma_1 \notin (0,M]^2  \right) \, \le \, \bP^{\otimes 2} \left( \underline{\sigma}_1 >  M \right) \,.
\end{equation}
Therefore, using that $\log(1- x) \leq -x$ for $x\in [0,1]$, we get that
\begin{equation}
\log \bP^{\otimes 2} \left( \sigma_1 \in (0,M]^2 \right) \le   - \bP^{\otimes 2} \left( \underline{\sigma}_1 > 2M \right)  \leq  - c_{9} /U_{M,M} ,
\end{equation}
where we used Lemma~\ref{lem:tau1} to estimate $\bP^{\otimes 2} \left( \underline{\sigma}_1 > 2M \right)$ (provided that $M$ is large enough), with $U_{N,M} =\sum_{n=0}^N \sum_{m=0}^M \bP((n,m)\in\tau)^2$.
Since $M\leq \gamma N$ and $U_{N,N}$ is regularly varying, see Proposition \ref{prop:transience}, we get that
\begin{equation}
\bE^{\otimes 2} \left[ \exp (t \cH_{M,M}(\sigma)) \right] \le
 1 + (e^{t}-1) \sum_{k=1}^N \exp \Big(  k (t -  c_{10}/U_{N,N} ) \Big) \,.
\end{equation}
We therefore choose $N$ such that $ c_{10}/U_{N,N} \ge 3 t   = 3c_{8} \gb^2 $. By Proposition \ref{prop:transience}, for $\ga >1$, we can choose
\begin{equation}
N = N_{\gb}=
\tilde \psi (1/\gb) \gb^{- (\frac{2\ga}{\ga-1} \vee 4)}   \, ,
\end{equation}
for some slowly varying function $\tilde\psi(\cdot)$.
For this choice of $N$, we therefore get that
\begin{equation}
\bbE \left[ \left( Z^{f,\beta,h_c^a(\beta)}_{N,M,\omega} \right)^2 \right]  \le 1 + (e^{c_8 \beta^2} -1) \sum_{k=1}^N \exp(- 2 c_8 \beta^2 k ) \le 1 + \frac{e^{c_8 \gb^2} -1}{1- e^{-2 c_8 \gb^2}} \,,
\end{equation}
which is smaller than $2$ provided that $\gb$ is small enough.
It therefore implies that there exists some $\gb_1>0$ such that
\begin{equation}
\label{eq:Nbeta}
N_\beta \ge
\tilde \psi (1/\gb) \gb^{- (\frac{2\ga}{\ga-1} \vee 4)}  \qquad \text{for } \gb \leq \gb_1 \,.
\end{equation}

The proof is therefore complete by putting \eqref{eq:Nbeta} in \eqref{eq:hNbeta}.
\end{proof}

\section{Lower bound on the critical point shift}
\label{sec:fm}

From now on, $L_i(\cdot)$ will denote slowly varying functions and $C_i$ positive constants for $i=1,2,...$ Also, we sometimes treat certain large quantities as if they were integers, simply to avoid the integer-part notation; in all cases these can be treated as if the integer-part notation were in use.

Our proof is based on combining the fractional moment method and a change of measure argument, following the same strategy adopted  in \cite{cf:DGLT09}. Let
\begin{equation}
z_{n,m} :=  \exp \left( \beta \omega_{n,m} + h \right) \,.
\end{equation}

Choose $k \le N$ and $M $ such that $M \sim \gamma N $ and decompose the partition function \eqref{eq:pf} as follows, see Figure \ref{fig:br}:
\begin{equation}
\label{eq:dpf}
Z_{N,M,\omega}= Z_{N,M,\omega}^1 + Z_{N,M,\omega}^2 + Z_{N,M,\omega}^3   \,,
\end{equation}
with (recall the notation \eqref{eq:pp})
\[
\begin{split}
Z_{N,M,\omega}^1 = \sum_{n=k}^N \sum_{m=k}^M Z_{N-n,M-m,\omega} \sum_{i=0}^{k-1} \sum_{j=0}^{k-1} K(n-i+m-j) z_{N-i,M-j} Z_{(N-i,M-j),(N,M),\omega} \, ,\\
Z_{N,M,\omega}^2 = \sum_{n=1}^{k-1} \sum_{m=k}^M Z_{N-n,M-m,\omega} \sum_{i=0}^{n-1} \sum_{j=0}^{k-1} K(n-i+m-j) z_{N-i,M-j} Z_{(N-i,M-j),(N,M),\omega} \,,\\
Z_{N,M,\omega}^3= \sum_{n=k}^N \sum_{m=1}^{k-1} Z_{N-n,M-m,\omega} \sum_{i=0}^{k-1} \sum_{j=0}^{m-1} K(n-i+m-j) z_{N-i,M-j} Z_{(N-i,M-j),(N,M),\omega} \,.
\end{split}
\]
Note that $Z_{(N-i,M-j),(N,M),\omega}$ has the same law as $Z_{i,j,\omega}$ and that $Z_{N-n,M-m,\omega}$, $ z_{N-i,M-j}$ and $Z_{(N-i,M-j),(N,M),\omega}$ are independent for $i < n$ and $j <m$.

\begin{SCfigure}[50][htb]
 \centering
\includegraphics[scale=0.34]{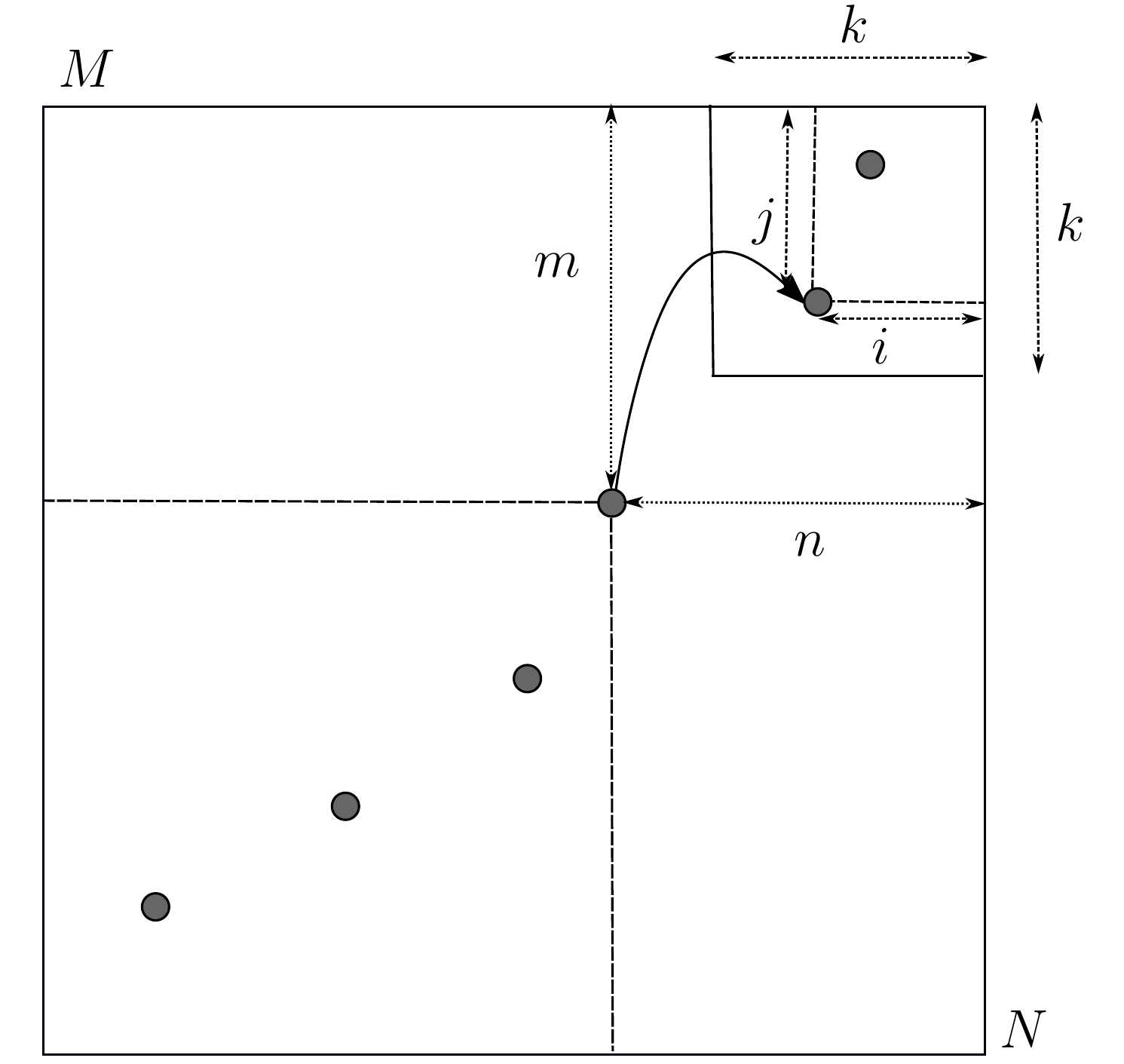}
\caption{\label{fig:br} Fixing a value $k$, the partition function is decomposed by summing over the values of the last renewal epoch outside the corner block $(N-k,N] \times (M-k,M]$, and the first one inside that block. We distinguish three cases: either the last renewal epoch is in $[0,N-k]\times [0,M-k]$ (which is the case represented in the figure, giving $Z_{N,M,\go}^1$), or it is in $(N-k,N]\times [0,M-k]$ ($Z_{N,M,\go}^2$) or in $[0,N-k]\times (M-k,M]$ ($Z_{N,M,\go}^3$). }
\end{SCfigure}

\medskip

Let $\delta \in (0,1)$ (that will be chosen close to $1$ later in the proof), and define
\begin{equation}
\cA_{N,M} := \bbE \Big[ {\left( Z_{N,M,\omega}  \right)}^{\delta} \Big] \qquad  \text{ for every } N,M \in \bbN^2, 
\end{equation}
with $\cA_{0,0}=1$, and $\cA_{i,0}= \cA_{0,i}=0$ for every $i\ge 1$. We apply the inequality ${ \left( \sum a_i \right) }^\delta \le \sum {a_i}^\delta$ (which holds for any finite and countable connection of positive real numbers) to the decomposition \eqref{eq:dpf} to get
\begin{equation}
\label{eq:A123}
\cA_{N,M} \le \cA_{N,M}^1 + \cA_{N,M}^2 + \cA_{N,M}^3 \,,
\end{equation}
where
\begin{equation}
\begin{split}
\cA_{N,M}^1 \le \bbE[z_{1,1}^\delta] \sum_{n=k}^N \sum_{m=k}^M \cA_{N-n,M-m} \sum_{i=0}^{k-1} \sum_{j=0}^{k-1} K(n-i+m-j)^\delta \cA_{i,j} \,.\\
\cA_{N,M}^2 \le \bbE[z_{1,1}^\delta] \sum_{n=1}^{k-1} \sum_{m=k}^M \cA_{N-n,M-m} \sum_{i=0}^{n-1} \sum_{j=0}^{k-1} K(n-i+m-j)^\delta \cA_{i,j} \,.\\
\cA_{N,M}^3 \le \bbE[z_{1,1}^\delta] \sum_{n=k}^N \sum_{m=1}^{k-1} \cA_{N-n,M-m} \sum_{i=0}^{k-1} \sum_{j=0}^{m-1} K(n-i+m-j)^\delta \cA_{i,j} \,.
\end{split}
\end{equation}

The key idea of the proof is to the following proposition.
\begin{proposition}
\label{th:rho}
For fixed $\beta$ and $h$, if there exist $k \in \bbN$ such that  $\rho_1 + \rho_2 + \rho_3 \le 1$ with
\begin{multline}
\label{eq:rho}
\rho_1 + \rho_2 + \rho_3 := \bbE[z_{1,1}^\delta] \left(  \sum_{n=k}^\infty \sum_{m=k}^\infty \sum_{i=0}^{k-1} \sum_{j=0}^{k-1}
+  \sum_{n=1}^{k-1} \sum_{m=k}^\infty \sum_{i=0}^{n-1} \sum_{j=0}^{k-1}
+  \sum_{n=k}^\infty \sum_{m=1}^{k-1} \sum_{i=0}^{k-1} \sum_{j=0}^{m-1}
 \right)  \\
 K(n-i+m-j)^\delta \cA_{i,j} \,,
\end{multline}
then $\tf_{\gamma} (\beta,h)=0$.
\end{proposition}

\begin{proof}
Define $\overline{A} := \max \big\{  \{ \cA_{i,j} ,\cA_{i,s}, \cA_{t,j} \} , \, 1 \le i,j \le k-1 , \, s,t \ge k \big\}$.
Note that by Jensen's inequality we have $\cA_{i,j} \leq \bbE[Z_{i,j}]^{\gd} \leq \exp(\gd  h \min\{i,j\})$, since there are at most $\min\{ i,j \}$ renewals in the region $\{1,\ldots,i\} \times \{1,\ldots,j\}$: we get that $\overline{A} \leq e^{h k}$.
Then from \eqref{eq:A123} and the fact that $\rho_1+\rho_2+\rho_3  \le 1$, we deduce (by induction) that $\cA_{N,M} \le \overline{A}\leq e^{k,h}$ for all $N,M$. Then by Jensen's inequality
\begin{equation}
\tf_\gamma^q(\beta ,h) = \limtwo{N \to \infty}{M/N\to\gamma} \frac{1}{\delta N} \bbE \log {(Z_{N,M,\omega})}^{\delta} \le \limtwo{N \to \infty}{M/N\to\gamma} \frac{1}{\delta N}  \log \cA_{N,M} = 0 \,.
\end{equation}
\end{proof}

Our aim is therefore to prove that for $h=h_c^a(\gb) + \gD_{\gb}^{\gep}$ (where $\gD_\gb^{\gep}$ is defined in Theorem~\ref{thm:rel}) we have that $\tf_1^{q}(\gb,h) =0$ (provided that $\gb$ is small enough), by showing that $\rho_1,\rho_2,\rho_3$ are smaller than $1/3$ for such $h$, for some $k=k_{\gb}$ wisely chosen.
For the choice of $k$, we pick $k$ proportional to the correlation length of the annealed system, that is $k \propto \tf(0,\gD_{\gb}^{\gep})^{-1}$, and in view of Theorem \ref{th:beta0} (here $\ga>1$), we can take
\begin{equation}
\label{def:k}
k = k_{\gb} = \frac{1}{\gD_{\gb}^{\gep}}
= \begin{cases}
\gb^{ - (1+\gep)\frac{2\ga}{\ga-1} } & \quad \text{ if } \ga\in(1,2]\, , \\
\gb^{-4} |\log \gb |^{6}  & \quad \text{ if } \ga>2.
\end{cases}
\end{equation}

\medskip
Note that, in view of \eqref{eq:rho} and \eqref{def:Kn}, provided that $\gd$ is close to $1$ so that $(2+\ga)\gd >2$, we have
\begin{equation}
\label{eq:rho1}
\rho_1 \le \sum_{i=0}^{k-1} \sum_{j=0}^{k-1} \frac{L_1(2 k -i -j)}{(2k -i -j)^{(2+\ga) \delta - 2}} \cA_{i,j} \,,
\end{equation}
and 
\begin{equation}
\label{eq:rho2}
\rho_2 \le  \sum_{i=0}^{k-1} \sum_{j=0}^{k-1}  \frac{L_2(k-j)}{(k-j)^{(2+\ga)\delta -2}}  \cA_{i,j} \,.
\end{equation}
The $\rho_3$ case being symmetric to $\rho_2$, we can therefore focus on $\rho_1$ and $\rho_2$.

\subsection{Finite-volume fractional moment estimate}
\label{sec:com}
To estimate \eqref{eq:rho1} and \eqref{eq:rho2}, we need a good control over the fractional moment $\cA_{i,j}$ for any $i,j\leq k$, and we provide estimates in this section.

First of all, using Jensen's inequality, we have that $\cA_{i,j} \le { (\bbE Z_{i,j,\omega})}^\delta$. Moreover, because $h=h_c^a(\gb)+\gD_{\gb}^\gep$,  we get that for any $i,j\leq k$
\[
\bbE Z_{i,j,\omega}  = \bE \left[ \exp \left(  \gD_{\gb}^{\gep} \vert \tau \cap \lbrace 1,...,i \rbrace \times \lbrace 1,..., j \rbrace \vert \right) \mathbf{1}_{(i,j)\in \tau} \right] \le e \, \bP((i,j) \in \tau) \,,
\]
since $\vert \tau \cap \lbrace 1,...,i \rbrace \times \lbrace 1,..., j \rbrace \vert \leq k$ and thanks to of our choice of $k = (\gD_{\gb}^{\gep})^{-1}$.
We therefore get that,
\begin{equation}
\label{eq:ineqA}
 \cA_{i,j} \leq e^{\gd} \bP((i,j) \in \tau)^{\gd} ,
\end{equation}
and $\bP((i,j)\in\tau)$ can be estimated thanks to Theorems \ref{th:re01}-\ref{th:reD}.

However, this estimate is rather rough, especially when $i,j$ is close to the diagonal (that is for example $i\leq j \leq i + a_i$ where ${(a_n)}_{n \ge 0}$ is the scaling sequence for $\tau_n$, defined in Section \ref{sec:defan}). We therefore prove the following proposition:
\begin{proposition}
\label{prop:chgmeasure}
Let $h=h_c^a(\gb)+ \gD_{\gb}^\gep$ and $k=(\gD_\gb^{\gep})^{-1}$. Then, define also
\begin{equation}
\label{def:ell}
 \ell_i := 
 \begin{cases}
  i^{(1+\gep^3)/\ga} & \quad \text{if } \ga\in(1,2] ,\\
 C \sqrt{i\log i} & \quad \text{if } \ga>2,
 \end{cases}
\end{equation}
so in any case $\ell_i \gg a_i$.
There exists some $k_0$ such that, provided that $k\geq k_0$ then for all $\sqrt{k}\leq i\leq k$ and $i\leq j\leq i+\ell_i$ we have that
\begin{equation}
\label{boundAij}
\cA_{i,j} \leq  L_{10}(i)\Big(  i^{\gd(1-\frac{1}{\ga\wedge2})} (\ell_i)^{-\gd \ga} +  \frac{i^{-\gd(1+\ga)} \ell_i^{ \gd}}{\gb^{2 \gd}} +  i^{-\frac{\gd}{\ga\wedge 2}} e^{- c (\gb^2 i /\ell_i)^{1/2}} \Big) .
\end{equation}
\end{proposition}
This result is the core of the proof, and is based on a change of measure argument.  
With this result in hand, we are able to show that $\rho_1$ and $\rho_2$ are small,  for $\ga>2$ in Section \ref{sec:ga2} and for $\ga\in(1,2]$ in Section \ref{sec:ga1}. 
Let us apply this proposition to get bounds on $\cA_{i,j}$ in the different cases.

\smallskip

{\bf Case $\ga>2$.} We get that uniformly for $k/2 \leq i\leq k$ and $i\leq j\leq i+ C' \sqrt{k\log k}$
\begin{equation}
\label{boundA2}
\begin{split}
\cA_{i,j} & \leq L_{11}(k) k^{\frac{\gd}{2} (1-\ga)} + L_{12}(k) \frac{k^{-\gd (1+\ga)}  (k \log k)^{\gd/2}}{\gb^{2 \gd}}  + L_{13}(k) k^{- \gd/2} e^{- c (\gb^4 k /\log k)^{1/4}} \\
& \leq L_{14}(k) \Big( k^{\frac{\gd}{2} (1-\ga)} + k^{-\gd \ga} +  k^{- \gd/2} e^{- c(\log k)^{5/4} }\Big) \leq  L_{15}(k) k^{-\frac{\gd}{2} (\ga-1)} \, ,
\end{split}
\end{equation}
where the choice \eqref{def:k} of $k=(\gD_{\gb}^{\gep})^{-1} = \gb^{-4} |\log \gb|^{6}$ is crucial, to get that $\gb^4 k /\log k  \geq c (\log k)^5$. For the last inequality, we observe that the first term dominates.

\smallskip
{\bf Case $\ga\in(1,2]$.} We use also the choice \eqref{def:k} of $k=(\gD_{\gb}^{\gep})^{-1} = \gb^{-(1+\gep) 2\ga /(\ga-1)} $ to get that uniformly for $k^{1-\gep^2} \leq i\leq k$, we have provided that $\gep$ is small enough
\begin{equation}
\gb^2 i /\ell_i = k^{- \frac{\ga-1}{(1+\gep)\ga}} i^{\frac{\ga -1-\gep^3}{\ga}} \geq k^{ \frac{1}{(1+\gep) \ga} (\gep (\ga-1) +O(\gep^2))}\geq k^{\gep (\ga-1)/2\ga} \, .
\end{equation}
Therefore, using also that for $i\leq k$, $\gb^{-2} \ell_i \leq k^{\frac{\ga-1}{(1+\gep) \ga}} k^{(1+\gep^3)/\ga} \leq k$ (if $\gep$ has been fixed small enough), we have that uniformly for $k^{1-\gep^2} \leq i\leq k$ and $i\leq j\leq i+ \ell_i$,
\begin{equation}
\label{boundA1}
\begin{split}
\cA_{i,j}  &\leq L_{10}(i) i^{-\gd(\frac1\ga +\gep^3)} + L_{10}(i) i^{-\gd \ga} k^{\gd}  + L_{10}(i) i^{-\gd/\ga} e^{-c k^{\gep (\ga-1)/4\ga }} \\
 & \leq L_{16}(k) k^{-\gd(1-\gep^2)(\frac1\ga +\gep^3) } \leq L_{16}(k) k^{-\frac\gd\ga (1  + \gep^2 /2) }
\end{split}
\end{equation}
where again, for the second to last inequality, we observe that the first term dominates, since $1/\ga >\ga$ and $\gep$ can be fixed arbitrarily small.

\begin{proof}[Proof of Proposition \ref{prop:chgmeasure}]
The idea is to use a change of measure argument.
We define a strip $J_{i,j}$ in which we will tilt the environment by some quantity $\lambda$ (to be chosen wisely): 
\begin{equation}
J_{i,j} := \Big\{ (n,m) \in  \llbracket 0, i \rrbracket \times \llbracket 0, j \rrbracket  \, ;\,     |n-m| \leq 2 \ell_i \Big\}  \,,
\end{equation}
and hence $\# J_{i,j} \leq 2 i \ell_i $. The width $2\ell_i$ of the strip is chosen because of the scaling of the bivariate renewal: it is very unlikely that the renewal deviates from the diagonal by more than $\ell_i$, see Theorem \ref{th:reD}.

Now, for $\lambda\in\bbR$ and $i,j\in\bbN$, we define a new probability measure $\bbP_{i,j,\lambda}$, under which the $\omega_{n,m}$ are still independent variables, but tilted by $\lambda$ in the strip $J_{i,j}$:
%$ \bbE_{i,j,\lambda} = - \lambda \mathbf{1}_{(n,m) \in J_{i,j}}$. This measure is absolutely continuous with respect to $\bbP$,
\begin{equation}
\label{eq:newmeasure}
\frac{\mathrm{d} \bbP_{i,j,\lambda}}{\mathrm{d} \bbP} (\omega) = \frac{1}{Q(-\lambda)^{ \# J_{i,j}}} \exp \Big( - \lambda \sum_{(n,m) \in J_{i,j}} \omega_{n,m} \Big) \,,
\end{equation}
where $Q(\cdot)$ is defined in \eqref{eq:defQ}. Observe now that by H\"older inequality

\begin{equation}
\label{eq:H1}
\begin{split}
\cA_{i,j} & =
\bbE_{i,j,\lambda} \Big[ {(Z_{i,j,\omega})}^\delta \frac{\mathrm{d}\bbP}{\mathrm{d}\bbP_{i,j,\lambda}} (\omega) \Big]  \le \bbE_{i,j,\lambda}\big[  Z_{i,j,\omega} \big]^{\gd}\,  
  \bbE_{i,j,\lambda} \bigg[ {\bigg(\frac {\mathrm{d}\bbP}{\mathrm{d}\bbP_{i,j,\lambda}}(\omega)  \bigg) }^{1/(1-\delta)}  \bigg] ^{1-\delta}   \,.
\end{split}
\end{equation}
The second term in the right-hand side of \eqref{eq:H1} is equal to
\begin{equation}
\label{eq:CM1}
{\bbE_{i,j,\lambda} \bigg[ {\bigg(\frac {\mathrm{d}\bbP}{\mathrm{d}\bbP_{i,j,\lambda}}(\omega)  \bigg) }^{1/(1-\delta)}  \bigg] }^{1-\delta}  
= { \left( Q(-\lambda)^{\delta} Q (\lambda \delta /(1-\delta) )^{1-\delta} \right)}^{\# J_{i,j}} \,.
\end{equation}
Observe that there exists $c_{11} >0$ such that $0 \le \log Q(x) \le c_{11} x^2$ for $\vert x \vert \le 1$.
% with $c_1 := (1/2) \max_{ x \in [-1,1] } \left\lbrace \mathrm{d^2} \left( \log Q(x)\right)/\mathrm{dx^2} \right\rbrace$.
Therefore for $\vert \lambda \vert \le \min (1,(1-\delta)/\delta)$ and by \eqref{eq:H1} and \eqref{eq:CM1}, we get
\begin{equation}
\cA_{i,j}\le \bbE_{i,j,\lambda}\big[  Z_{i,j,\omega} \big]^{\gd}  \exp \Big( c_{11} \Big( \frac{\delta(1+\gd)}{1-\delta} \Big) \lambda^2 \# J_{i,j} \Big)\,.
\end{equation}
Now, we choose $\lambda:= ( i\ell_i)^{-1/2}$, so that $\lambda^2 \# J_{i,j} \leq 2$, and
\begin{equation}
\label{eq:afterHolder}
\cA_{i,j} \leq e^{4c_{11}/(1-\gd)} \bbE_{i,j,\lambda}\big[  Z_{i,j,\omega} \big]^{\gd},
\end{equation}
so that we are left with estimating $\bbE_{i,j,\lambda}\big[  Z_{i,j,\omega}) \big]$ for $\lambda:= (i\ell_i)^{-1/2}$.

\medskip
%
%By Section~\ref{sec:com}, choosing $\lambda = 1/\sqrt{\# J_{i,j}} \le \min (1,(1-\delta)/\delta)$ (for $a$ small),we get
%\begin{equation}
%\cA_{i,j} \le { \left[ \bbE_{i,j,1/\sqrt{\# J_{i,j}}} (Z_{i,j,\omega}) \right] }^{\delta} \exp(c_1 \delta/(1-\delta)) \,.
%\end{equation}
%It suffices then to prove that $\bbE_{i,j,1/\sqrt{\# J_{i,j}}} (Z_{i,j,\omega})$ is small.
Recall \eqref{eq:hann} and the definition \eqref{eq:newmeasure} of $\bbP_{i,j,\gl}$. Using that $\bbE_{i,j,\gl}[e^{\gb \go_{n,m}}]$ equals $Q(\gb) =e^{- h_c^a(\gb)}$ if $(n,m) \notin J_{i,j}$ and $Q(\gb-\gl)/Q(-\gl)$ if $(n,m)\in J_{i,j}$, we have that for every $\beta$, $h$, $\lambda$ and $(i,j)$
\begin{equation}
\label{eq:Q1}
\begin{split}
\bbE_{i,j,\lambda} \big[ Z_{i,j,\omega} \big] & = \bE \bigg[  e^{(h-h_c^a(\gb) ) \vert \tau\cap \llbracket 0, i \rrbracket \times \llbracket 0, j \rrbracket \vert } 
      { \Big(  \frac{Q(\beta-\lambda)}{Q(\beta) Q(-\lambda)} \Big)}^{ \vert \tau \cap J_{i,j}  \vert } \mathbf{1}_{(i,j) \in \tau} \bigg] \\
      & \leq e\,  \bE \bigg[ 
      { \Big(  \frac{Q(\beta-\lambda)}{Q(\beta) Q(-\lambda)} \Big)}^{ \vert \tau \cap J_{i,j}  \vert } \mathbf{1}_{(i,j) \in \tau} \bigg],
\end{split}
\end{equation}
where we used that $\vert \tau\cap \llbracket 0, i \rrbracket \times \llbracket 0, j \rrbracket \vert \leq k$ and $h-h_c^a(\gb) = \gD_{\gb}^{\gep} = k^{-1}$.

Now, observe that $ \frac{Q(\beta-\lambda)}{Q(\beta) Q(-\lambda)} = 1 - \gl \gb +o(\gl^2 + \gb^2)$ as $\gl,\gb\downarrow 0$.
Here, because of our choice \eqref{def:k} of $k=(\gD_{\gb}^{\gep})^{-1}$, we have that $k\geq \gb^{-4}$. Since we are considering $i\geq \sqrt{k}$, and using that $\ell_i \geq \sqrt{i}$, we have that $\gl:= (i\ell_i)^{-1/2} \leq i^{-3/4} \leq k^{-3/8}$, and hence we have that $\gl \leq \gb^{3/2}\leq \gb$.
Therefore, there exists a constant $c_{12}>0$ such that provided that $\gb$ is small enough (or $k$ is large enough) we have
\begin{equation}
\label{eq:Q2}
\frac{Q(\beta-\lambda)}{Q(\beta) Q(-\lambda)} \le \exp(- c_{12} \beta \lambda ) \,,
\end{equation}
and we end up with
\begin{align}
\label{eq:max1}
e^{-1}\bbE_{i,j,\lambda} &\big[ Z_{i,j,\omega} \big]  \leq  \bE \left[ e^{ -c_{12} \gb \lambda \vert \tau \cap J_{i,j}  \vert } \mathbf{1}_{(i,j) \in \tau} \right] \notag \\
      & \leq  \bP\left(  \exists s , \tau_s \notin J_{i,j} \, , \, (i,j)\in\tau\right) + \bE\left[ e^{- c_{12} \gl \gb \vert \tau \cap \lbrace 1,...,i \rbrace \times \lbrace 1,...,j \rbrace \vert } \ind_{(i,j)\in\tau} \right] ,
\end{align}
where in the last term we dropped the indicator function that all renewals occur in the strip $J_{i,j}$.
We now estimate these two terms separately.

\begin{lemma}
\label{lem:notinJ}
There exists a slowly varying function $L_{4}$ such that, for every $1\leq i\leq j\leq i+\ell_i$ we have
\begin{equation}
\label{eq:notinJ}
 \bP\left(  \exists s , \tau_s \notin J_{i,j} \, , \, (i,j)\in\tau\right) \leq  
 L_{4}(i) \, i^{1-1/\ga \wedge 2} \, (\ell_i)^{- \ga} \,.
\end{equation}
\end{lemma}
\begin{proof}
Let us first observe that by symmetry, we get that
\begin{multline}
\label{eq:ts1}
 \bP\left(  \exists s , \tau_s \notin J_{i,j} \, , \, (i,j)\in\tau\right)   = 2\, \bP\left(\exists s , \tau_s \notin J_{i,j}, \tau_s^{(1)}\leq i/2 \, , \, (i,j)\in\tau \right)
% = \sum_{s_0=1}^{i-1}   \bP\left( \tau_{s_0} \notin J_{i,j} \, , \, \tau_{s_0-1} \in J_{i,j} \, , \, (i,j)\in\tau\right)
  \\
 \leq 2 \sum_{(a,b) \in J_{i,j}} \sumtwo{(k,l);\, (a+k,b+l) \notin J_{i,j}}{a+k\leq i/2} \bP((a,b) \in \tau) K(k+l) \bP \left( (i-a-k, j-b-l) \in \tau \right) \,.
\end{multline}

From Theorem \ref{th:reD}, we see that the last term in the double sum of \eqref{eq:ts1} is bounded above by $c_7/a_i$ (since $i-a-k \geq i/2$). We get that \eqref{eq:ts1} is bounded above by
\begin{equation}
\label{eq:ts2}
 \frac{c_{13}}{a_i} \,  \sum_{a=1}^{i/2} \sum_{r=0}^{\ell_i} \bP \left( (a,a+r) \in \tau \right)  \sum_{k = 1}^{i/2 - a} \sum_{l \ge \ell_i - a-r} K(k+l) 
 \le  \frac{c_{12}}{a_i} \,  \bP\left(  \exists t \le i/2 , \tau_t \notin \overline{J}_{i} \right)\,,
\end{equation}
with
$\overline{J}_{i} := \big\{ (a,b)  \, ;\,     |a-b| \leq  \ell_i \big\}  \,.$

Let us now define 
\begin{equation}
S_k = \tau_k^{(1)} - \tau_k^{(2)} \,,
\end{equation}
Then we see that
\begin{equation}
\label{eq:ts3}
 \bP\left(  \exists t \le i/2 , \tau_t \notin \overline{J}_{i} \right)  = 
 \bP \Big( \max_{t \le i/2} \vert S_t \vert \ge \ell_i \Big) \,.
\end{equation}
Observe that $\left\lbrace S_k \right\rbrace$ is a centred random walk in the domain of attraction of a stable law of index $\ga >1$. From the Lemma in \cite{cf:SW11} for the case $\ga \in (1,2]$ (and infinite variance) and \cite[Corollary 1]{cf:Pi81} or equation (12) in \cite{cf:BB00} for the case $\ga >2$ (or $\ga=2$ and finite variance) we get that
\begin{equation}
\label{eq:maxS}
 \bP \Big( \max_{t \le i/2} \vert S_t \vert \ge \ell_i \Big) \le i \, L_{3}(i) \, (\ell_i)^{- \ga} \,.
\end{equation} 
Therefore by \eqref{eq:ts2}, \eqref{eq:ts3} and \eqref{eq:maxS}, we obtain \eqref{eq:notinJ}.
\end{proof}

%For $U_2$, since $j \le i + \ell_i$ and using again Theorem \ref{th:reD} we obtain 
%\begin{equation}
%U_2 \le \frac{C}{a_i} \sum_{s=1}^{i/2} \sum_{ t= \ell_i \wedge s}^s \bP( (s,s-t)\in\tau )\,,
%\end{equation}
%which is interpreted as $U_1$ and we are done.

\begin{lemma}
\label{lem:homogen}
Assume that $i\leq j$ and $\ga>1$. There exist  constants $c_{15},c_{16}>0$ such that, for any sequence $u_i \leq 1$ (we may take $u_i\to 0$ as $i\to+\infty$), we have
\[
Z_{i,j}(-u_i) := \bE\left[ e^{- u_i\, \vert \tau \cap \lbrace 1,...,i \rbrace \times \lbrace 1,...,j \rbrace \vert } \ind_{(i,j)\in\tau} \right] 
\leq  c_{15}\frac{ K(i+j)}{u_i^2} + \bP((i,j)\in\tau) \, e^{ - c_{16} i u_i }\, .
\]
In particular, we always have
\[ Z_{i,j}(-u_i) \leq   \frac{L_{5}(i) i^{-(2+\ga)}}{u_i^2}  + L_{6} (i) i^{-1/\ga\wedge 2} e^{- c_{16} i u_i}\]
\end{lemma}
\begin{proof}
The last inequality comes from the fact that for $i\leq j$ we have $K(i+j) \leq c L(i) i^{-(2+\ga)}$, and the fact that Theorem \ref{th:reD}  give
$\bP((i,j)\in\tau) \leq c_{14} /a_i$ with $a_i = \psi(i) i^{1/\ga\wedge 2}$.

We write
\[Z_{i,j}(-u_{i}) = \sum_{k=1}^i e^{- k u_i} \bP(\tau_k = (i,j) ) =  \Big( \sum_{k=1}^{i/2\mu} + \sum_{k=i/2\mu}^i \Big) e^{- k u_i} \bP \big( \tau_k = (i,j) \big) \, .\]
For the first sum, we use Theorem \ref{thm:locallargedev} to get that $\bP(\tau_k = (i,j) ) \leq c_{15} k K(i+j)$ for $k\leq i/2\mu$, so
\begin{equation}
\label{sum1}
\sum_{k=1}^{i/2\mu}  e^{- k u_i} \bP(\tau_k = (i,j) ) \leq c_{15} \frac{K(i+j)}{u_i^2} \sum_{k=1}^{+\infty} u_i ku_i e^{-k u_i} \leq c_{15}  \frac{K(i+j)}{u_i^2} ,
\end{equation}
where for the last inequality we bounded the sum by a constant times $\int_{\bbR_+} x e^{-x} \dd x$ (thanks to a Riemann-sum approximation for sequences $u_i\to 0$).

For the second sum we simply bound $k$ by $i/2\mu$ to get that it is smaller than
\begin{equation}
\label{sum2}
e^{- i u_i /2\mu } \sum_{k=i/2\mu}^i   \bP(\tau_k = (i,j) ) \le  \bP((i,j) \in\tau) e^{- c_{16} i u_i} \, .
\end{equation}
Combining \eqref{sum1} and \eqref{sum2}, we obtain Lemma \ref{lem:homogen}.
\end{proof}

Using Lemma \ref{lem:notinJ} and Lemma \ref{lem:homogen} in \eqref{eq:max1}, and with $u_i= c_{17} \gl\gb = c_{17}\gb (i \ell_i)^{-1/2}  \leq 1$,
\begin{equation}
\bbE_{i,j,\lambda} \big[ Z_{i,j,\omega} \big]  \leq  L_{7}(i) i^{1-1/\ga\wedge 2} (\ell_i)^{-\ga} + L_{8}(i) \frac{i^{-(1+\ga)} \ell_i}{\gb^2} + L_{9}(i) i^{-1/\ga\wedge 2} e^{- c_{17} \gb i^{1/2} \ell_i^{-1/2} }\, .
\end{equation}

Finally, this concludes the proof of Proposition \ref{prop:chgmeasure} thanks to \eqref{eq:afterHolder}, using that $(a+b+c)^{\gd} \leq a^\gd + b^{\gd} + c^{\gd}$ for $\gd\in(0,1)$.

\end{proof}

\subsection{Conclusion of the proof of Theorem \ref{thm:rel} in the case $\ga > 2$}
\label{sec:ga2}

%Remark that $\bbE[z_{1,1}^\delta]$ is bounded above by a constant independent of $a$, then by \eqref{eq:1rec} and \eqref{eq:sup}, there exists a slowly varying function $L_1(\cdot)$ such that (see \eqref{eq:rho} for the definition of $\rho_1$)

Let $\delta < 1$ be sufficiently close to $1$ to have
\begin{equation}
\label{eq:condga}
(2+\ga) \delta > 4 \,,
\end{equation} 
which implies that $\gd(\ga-1) >1$.

We start by estimating $\rho_1$.
Let $R$ be a large constant and split the sum in \eqref{eq:rho1} as 
\begin{equation}
S_1+ S_2 \, := \, \bigg( \sum_{i,j=0}^{k-R-1} + \sum_{i,j=k-R}^{k-1} \bigg) \frac{L_1(2k -i -j)}{(2k -i-j)^{(2+\ga) \delta - 2}} \cA_{i,j} \,,
\end{equation}
and 
\begin{equation}
S_3+S_4 \, := \, \bigg(  \sum_{i=0}^{k-R-1} \sum_{j=k-R}^{k-1} + \sum_{i=k-R}^{k-1} \sum_{j=0}^{k-R-1} \bigg) \frac{L_1( 2k -i -j)}{(2k -i-j)^{(2+\ga) \delta - 2}} \cA_{i,j} \,.
\end{equation}

%
%Using Jensen's inequality, we have that $\cA_{i,j} \le { (\bbE Z_{i,j,\omega})}^\delta$. From \eqref{eq:pf} and \eqref{eq:hann} one sees that
%\begin{equation}
%\label{eq:ineqA}
%\bbE Z_{i,j,\omega}  = \bE \left[ \exp \left( \frac{a \beta^4}{{\left( \log(1/\beta^4) \right)}^{1+\epsilon}} \vert \tau \cap \lbrace 1,...,i \rbrace \times \lbrace 1,..., j \rbrace \vert \right) \mathbf{1}_{(i,j)\in \tau} \right] \le e \, \bP((i,j) \in \tau) \,,
%\end{equation}
%where the last inequality follows from the fact that $\vert \tau \cap \lbrace 1,...,i \rbrace \times \lbrace 1,..., j \rbrace \vert \le (i \wedge j) \le  k =\frac{{\left( \log(1/\beta^4) \right)}^{1+\epsilon}}{(a \beta^4)}$. 
% Moreover, from Theorem~\ref{th:reD} and Theorem~\ref{th:reoff}, there exists $C_1 >0$ such that
%\begin{equation}
%\label{eq:cons1}
%\bP \left( (i,j) \in \tau \right) \le \frac{C_1}{\sqrt{\min (i,j)}} \,.
%\end{equation}

Using the fact that $A_{i,j} \le e^\delta$ from \eqref{eq:ineqA}, we get 
\begin{equation}
\label{eq:S1}
S_1 \le \frac{L_{17}(R)}{R^{(2+\ga) \delta-4}} \,,
\end{equation}
and the right-hand side of \eqref{eq:S1} can be made small by \eqref{eq:condga} and because $R$ is large.

For $S_2$, there exists $C_4$ such that
$S_2 \le C_4 \max_{k - R \le i,j < k} \cA_{i,j} $,
and from \eqref{eq:ineqA}, combined with Theorem \ref{th:reD}, there exists $C_5$ such that
\begin{equation}
\max_{k - R \le i,j < k} \cA_{i,j} \le   e^{\delta} \, \max_{k - R \le i,j < k}  \bP((i,j) \in \tau) ^\delta \le \frac{C_5}{k^{\delta/2}} \,,
\end{equation}
then $S_2$ is arbitrarily small for $k$ large.

Since $S_3$ and $S_4$ are the same quantity, we just focus on $S_3$. Since $A_{i,j} \le e^ \delta$ from \eqref{eq:ineqA}, we obtain
\begin{equation}
S_3 \le \frac{L_{18}(R)}{R^{(2+\ga) \delta-4}} \,,
\end{equation}
which again can be made small in view of the condition \eqref{eq:condga} and because $R$ is large.
Hence $\rho_1$ can be made arbitrarily small by choosing $R$ large and $k$ large (\textit{i.e.}\ $\gb$ small).

\medskip
Let us now look at $\rho_2$ in \eqref{eq:rho2}.
We split the sum  to:
\begin{equation}
S_5 + S_6 =  \bigg( \sum_{i=0}^{k-1} \sum_{j=0}^{i}  +  \sum_{i=0}^{k-1} \sum_{j=i+1}^{k-1} \bigg)  \cA_{i,j} \frac{L_{2}(k-j)}{(k-j)^{(2+\ga)\delta -2}} \,.
\end{equation}

Let us first study $S_5$: 
\begin{equation}
\begin{split}
S_5 & = \sum_{j=0}^{k-1} \sum_{i=j}^{k-1}  \cA_{i,j}  \frac{L_{2}(k-j)}{(k-j)^{(2+\ga)\delta -2}} \\
& = \bigg( \sum_{j=0}^{ k/2  } \sum_{i=j}^{k-1} + \sum_{ j=  k/2  +1 }^{k-1}  \sum_{i=j}^{k-1} \bigg)
\cA_{i,j}  \frac{L_{2}(k-j)}{(k-j)^{(2+\ga)\delta -2}} := S_{5a} + S_{5b} \,.
\end{split}
\end{equation}

Using that  $\cA_{i,j} \le e^{\delta}$ from \eqref{eq:ineqA}, we get
\begin{equation}
S_{5a} \le \frac{L_{19}(k)}{k^{(2+\ga)\delta -4 }} \,.
\end{equation}

For $S_{5b}$, we use \eqref{eq:ineqA} and Theorem \ref{th:reD} which gives that if $i\geq j$ $\cA_{i,j} \leq  cst. j^{-\gd/2}$ to get
\begin{equation}
S_{5b} \le \sum_{ j=  k/2  +1 }^{k-1}  \sum_{i=j}^{k-1} \frac{C_6}{j^{\delta/2}} \frac{L_{2}(k-j)}{(k-j)^{(2+\ga)\delta -2}}   \le \frac{C_7}{k^{ \delta/2}}  \,. 
\end{equation}
Then $S_5$ can be made small for $k$ large and from condition \eqref{eq:condga}.

Now we split $S_6$ as

\begin{equation}
S_{6a} + S_{6b} = \bigg( \sum_{i=0}^{ k/2 } +  \sum_{i=  k/2  + 1}^{k-1} \bigg) \sum_{j=(i+ \ell_i +1) \wedge (k-1)}^{k-1} \cA_{i,j} \frac{L_{2}(k-j)}{(k-j)^{(2+\ga)\delta -2}} \,.
\end{equation}
and
\begin{equation}
S_{6c} + S_{6d} = \bigg( \sum_{i=0}^{ k/2 } +  \sum_{i= k/2  +1}^{k-1} \bigg) \sum_{j=i + 1}^{(i+ \ell_i) \wedge (k-1)} \cA_{i,j} \frac{L_{2}(k-j)}{(k-j)^{(2+\ga)\delta -2}} \,.
\end{equation}

Using \eqref{eq:ineqA} and Theorem \ref{th:reD}, we see that
\begin{equation}
\begin{split}
S_{6a} & \le \sum_{i=0}^{ k/2 }  \sum_{j= 1}^{ 3k/4} \frac{L_{20}(k-j)}{(k-j)^{(2+\ga)\delta -2}} + \sum_{i=0}^{ k/2 } \sum_{j= 3k/4}^{k-1} \frac{C_8 i^{\gd} L(j-i)^{\gd}}{(j-i)^{(1+\ga) \delta }} \frac{L_{2}(k-j)}{(k-j)^{(2+\ga)\delta -2}} \\
& \le \frac{L_{21}(k)}{k^{(2+\ga)\delta -4}} + \frac{L_{22}(k)}{k^{\ga \delta  -1}} \,,
\end{split}
\end{equation}
and
\begin{align*}
S_{6b} &\le \sum_{i=  k/2  +1}^{k-1} \sum_{j=(i +c\sqrt{k\log k} ) \wedge (k-1)}^{k-1} \frac{C_8 i^{\gd} L(j-i)^{\gd}}{(j-i)^{(1+\ga) \delta }} \frac{L_{2}(k-j)}{(k-j)^{(2+\ga)\delta -2}} \\
&\le C_8 k^{\gd} \sum_{j=k/2}^{k-1}  \sum_{x\geq c \sqrt{k\log k}} \frac{L(x)^{\gd}}{ x^{(1+\ga) \gd}} \frac{L_{2}(k-j)}{(k-j)^{(2+\ga)\delta -2}} \leq  \frac{L_{23}(k) \, k^{\gd}}{ k^{\tfrac12 ( (1+\ga)\gd-1)}} = \frac{L_{23}(k)}{ k^{\tfrac12((\ga-1)\gd-1)}} \, .
\end{align*}
Hence, both $S_{6a}$ and $S_{6b}$ are arbitrarily small for $k$ large,  by the condition \eqref{eq:condga}.

\smallskip
By \eqref{eq:ineqA}, and since provided that $k$ is large enough we have $i+\sqrt{i\log i}\leq 3k/4$ for $i\leq k/2$, we obtain
\begin{equation}
S_{6c} \le C_9 \sum_{i=0}^{ k/2 } \sqrt{i\log i}  \frac{L_{2}(k)}{k^{(2+\ga)\delta -2}} \le \frac{L_{24}(k)}{ k^{(2+\ga)\delta - 7/2 } } \,,
\end{equation}
which is arbitrarily small for $k$ large.

For the term $S_{6d}$, since for every $j\in \{k/2+2,\ldots, k-1\}$ there are at most $C_{10} \sqrt{k\log k}$ corresponding terms in the sum over $i$, we have
\begin{align}
\label{eq:S6d}
S_{6d} & \le C_{11} \sqrt{k\log k} \maxtwo{ k/2  \le i \le k}{i  \le j \le  i+ \sqrt{i \log i} } \cA_{i,j}  .
\end{align}
Then we use Proposition \ref{prop:chgmeasure}, and more precisely \eqref{boundA2}, to get that
\begin{equation}
S_{6d} \le L_{25}(k) k^{- \frac12 (\gd (\ga-1)-1)}
\end{equation}
In view of the condition \eqref{eq:condga}, $S_{6d}$ can be made arbitrarily small for $k$ large.
This completes the proof of  \eqref{eq:rel} in the case $\ga>2$.

%From \cite{cf:Hun}, we have that
%\begin{equation}
%\lim_{N \to \infty} \frac{1}{N} \sum_{n=1}^N \sum_{m=1}^N \mathbf{1}_{(n,m) \in \tau} \stackrel{N \to \infty}{\longrightarrow} \frac{1}{\max \lbrace \bE(\tau_1^{(1)}),\bE(\tau_1^{(2)}) \rbrace} \,.
%\end{equation}
%choosing $a$ small enough, the right-hand side of \eqref{eq:S6} can be made small.

\subsection{Conclusion of the proof of Theorem \ref{thm:rel} in the case $\ga\in(1,2]$}
\label{sec:ga1}

Fix $\gep >0$ small and let $0 < \delta <1$ such that
\begin{equation}
\label{eq:cond1}
\delta \left[ (2+\ga) + (1-\gep^2)/\ga \right] > 4 -  \gep^2 \,.
\end{equation}
which implies in particular that $\delta (2+\ga) > 3$.
We also assume that
\begin{equation}
\label{eq:cond2}
 \gd\ga >1\, , \quad \gd \,\big( 1+(1+\ga) \gep^4 \big) > 1+\gep^4 \, , \quad \text{and}\quad \gd>(1+\gep^4)/(1+\gep^2/2) \, .
\end{equation}

%We take $\beta \le \beta_0$ and 
%\begin{equation}
%h = h_c^{a}(\beta) + \Delta :=  h_c^{a}(\beta)  +  a \beta^{\frac{2 \ga }{\ga -1}(1+\gep)} \,.
%\end{equation}
%
%Now for fixed $\gep$ and $\delta$ verifying \eqref{eq:cond1} and \eqref{eq:cond2}, we choose $k$ to be of the order of the correlation lenght of the annealed system as in the previous section: then by Theorem~\ref{th:beta0} we have $\tf_1(0,\Delta) \sim L_{\ga,1}(\Delta) \Delta$ and choose 
%\begin{equation}
%\label{eq:k2}
%k := \frac{1}{\Delta} = \frac{1}{a \beta^{\frac{2 \ga }{\ga -1}(1+\gep)} } \,.
%\end{equation}
%
%Thanks to Proposition~\ref{th:rho}, to prove Theorem~\ref{th:2ga3}, it suffices then to show that \eqref{eq:rho} is $o(1)$ for $k$ large. 

\medskip
Let us start with showing that $\rho_1$ is small: we split the sum in \eqref{eq:rho1} to
\begin{equation}
T_1 + T_2 := \bigg( \sum_{i,j=0}^{ k^{1- \gep^2} }  + \sum_{i,j=  k^{1- \gep^2}  + 1}^{k-1}  \bigg) \frac{L_1( 2 k -i-j) }{(2k -i-j)^{(2+\ga)\delta -2}} A_{i,j} \,,
\end{equation}
and
\begin{equation}
T_3 + T_4 := \bigg( \sum_{i=0}^{  k^{1- \gep^2} } \sum_{j=  k^{1- \gep^2}  + 1}^{k-1} + \sum_{i=  k^{1- \gep^2}  + 1}^{k-1} \sum_{j=0}^{ k^{1- \gep^2} }  \bigg) \frac{L_1( 2 k -i-j) }{(2k -i-j)^{(2+\ga)\delta -2}} A_{i,j} \,.
\end{equation}

%Observe that by Jensen's inequality we have
%\begin{equation}
%\label{eq:A2}
% A_{i,j} \le ( \bbE Z_{i,j,\omega} )^\delta \le e^\delta { \bP((i,j) \in \tau)}^\delta \,.
%\end{equation}
%As a consequence of Theorem~\ref{th:reD} and Theorem~\ref{th:reoff}, there exists $C_7$ such that (assuming $i <j$)
%\begin{equation}
%\label{eq:pa}
%\bP((i,j) \in \tau) \le \frac{C_7}{a_i} \,.
%\end{equation}

\medskip
For $\ga\leq 2$, we know that there exists a slowly varying function $\psi(\cdot)$ such that $a_i = \psi(i) i^{1/\ga}$. For $T_1$, using \eqref{eq:ineqA} and Theorem \ref{th:reD}, we get
\begin{equation}
T_1 \le \frac{L_{26}(k)}{k^{(2+\ga) \delta - 2}}  \sum_{i,j =0}^{k^{1-\gep^2}}   \frac{1}{(a_{\min(i,j)})^\gd}  \leq \frac{L_{27}(k)}{ k^{(1-\gep^2) ( \delta/\ga - 2) +(2+\ga) \delta - 2}} \,,
\end{equation}
and from the condition \eqref{eq:cond1}, $T_1$ can be made small for $k$ large.

For $T_2$, since $(2+\ga)\gd-2 \in(1,2)$, we have
\begin{equation}
T_2 \le L_{28}(k) k^{-(2+\ga)\gd + 4} \max_{ k^{1-\gep^2}  \le i,j \le k} \cA_{i,j} < \frac{L_{29}(k) }{ k^{(1- \gep^2)\delta/\ga +(2+\ga)\gd-4}} \,,
\end{equation}
where for the last inequality we used \eqref{eq:ineqA} and Theorem \ref{th:reD}.
Then $T_2$ is small for $k$ large thanks to \eqref{eq:cond1}.

For $T_3$ (which is equal to $T_4$), since for the range of $i,j$ considered we have  $2k-i-j \geq k/2$, we get using \eqref{eq:ineqA} and Theorem \ref{th:reD}
\begin{equation}
T_3 \le \sum_{i=0}^{k^{1-\gep^2}} \frac{1}{(a_i)^{\gd}}\frac{L_{30}(k)}{k^{(2+\ga) \delta - 3}} \leq  \frac{L_{31}(k)}{ k^{(2+\ga) \delta - 3 +(1-\gep^2)(\gd/\ga-1)}}\,,
\end{equation}
which can be made small by taking $k$ large, thanks to \eqref{eq:cond1}. In the end, we get that $\rho_1$ is bounded from above by a small constant for $k$ large.

\medskip

As far as $\rho_2$ is concerned,  we split the right-hand side of \eqref{eq:rho2} to 
\begin{equation}
T_5 + T_6 = \sum_{i=0}^{k-1}  \sum_{j=i+1}^{k-1} \cA_{i,j} \frac{L_2(k-j)}{(k-j)^{(2+\ga)\delta -2}} + \sum_{i=0}^{k-1} \sum_{j=0}^{i} \cA_{i,j}  \frac{L_2(k-j)}{(k-j)^{(2+\ga)\delta -2}} \,.
\end{equation}

Recall the definition of $\ell_i$ in Proposition \ref{prop:chgmeasure}, and define $\bar \ell_i=i^{(1+\gep^4)/\ga} \gg \ell_i$. We split $T_5$ as
\begin{equation}
T_{5a} + T_{5b} = \bigg( \sum_{i=0}^{ k/2} +  \sum_{i= k/2 + 1}^{k-1} \bigg) \sum_{j=(i + \bar \ell_i) \wedge (k-1)}^{k-1} \cA_{i,j} \frac{L_2(k-j)}{(k-j)^{(2+\ga)\delta -2}} \,.
\end{equation}
and 
\begin{equation}
T_{5c} + T_{5d} = \bigg( \sum_{i=0}^{ k^{1- \gep^ 2} } +  \sum_{i= k^{1- \gep^ 2} +1}^{k-1} \bigg) \sum_{j=i + 1}^{(i + \bar \ell_i ) \wedge (k-1)} \cA_{i,j} \frac{L_2(k-j)}{(k-j)^{(2+\ga)\delta -2}} \,.
\end{equation}

From \eqref{eq:ineqA} and Theorem \ref{th:reD}, we get that (using that $(2+\ga)\gd -2>1$ for the second line)
\begin{equation}
\begin{split}
T_{5a} & \le \sum_{i=0}^{ k/2 } \bigg( \sum_{j=i +   \bar \ell_i }^{ 3k/4}   + \sum_{j= 3k/4+1 }^{k-1} \bigg) \frac{C_{}12 i^{\gd}L(j-i)^{\gd}}{(j-i)^{(1+\ga) \delta  }} \frac{L_2(k-j)}{(k-j)^{(2+\ga)\delta -2}} \\
& \le  \frac{L_{32}(k)}{k^{(2+\ga)\gd -2}} \sum_{i=1}^{k/2} i^{\gd} \frac{L_{33}(\bar \ell_i)}{(\bar \ell_i)^{(1+\ga)\gd -1 }} +L_{34}(k) k^{\gd+1} \frac{1}{k^{(1+ \ga)\gd  }}  \\
&\le   \frac{L_{35}(k) k^{1+\gd}}{ k^{(2+\ga)\gd -2 + \tfrac{1+\gep^4}{\ga}((1+\ga)\gd-1) }}+ \frac{L_{34}(k)}{k^{\delta \ga -1 }} \leq   \frac{L_{35}(k) }{ k^{\tfrac{1+\gep^4}{\ga}((1+\ga)\gd-1) -\gd }} + \frac{L_{34}(k)}{k^{\delta \ga -1 }} \,,
\end{split}
\end{equation}
and also  (using also here that $(2+\ga)\gd -2>1$ for the third line)
\begin{equation}
\begin{split}
T_{5b} & \le \sum_{i= k/2+1}^{k-1} \sum_{j=(i+  c \bar \ell_k ) \wedge (k-1)}^{k-1} \frac{C_{12} i^{\gd}L(j-i)^{\gd}}{(j-i)^{(1+\ga) \delta }} \frac{L_2(k-j)}{(k-j)^{(2+\ga)\delta -2}} \\
&\le C_{12} k^{\gd} \sum_{j=k/2}^{k-1} \frac{L_2(k-j)}{(k-j)^{(2+\ga)\delta -2}} \sum_{x\geq  c  \bar\ell_k}\frac{L(x)}{ x^{(1+\ga)\gd}} \\
& \le  L_{36}(k) k^{\gd}  (\bar\ell_k)^{1-(1+\ga)\gd} =\frac{L_{36}(k)}{ k^{\tfrac{1+\gep^4}{\ga}((1+\ga)\gd -1) - \gd }} \,.
\end{split}
\end{equation}
Then the condition \eqref{eq:cond2} guarantees that $T_{5a}$, $T_{5b}$ can be made arbitrarily small by choosing $k$ large.

Using \eqref{eq:ineqA} and Theorem \ref{th:reD}, we get 
\begin{equation}
\begin{split}
T_{5c} & \le \sum_{i=0}^{  k^{1- \gep^ 2} } \sum_{j=i + 1}^{(i+  \bar\ell_i ) \wedge (k-1)}  \frac{C_{13}}{a_i^{\gd}} \frac{L_2(k-j)}{(k-j)^{(2+\ga)\delta -2}} \le   \frac{L_{37}(k)}{k^{(2+\ga ) \delta -2}} \sum_{i=0}^{  k^{1- \gep^ 2} }  L_{38}(i) i^{(1+\gep^4)/\ga -\gd/\ga} \\
& \le   \frac{L_{39}(k)}{k^{(2+\ga ) \delta -2}} k^{(1-\gep^2) ( (1+\gep^4-\gd)/\ga  +1 )  }\, .
\end{split}
\end{equation}
Again, \eqref{eq:cond1} insures that $T_{5c}$ can be made arbitrarily small by choosing $k$ large.

\smallskip
Finally, it remains to bound $T_{5d}$. As for \eqref{eq:S6d}, there are at most $\bar\ell_k$ terms in the sum over $i$ (and $(2+\ga)\gd-2>1$), so that
\begin{equation}
T_{5d} \le C_{14} \, \bar\ell_{k} \maxtwo{  k^{1- \gep^2}  \le i < k-1}{i  \le j \le i+ \ell_i} \cA_{i,j} \,.
\end{equation}
Then we use Proposition \ref{prop:chgmeasure}, and more precisely \eqref{boundA1}, to get that
\begin{equation}
T_{5d} \le L_{40}(k) k^{(1+\gep^4)/\ga} k^{-\gd (1+\gep^2/2)/ \ga} , 
\end{equation}
which can be made arbitrarily small by choosing $k$ large, because of condition \eqref{eq:cond2}.

%By Section~\ref{sec:com}, for $\lambda= 1/\sqrt{\# J_{i,j}}$ we see that
%\begin{equation}
%\label{eq:A22}
%\cA_{i,j} \le { \left[ \bbE_{i,j,\lambda} [Z_{i,j,\omega}] \right] }^{\delta} \exp(c_1 \delta/(1-\delta)) \,.
%\end{equation}
%
%Applying Proposition~\ref{th:12ez}, we get
%\begin{equation}
%\begin{split}
%T_{5d} & \le   C_{12} \, a_k  \, \max_{ \lfloor k^{1- \gep^2} \rfloor +1 \le i < k-1} \left[    ?? + \frac{L_{17}(i) \psi_4(i)}{i^{\delta \left[ \ga + \frac{\gep}{2} ( \frac{\ga-1 }{ \ga}) \right]}}   \right] \\
%& \le ?? + \frac{L_{21}(k) \psi_7(k)}{k^{\delta (1- \gep^2) \left[ \ga + \frac{\gep}{2} ( \frac{\ga-1 }{ \ga}) \right]}}  \\
%\end{split}
%\end{equation}
%and $T_{5d}$ can be made small by \eqref{eq:cond1}, \eqref{eq:cond2} and $k$ large.

For $T_6$, we have
\begin{equation}
T_6 \le  \sum_{j=0}^{k-1} \sum_{i=j}^{k-1} \cA_{i,j}  \frac{L_2(k-i)}{(k-i)^{(2+\ga)\delta -2}} \,.
\end{equation}
By following the same procedure adopted for $T_{5}$,  $T_{6}$ is bounded above by a small term when $k$ is large. 
The proof of \eqref{eq:rel} in the case $\ga\in(1,2]$ is therefore complete and, with it, also the proof
of the lower bound part of Theorem~\ref{thm:rel}.

\begin{appendix}

\section{Bivariate renewal theory, important estimates}
\label{app:bivrenewal}

We present here some results on the bivariate renewal process $\tau$ defined in Section~\ref{sec:model}, and in particular Proposition \ref{prop:transience} which gives some conditions on the transience/recurrence of the intersection renewal $\sigma=\tau\cap\tau'$. Recall the notations of Section \ref{sec:defan} for the recentering sequence $(b_n)_{n\geq 0}$ and for the scaling sequence $(a_n)_{n\geq 0}$.

\subsection{Local large deviations and a useful Lemma}

We first present some local large deviation estimate,  which is used in the proof of Lemma \ref{lem:homogen}.
\begin{theorem}[Theorem~2.4 in \cite{cf:B}]
\label{thm:locallargedev}
Assume that $\mu<+\infty$. We have that there exists a constant $C>0$ such that uniformly for $n$ such that $n - \mu k \geq a_k \wedge (C\sqrt{k\log k}),$
\[\bP(\tau_k =(n,n)) \leq  C\,  k \, K(n-\mu k)  \, .\]
\end{theorem}

We give another useful lemma, that controls the number of renewals in $(0,N]\times(0,M]$, in the case $\ga\in(0,1)$.
\begin{lemma}
\label{th:limp}
Assume $\ga\in(0,1)$.
Given $\delta>0$ there exists $\gep>0$ such that for $N$ sufficiently large and $M \sim \gamma N$ we have
\begin{equation}
 \bP \Big( \vert \tau \cap (0,N] \times (0,M] \vert \ge \gep N^{\ga}/L(N) \Big) \ge 1- \delta \,.
\end{equation}
\end{lemma}

\begin{proof}
Set $n = n(\gep,N) = \gep N^{\ga}/L(N)$ and $B_{N,M} := (0,N]\times (0,M]$. We want to prove
\begin{equation}
\bP \left( \tau_n \notin B_{N,M} \right) \le \delta \,.
\end{equation}
Let us define $\tilde \tau_n = (\tilde \tau_n^{(1)} ,\tilde \tau_n^{(2)})$ with
\begin{equation}
\tilde{\tau}^{(1)}_n := \sum_{i=1}^n (\tau_i^{(1)} - \tau_{i-1}^{(1)}) \mathbf{1}_{ \lbrace \tau_i^{(1)} - \tau_{i-1}^{(1)} \le N \rbrace} \, ,\quad 
\tilde{\tau}^{(2)}_n := \sum_{i=1}^n (\tau_i^{(2)} - \tau_{i-1}^{(2)}) \mathbf{1}_{ \lbrace \tau_i^{(2)} - \tau_{i-1}^{(2)} \le M \rbrace} \,.
\end{equation}

Then we have 
\begin{equation}
\label{eq:tautilde}
\bP \left( \tau_n \notin B_{N,M} \right) \le \bP \left( \tilde{\tau_n} \notin B_{N,M} \right) + \bP \left( \exists i \le n ; (\tau_i - \tau_{i-1}) \notin B_{N,M} \right) \,.
\end{equation}
Note that the marginals $\tau^{(1)}_n$ and $\tau^{(2)}_n$ have the same distribution: as $N \to \infty$ we have
\begin{equation}
\bP( \tau^{(1)}_1= N) = \bP( \tau^{(2)}_1= N) \sim \frac{1}{(1+\ga)} L(N) N^{-(1+ \ga)} \,.
\end{equation}

Therefore, we can bound the second term in \eqref{eq:tautilde} by
\begin{equation}
\begin{split}
 n \left( \bP ( \tau^{(1)}_1 > N  )+ \bP ( \tau^{(2)}_1 > M ) \right)
& \le n \left(  C_{\ga} L(N) N^{-\ga} +C_{\ga} L(M) M^{-\ga}\right) \le \gep  C_{\ga,\gamma}   \,,
\end{split}
\end{equation}
which is smaller than $\gd/2$ if $\gep \leq  \gd/(2 C_{\ga,\gamma})$.

The first term in \eqref{eq:tautilde} is bounded by $ \bP(\tilde{\tau}^{(1)}_n > N) + \bP(\tilde{\tau}^{(2)}_n > M)$. Observe that for every choice of $\lambda_1 >0$ we have
\begin{equation}
\bP ( \tilde{\tau}^{(1)}_n > N ) \le e^{ -\lambda_1 N}  \bE [e^{\lambda_1 \tilde{\tau}^{(1)}_n} ] 
\le e^{n \log \bE [e^{\lambda_1 \tilde{\tau}^{(1)}_1} ] - \lambda_1 N  }\,.
\label{eq:expobound}
\end{equation}
Using the fact that $\tilde \tau_1^{(1)}  = \tau_1^{(1)} \ind_{\{\tau_1^{(1)}\leq N \}}\leq N$, we get that for any $s \ge 1$
\begin{equation}
\bE [(\tilde{\tau}^{(1)}_1)^s   ] \le N^{s-1}
\bE [\tau^{(1)}_1  \mathbf{1}_{\{\tau^{(1)}_1 \leq N\}} ] \le c_{\ga} L(N) N^{- \ga } N^s \,,
\end{equation}
where we used that $\bP(\tau_1^{(1)} =N) \leq cst. L(N) N^{-(1+\ga)}$ with $\ga\in(0,1)$ to estimate the second expectation. 
In the end, expanding the exponential and using the above bound, we get
\begin{equation}
\log \bE \left[e^{\lambda_1 \tilde{\tau}^{(1)}_1}  \right] \le \log \left( 1 + c_{\ga} L(N) N^{- \ga } (e^{\lambda_1 N} -1) \right) \le  c_{\ga} L(N) N^{- \ga } e^{\lambda_1 N} \,.
\end{equation}

We pick $k_0$ such that $\gd^{k_0-1}\leq e^{-1}/4$, and choose $\lambda_1 = k_0  N^{-1} \log(1/ \delta)$, then with the definition of $n=\gep L(N)^{-1} N^{\ga}$, we get
$n \log \bE [e^{\lambda_1 \tilde{\tau}^{(1)}_1} ] \le  c_{\ga} \gep \delta^{-k_0} $, and choosing $\gep\leq c_{\ga}^{-1} \gd^{k_0}$ we get from \eqref{eq:expobound}
\begin{equation}
\bP(\tilde{\tau}^{(1)}_n > N) \le  e \gd^{k_0} \leq \gd/4  \,.
\end{equation}

Using the same reasoning and choosing $\lambda_2 =  k_0 M^{-1} \log (1/\delta)$ , we have that if $\gep\leq  c_{\ga,\gamma}^{-1} \gd^{k_0}$ (for some constant $c_{\ga,\gamma}$), 
\begin{equation}
\bP(\tilde{\tau}^{(2)}_n > M) \le e \delta^{k_0} \leq \gd/4 \,.
\end{equation}

The proof is therefore complete by taking $\gep = \min \lbrace \delta / (2 C_{\ga,\gamma}) , c_{\ga}^{-1} \delta^{k_0} , c_{\ga,\gamma}^{-1} \delta^{k_0} \rbrace$.
\end{proof}

\subsection{Renewal theorems, and the intersection of two independent copies}

The goal of this section is to estimate the mean overlap of two copies $\tau$ and $\tau'$ in the region $(0,N]\times (0,M]$. We leave aside the case $\ga=1$ which is  more technical (in particular if $\mu=+\infty$): we refer to Remark \ref{rem:alpha=1} for more comments on this case.
We define
\begin{equation}
\label{eq:UNM}
U_{N,M}\,:= \,\bE\left[ \left| \sigma\cap \left([0,N]\times [0,M]\right) \right|\right] \,=\,\sum_{n=0}^{N} \sum_{m=0}^{M} \bP((n,m)\in\tau )^2 \, ,
\end{equation}
and for any $\gl>0$
\begin{equation}
\hat U (\gl) := \sum_{n,m=0}^{+\infty} e^{-\gl (n+m)}  \bP((n,m)\in\tau )^2 \, .
\end{equation}
%Recall that $a_n = \psi(n) n^{1/(\ga\wedge 2)}$ is the normalizing sequence for $\tau_n$ (cf. \eqref{def:an}-\eqref{an}). 

\begin{proposition}
\label{prop:transience}
If $\ga<1$, then 
$\sup_{N,M\in\bbN} \  U_{N,M}  <+\infty.$

If $\ga>1$, then set $\rho:=1-\min(\ga,2)^{-1} \in [0,1/2]$. We have
\begin{equation}
\label{suman}
\sum_{n=1}^N \frac{1}{a_n} \sim   \varphi(N) N^{\rho} \to +\infty \qquad \text{as } N\to\infty,
\end{equation}
for some slowly varying function $ \varphi (\cdot)$.
Moreover, 
\begin{equation}
\label{eq:asympUn}
\begin{aligned}
U_{N,N}&\sim 2 c_{\ga}  \varphi(N) N^{\rho}  &\quad \text{as } N\to\infty, \\
  \hat U(\gl) & \sim \frac{2^{1-\rho} c_{\ga}}{\Gamma(1+\rho)}  \varphi(1/\gl) \gl^{-\rho} &\quad \text{as } \gl\downarrow 0, 
\end{aligned}
\end{equation}
with $c_{\ga} = \int_{0}^{\infty} c_{\ga}(t)^2 \mathrm{d}t$, $c_{\ga}(t)$ being the constant appearing in Theorem \ref{th:reD}.

As a consequence, 
$\sigma=\tau\cap\tau'$ is terminating if $\ga<1$, and persistent if $\ga> 1$.
\end{proposition}

This proposition is based on renewal theorems (see Theorems~\ref{th:re01}-\ref{th:reD} below), that can be found in \cite{cf:B} (in a morel general setting), giving sharp asymptotics along the favorite direction, and general upper bounds away from it.
%We collect the results giving estimates of the quantity $\bP((n,n+r)\in\tau)$ for any $r\geq 0$ (this is enough by symmetry) in the two following theorems, dealing with the cases $\ga\in(0,1)$ and $\ga > 1$.
The case $\ga=1$  can also be found in \cite{cf:B} but we do not include it here, see Remark \ref{rem:alpha=1}.

\begin{theorem}
\label{th:re01}
If  $\ga \in (0,1)$, then for $n\to+\infty $ and $r$ such that $r/n \to t \in \bbR_+$, we have
\begin{equation}
\label{eq:D01}
\bP \left( (n,n+r) \in \tau \right) \stackrel{n \to \infty}\sim C_{\ga}(t) L(n)^{-1} n^{-(2-\ga)} \,,
\end{equation}
with $C_{\ga}(t) := \ga \int_{0}^{+ \infty} x^{1- \ga} g_\ga (x , (1+t) x ) \, \mathrm{d}x$.
Moreover, for any $\gd>0$ there is a constant $C_{\gd}>0$ such that for any $r\geq n$,
\begin{equation}
\bP\big( (n,n+r)\in\tau \big) \leq C_{\gd} L(n)^{-1} n^{-(2-\ga)} \times    \Big(\frac rn \Big)^{-(1+\ga) +\gd} \, .
\end{equation}
\end{theorem}

\begin{theorem}
\label{th:reD}
If $\ga> 1$, for  $n \to \infty$ and $r$ such that $r/a_n \to t \in \bbR_+$, we have that
\begin{equation}
\bP \left( (n,n+r) \in \tau \right) \stackrel{n \to \infty}\sim c_{\ga}(t)\, \frac{1}{a_n} \,,
\end{equation}
where $c_{\ga}(t) = \mu_\ga \int_{- \infty}^{+ \infty} g_\ga (x,x+\mu_\ga t ) \mathrm{d} x$ with $\mu_\ga := \mu^{1/\min(\ga,2)}$.
Moreover, for any $\gd>0$ there exists a constant $C_{\gd}>0$ such that for any $r\geq a_n$,
\begin{equation}
\label{eq:offdiag}
\bP\big( (n, n+r) \in \tau \big) \leq \frac{C}{a_n}  \Big( \frac{r}{a_n} \Big)^{-(1+\ga) +\gd}\, .
\end{equation}
\end{theorem}

Theorems~\ref{th:re01} and \ref{th:reD} are extracted from \cite{cf:B}, Theorems~3.1, 4.1 and Theorems~3.3, 4.2 respectively,: we refer to Equations (3.4), (4.2) and (3.7), (4.5) in \cite{cf:B} respectively, for a statement  in the symmetric setting we are considering here.

\begin{proof}[Proof of Proposition \ref{prop:transience}]

{\bf Case $\ga<1$.}
Notice that by symmetry, for $M\geq N$,
\[U_{N,M} \leq U_{M,M} \leq 2 \sum_{n=1}^M \sum_{r=0}^{M-n} \bP((n,n+r) \in\tau)^2\, .\]
We therefore need to control the last sum. Let us denote
\[W_n:=\sum_{r\geq 0} \bP((n,n+r) \in\tau)^2 \, .\]
Using Theorem \ref{th:re01} (and properties of slowly varying functions), we get that there is a constant $c$ such that for all $n\geq 1$
\begin{align*}
W_n \leq c\sum_{r=1}^n  L(n)^{-2} n^{-2(2-\ga)} + c\sum_{r\geq n} L(n)^{-2} n^{-2(2-\ga)}  (r/n)^{-2}
 \leq C' L(n)^{-2} n^{ 2\ga -3},
\end{align*}
where we used that $\sum_{r\geq n} (r/n)^{-2} \sim n \int_1^{\infty} x^{-2}$ as $n\to\infty$.
Therefore, since $\ga<1$, we get
\[\sup_{N,M} U_{N,M}  \leq \sum_{n=1}^{+\infty} W_n <+\infty \, .\]

\medskip

{\bf Case $\ga > 1$.}
First of all, it is immediate that $\sum_{n=1}^N 1/a_n$ diverges as a slowly varying function with exponent $\rho$, since $a_n \sim \psi(n) n^{\frac{1}{\min(\ga,2)}}$, see \eqref{an}: it directly gives \eqref{suman}.
%If $\ga=1,\mu<+\infty$, we show that $\psi(n)=a_n/n\to 0$.
%Therefore, $D_N:=\sum_{n=1}^N (a_n)^{-1}=\sum_{n=1}^n \psi(n)^{-1} n^{-1}$,  diverges to $+\infty$ as a slowly varying function, and $D_N/\log N \to+\infty$, giving also \eqref{suman} in that case.

%We have that $L(n)n^{-1}$ is asymptotically decreasing, there exists a slowly varying function $\mathbf{L}$ and a constant $c$ such that 
%$ \mathbf{L}(n)n^{-1} \leq L(n)n^{-1} \leq c\mathbf{L}(n) n^{-1}$, and with  $\mathbf{L}(n) n^{-1}$  non-increasing in $n$.
%Hence, because $\sum_{n\geq 1} L(n) n^{-1} <+\infty $ (since $\mu<+\infty$), we have that $\sum_{n\geq 1} \mathbf{L}(n)n^{-1} <+\infty $, and because $\mathbf{L}(n)n^{-1}$ is non-increasing we get that 
%\[
%k\times \frac{\mathbf{L}(2k)}{2k} \leq \sum_{n =k}^{2k} \mathbf{L}(n)n^{-1} \to 0, \quad \text{as } k\to\infty.
%\]
%We therefore get that $\mathbf{L}(n)\to 0$ as $n\to\infty$, which yields that $L(n)\to 0$, and $\psi(n)\to 0$ because of \eqref{def:an}.

We now prove \eqref{eq:asympUn}.
We fix $\gep>0$, and denote, in complement to the definition of $W_n$ above
\[W_n^{(\gep)} := \sum_{r = 0}^{ \lfloor \frac1\gep a_n \rfloor} \bP((n,n+r) \in \tau)^2\, . \]

As a preliminary, we show that there exists some $n_\gep$ such that, provided that $n\geq n_{\gep}$
\begin{equation}
\label{eq:sumk}
(1-\gep) \frac{c_{\ga}}{a_n} \leq W_n^{(\gep)}  \leq  W_n  \leq (1+\gep) \frac{c_{\ga}}{a_n}.
\end{equation}
Note that we also have that $W_n \leq \sum_{r = 0}^{+\infty} \bP((n,n+r) \in \tau) \leq 1$ for any $n$.

To prove \eqref{eq:sumk}, we use Theorem \ref{th:reD} to get that uniformly for $0\leq r\leq \tfrac 1\gep a_n$, we have $\bP((n,n+r) \in \tau)^2 \sim (a_n)^{-2} c_{\ga}(r/a_n)^2$ as $n\to\infty$. Hence, provided that $n$ is large enough, we get that 
\[ a_n W_n^{(\gep)} \geq (1-\gep^2) \frac{1}{a_n} \sum_{r=0}^{\lfloor\frac1\gep a_n \rfloor} c_{\ga}(r/a_n)^2 \geq (1-2\gep^2) \int_{0}^{1/\gep} c_{\ga}(t)^2 \mathrm{d} t\, ,  \]
the last inequality holding by Riemann-sum approximation. 
Note that a similar upper bound, with $1-2\gep^2$ replaced with $1+2\gep^2$ holds.
Now, thanks to \eqref{eq:offdiag} (and since $1+\ga-\gd \geq 3/2$), there exists a constant $c>0$ such that
\[a_n(W_n-W_n^{(\gep)}) = a_n \sum_{r > \frac1\gep  a_n}^{+\infty} \bP((n,n+k) \in \tau)^2 \leq c\, \frac{1}{a_n} \sum_{r>\frac1\gep  a_n} \Big(\frac{r}{a_n} \Big)^{-3} \leq c' \gep^2 ,\]
where the last inequality also comes from a Riemann-sum approximation.
Finally, note that $c_\ga -\int_{0}^{1/\gep} c_{\ga}(t)^2 \mathrm{d} t $ is positive, and thanks to \eqref{eq:offdiag} smaller than $\int_{1/\gep} c t^{-3} \mathrm{d} t \leq c'' \gep^2$.
In the end, we get that, provided that $n$ is large enough,
\begin{equation}
(1-2\gep^2) (c_{\ga} - c'' \gep^2) \leq a_n W_n^{(\gep)} \leq  a_n W_n  \leq (1+2\gep^2) c_{\ga} + c' \gep^2,
\end{equation}
which gives \eqref{eq:sumk} provided that $\gep$ has been fixed small enough.

\smallskip
We are now ready to estimate $U_{N,N}$. We write
\[U_{N,N} = 2\sum_{ n=0}^{N}  \sum_{r=0}^{N-n} \bP((n,n+r) \in\tau)^2   - \sum_{n=0}^N \bP((n,n)\in\tau)^2. \]
The second sum is negligible compared to $\sum_{n=1}^N \tfrac{1}{a_n}\sim  \varphi(N) N^{\rho}$, since $\bP((n,n)\in\tau)^2 \sim c (a_n)^{-2}$, with $a_n \to +\infty$. We therefore focus on the first sum.

An upper bound is simply
\[ \sum_{n=0}^N \sum_{r=0}^{N-n} \bP((n,n+r) \in\tau)^2 \leq \sum_{n=0}^N W_n ,\]
and since we have that $W_n \sim c_{\ga}/a_n$ together with \eqref{suman}, we get that for $n$ large enough
\[\sum_{n=0}^N \sum_{r=0}^{N-n} \bP((n,n+r) \in\tau)^2 \leq (1+2\gep) c_{\ga} \sum_{n=1}^N \frac{1}{a_n} \leq (1+3\gep)  \varphi(N) N^{\rho}\, .\]
For a lower bound, because $a_N\leq \gep N$ provided that $N$ is large enough, we have
\[\sum_{ n=0}^{N}  \sum_{k=0}^{N-n} \bP((n,n+k) \in\tau)^2  \geq \sum_{n=n_{\gep}}^{(1-\gep) N} W_n^{(\gep)} \geq (1-2\gep) c_{\ga} \sum_{n=1}^{(1-\gep)N} \frac{1}{a_n}  \geq (1- c\gep)c_{\ga}  \varphi(N) N^{\rho} ,\]
 where we used the lower bound \eqref{eq:sumk} valid for $n$ large enough, together with \eqref{suman} for the last inequality.

\smallskip
We now turn to estimating $\hat U (\gl)$ as $\gl \downarrow 0$.
By symmetry, we can write that
\[\hat U(\gl)  = 2\sum_{n=0}^{+\infty} e^{-2 \gl n}\sum_{r=0}^{+\infty}  e^{-\gl r} \bP((n,n+r)\in\tau)^2 - \sum_{n=0}^{+\infty} e^{-2\gl n} \bP((n,n)\in\tau)^2 .\]
The second term is negligible compared to $ \varphi(1/\gl) \gl^{-\rho}$ as $\gl \downarrow 0$, since $\sum_{n=0}^N \bP((n,n)\in\tau)^2$ is negligible compared to $ \varphi(N)N^{\rho}$, by standard properties of Laplace transforms, and we again focus on the first term.

First of all, an upper bound is
\[\sum_{n=0}^{+\infty} e^{-2 \gl n}\sum_{r=0}^{+\infty}  e^{-\gl r} \bP((n,n+r)\in\tau)^2 \leq \sum_{n=0}^{+\infty} e^{-2\gl n} W_n \,. \]
Since $\sum_{n=0}^{N}  W_n \sim c_{\ga}  \varphi(N)N^{\rho}$, we get by standard properties of Laplace transforms (see Corollary~1.7.3 in \cite{cf:BGT}) that
\[\sum_{n=0}^{+\infty} e^{-2\gl n} W_n \sim  \frac{c_\ga}{\Gamma(1+\rho )}   \varphi(1/2\gl) (2\gl)^{-\rho}  \qquad  \text{as } \gl \downarrow 0 \, .\]

For a lower bound, we get that 
\[\sum_{n=0}^{+\infty} e^{-2 \gl n}\sum_{r=0}^{+\infty}  e^{-\gl r} \bP((n,n+r)\in\tau)^2  \geq \sum_{n=0}^{+\infty} e^{-2\gl n}  e^{-\gl a_n /\gep} W_n^{(\gep)} \, .\]
Now, we use  that there is some $n_{\gep}$ such that for $n\geq n_{\gep}$ we have that $W_n^{(\gep)}\geq (1-\gep) c_{\ga} /a_n$ (see \eqref{eq:sumk}), and that $a_n/\gep \leq \gep n$.
We therefore get that
\begin{align*}
\sum_{n=0}^{+\infty} e^{-2 \gl n}&\sum_{r=0}^{+\infty}  e^{-\gl r} \bP((n,n+r)\in\tau)^2 \\
& \geq (1-\gep) c_\ga \sum_{n=n_{\gep}}^{+\infty} e^{-2(1+\gep)\gl n} \frac{1}{a_n} \stackrel{\gl\to 0}{\sim}  \frac{(1-\gep) c_\ga}{\Gamma(1+\rho )}   \varphi(1/\gl) ( 2(1+\gep) \gl )^{-\rho} \, ,
\end{align*}
where we used again Corollary 1.7.3 in \cite{cf:BGT} for the last asymptotics.

By letting $\gep \downarrow 0$, we obtain matching upper and lower bound, so that \eqref{eq:asympUn} is proved.

\end{proof}

We now use Proposition \ref{prop:transience}, and in particular the estimate of the Laplace transform $\hat U(\gl)$, to obtain estimates on the tail probability of the intersection renewal $\sigma=\tau\cap \tau'$.
More precisely, we define $\underline{\sigma}:= \sigma^{(1)} + \sigma^{(2)}$ and estimate $\bP^{\otimes 2}( \underline \sigma _1 >N)$.

\begin{lemma}
\label{lem:tau1}
Assume that $\ga>1$.
Then recalling that $\rho=1-\min(\ga,2)^{-1} \in[0,1/2]$, we get that 
\begin{equation}
\bP^{\otimes 2} \left( \underline{\sigma}_1 > N \right) \, \stackrel{N \to \infty}\sim \,  \frac{2^{\rho} \sin(\pi \rho)}{ \pi \rho} \ (U_{N,N})^{-1} \stackrel{N \to \infty}\sim C_{\ga,\rho} \, \varphi(N)^{-1} N^{-\rho} \,.
\end{equation}
\end{lemma}

\begin{proof}
Recall the definition of $\hat U (\gl) = \sum_{n,m \ge 0} e^{- \lambda (n+m)} \bP^{\otimes 2} \left( (n,m) \in \sigma \right) $.
We also set, for any $\gl>0$,
\begin{equation}
\hat{K}(\lambda) \, := \, \sum_{n,m \ge 1} e^{- \lambda (n+m)} \bP^{\otimes 2} (\sigma_1 = (n,m) ) \, = \, \sum_{k \ge 2} e^{- \lambda k} \bP^{\otimes 2} ( \underline{\sigma}_1 = k ) \,.
\end{equation}
The key idea of this proof is the following identity
\begin{equation}
\label{eq:id}
\hat U(\gl) = 1+ \hat K(\gl) \hat U(\gl) \quad\Leftrightarrow \quad 
1- \hat{K}(\lambda)  \, = \, \frac{1}{\hat{U}(\lambda)} \,,
\end{equation}
which is obtained from the identity
\begin{equation}
\bP^{\otimes 2} \left( (n,m) \in \sigma \right)  = \ind_{\{n=m=0\}} + \sum_{i=1}^{n}\sum_{j=1}^{m} \bP^{\otimes 2} (\sigma_1 = (i,j) ) \bP^{\otimes 2} ( (n-i,n-j) \in\sigma ). 
\end{equation}

Now, since we know the behavior of $\hat U(\gl)$ as $\gl\downarrow0$, we get the behavior of $\hat K(\gl)$, from which we should be able to infer that of $\bP^{\otimes 2}(\underline \sigma_1 >N)$. Let us develop here how we proceed: we use Corollary 1.7.3 and Theorem 8.7.3 in \cite{cf:BGT}. We can view $\underline\sigma$ as a renewal process with inter-arrival distribution $\bP^{\otimes 2} (\underline \sigma_1=k) = \bP^{\otimes 2} ( \sigma_1^{(1)} + \sigma_1^{(2)} =k)$, and we set $u_n:= \bP^{\otimes 2} (n \in \underline \sigma)$ its renewal mass function, so we have $\hat U(\gl) = \sum_{n=0}^{\infty} e^{-\gl n} u_n$ (and \eqref{eq:id} is standard from the one-dimensional renewal equation). Now, \cite[Corollary 1.7.3]{cf:BGT} tells that since $\hat U(\gl)$ is regularly varying with exponent $-\rho$ (recall $\rho = 1-\min(\ga,2)^{-1}$), we have that $\sum_{n=0}^N u_n \sim \Gamma(1+\rho) \hat U(1/N) \sim 2^{-\rho} U_{N,N}$ (where we used \eqref{eq:asympUn}).
In turn \cite[Theorem 8.7.3]{cf:BGT} gives that 
\[\bP^{\otimes 2} (\underline\sigma _1 >N) \stackrel{N\to\infty}{\sim}  \frac{(2^{-\rho} U_{N,N})^{-1}}{\Gamma(1+\rho) \Gamma(1-\rho) },\]
and we are done.
\end{proof}

\begin{rem}\rm
\label{rem:alpha=1}
The case $\ga=1$  has been left aside, mostly to avoid.
Denote $\mu(n):=\bE[\min(\tau_1^{(1)}, n)]$ the truncated first moment of $\tau_1^{(1)}$.
It is shown in \cite[Theorem~3.4]{cf:B} (or (3.11) in the symmetric context) that along the favorite direction, for $n\to \infty$ and $r$ with $r/a_{n/\mu(n)} \to t\in \bbR_+$, we have 
\begin{equation}
\label{alpha=1}
\bP((n,n+r)\in\tau) \sim  \frac{c_1(t)}{\mu(n) a_{n/\mu(n)}} \,,
\end{equation}
with $c_1(t):= \int_{-\infty}^{+\infty} g_{\alpha}(x,(1+t)x) dx$. Notice that $n/\mu(n)$ is the typical number of steps to reach distance $n$.
Again, estimates away from the favorite direction are provided in \cite[Theorem~4.2]{cf:B} (or (4.6) in the symmetric case): for any $\gd>0$, there is a constant $C_{\gd}$ such that for any $r \ge a_{n/\mu(n)}$,
\begin{equation}
\label{alpha=1away}
\bP((n,n+r)\in\tau) \le   \frac{C_{\gd}}{\mu(n) a_{n/\mu(n)}} \, \Big( \frac{r}{\mu(n) a_{n/\mu(n)}} \Big)^{-2+\gd} \, .
\end{equation}
This shows that the main contribution to $U_{N,N}$ comes also here from the terms close to the diagonal, that is
\begin{align*}
U_{N,N} &\asymp 2\sum_{n=1}^N  \sum_{r=0}^{a_{n/\mu(n)}} \bP((n,n+r)\in\tau)^2  \asymp \sum_{n=1}^N \frac{1}{\mu(n)^2 a_{n/\mu(n)}}\,  .
\end{align*}
(We denoted $x_n\asymp y_n$ if $x_n/y_n$ is bounded away from $0$ and $+\infty$.)
Let us stress that we have $\mu(n) \sim  \mu(a_{n/\mu(n)})$ (this comes from \cite[Lemma~4.3]{cf:B17}): by a change of variable $x= n/\mu(n)$ (comparing the sum to an integral, and considering $\mu(n),a_n$ as functions of positive real numbers), we get that
\[U_{N,N}\asymp \int_{1}^{N/\mu(N)} \frac{\mathrm{d} x}{ a_{x} \mu(a_x)} \asymp  \int_1^{a_{N/\mu(N)}} \frac{\mathrm{d} u}{ u L(u) \mu(u)} \, ,\]
where we used another change of variables $u =a_x$ ($\mathrm{d}x  \sim L(u)^{-1}\mathrm{d}u$, since $n\sim a_n/L(a_n)$).
As a conclusion, we expect to have the following criterion:
\[\sigma=\tau\cap \tau' \text{ is persistent } \quad \Leftrightarrow \quad  \sum_{n\geq 1} \frac{1}{a_n \mu(a_n)} =+\infty \quad \Leftrightarrow \quad  \sum_{n\geq 1} \frac{1}{n L(n)\mu(n)} =+\infty .\]
%and one should have that $U_{N,N} \sim c \int_1^{a_{N/\mu(N)}} \frac{\mathrm{d} u}{ u L(u) \mu(u)}$ for some constant $c$.
As an example, if $L(n) =(\log n)^{\kappa}$ with $\kappa \geq -1$, then $\mu(n) \sim c_{\kappa} \max(1, (\log n)^{1+\kappa})$ and  hence $\sigma=\tau\cap \tau'$ should be persistent if and only if $\kappa>0$. 
\end{rem}

\end{appendix}


\begin{thebibliography}{99}



\bibitem{cf:A08}
K. S. Alexander, \emph{The effect of disorder on polymer depinning transitions}, Commun. Math. Phys.  {\bf 279 } (2008), 117-146.

\bibitem{cf:AB18} K. S. Alexander and  Q. Berger, \textit{Pinning of a renewal on a quenched renewal}, Electron. J. Probab., \textbf{23} (2018), no 6, 48 pp.


\bibitem{cf:AZ09}
K. S. Alexander and N. Zygouras, \emph{Quenched and annealed critical points in polymer pinning models}, Comm. Math. Phys.  {\bf 291} (2009), 659-689. 

%\bibitem{cf:AZ10} K. S. Alexander and N. Zygouras, \emph{Equality of Critical Points for Polymer Depinning Transitions with Loop Exponent One}, Ann. Appl. Prob.  {\bf 20} (2010), 356-366.

%\bibitem{cf:AZ14}
%K. S. Alexander and N. Zygouras, \emph{Path properties of the disordered pinning model in the delocalized regime}, Ann. Appl. Prob.  {\bf 24} (2014), 599-615.

%\bibitem{cf:B14} Q. Berger, \emph{Pinning model in random correlated environment: appearance of an infinite disorder regime}, J. Stat. Phys. {\bf  155} (2014),  544-570. 

%\bibitem{cf:BCPSZ14} Q. Berger, F. Caravenna, J. Poisat, R. Sun and N. Zygouras, \emph{The Critical Curve of the Random Pinning and Copolymer Models at Weak Coupling}, Commun. Math. Phys. {\bf  326} (2014),  507-530. 

\bibitem{cf:B} Q. Berger, \emph{Strong renewal theorems and local large deviations for multivariate random walks and renewals}, preprint: arXiv:1807.03575, 2018.

\bibitem{cf:B17} Q. Berger, \emph{Notes on random walks in the Cauchy domain of attraction}, preprint: arXiv:1706.07924v2 [math.PR], 2017.

\bibitem{cf:BGL}
Q. Berger, G. Giacomin and H. Lacoin,
\emph{Disorder and critical phenomena: the $\ga =0$ copolymer model},
arXiv:1712.02261

\bibitem{cf:BGK}
Q. Berger, G. Giacomin and M. Khatib,
\emph{DNA melting structures in the generalized Poland-Scheraga model},
arXiv:1703.10343


\bibitem{cf:BL12}
 Q. Berger and H. Lacoin, \emph{Sharp critical behavior for pinning models in a random correlated environment}, Stochastic Process. Appl. {\bf  122} (2012),  1397-1436.
 
\bibitem{cf:BL11}  Q. Berger and H. Lacoin, \emph{The effect of disorder on the free-energy for the Random Walk Pinning Model: smoothing of the phase transition and low temperature asymptotics}, J. Stat. Phys. {\bf 42} (2011), 322-341.

\bibitem{cf:BL16}  Q. Berger and H. Lacoin, \emph{Pinning on a defect line: characterization of marginal disorder relevance and sharp asymptotics for the critical point shift}, J. Inst. Math. Jussieu, Firstview 1-42 (2016). 

\bibitem{cf:BL16bis} Q. Berger and H. Lacoin, \emph{The high-temperature behavior of the directed polymer in dimension $1+2$}, Ann. Inst. Henri Poincar\'e Probab. Stat., to appear.

\bibitem{cf:BP15}  Q. Berger and J. Poisat, \emph{On the critical curve of the pinning and copolymer models in correlated Gaussian environment}, Electron. J. Probab. {\bf 20}, Article 71 (2015). 

\bibitem{cf:BT10} Q. Berger and  F. Toninelli, \emph{On the critical point of the Random Walk Pinning Model in dimension $d=3$}, Electron. J. Probab. {\bf 15} (2010), 654-683 
 
\bibitem{cf:BGT} N. H. Bingham, C. M. Goldie and J. L. Teugels, \textit{Regular variations}, Cambridge University Press, Cambridge, 1987.
 
 
 \bibitem{cf:BS1} 
 M. Birkner and R. Sun,  \emph{Annealed vs quenched critical points for a random walk pinning model}, 
 Ann.
Inst. H. Poincar\'e {\bf 46} (2010), 414-441. 

 \bibitem{cf:BS2} 
M. Birkner and R. Sun,  
\emph{Disorder relevance for the random walk pinning model in dimension 3}, Ann.
Inst. H. Poincar\'e {\bf 47} (2011), 259-293.


\bibitem{cf:Blake1}
R.~D.~Blake and S.~G.~Delcourt, \emph{Thermal stability of DNA}, Nucleic Acids Research {\bf 26} (1998), 3323-3332.

\bibitem{cf:Blake2}
R.~D.~Blake, J.~W.~ Bizzaro, J.~D.~Blake, G.~R.~Day, S.~G~Delcourt, J.~Knowles, K.~A.~Marx, K.A. and J.~Jr~SantaLucia, \emph{Statistical Mechanical Simulation of Polymeric DNA Melting with MELTSIM}, Bioinformatics {\bf 15} (1999), 370-375.


\bibitem{cf:BB00}
A.A. Borovkov, K.A. Borovkov, \emph{On probabilities of large deviations for random walks. I. Regularly varying distribution tails}, Theory Probab. Appl. {\bf 46} (2000), 193-213.

\bibitem{cf:BundHwa} R. Bundschuh and T. Hwa, \emph{Statistical mechanics of secondary structures formed by random RNA sequences}, Phys. Rev. E {\bf 65} (2002), 031903 (22 pages).

\bibitem{cf:CCP17}
D. Cheliotis, Y. Chino and J. Poisat,
\emph{The random pinning model with correlated disorder given by a renewal set},  arXiv:1709.06899 [math.PR].

 
\bibitem{cf:Co07}
F. Comets, \emph{Weak disorder for low dimensional polymers: the model of stable laws. Markov Process}, Markov Process. Related Fields {\bf 13} (2007), 681-696 .

\bibitem{cf:CH97}
 D. Cule and T. Hwa, \emph{Denaturation of Heterogeneous DNA}, Phys. Rev. Lett. {\bf 79 } (1997), 2375 .
 
 
 \bibitem{cf:CH13}
F. Caravenna and F. den Hollander , \emph{A general smoothing inequality for disordered polymers}, Electron. Commun.
Probab. {\bf 18 } (2013), 1-15 .

 \bibitem{cf:CTT}
 F. Caravenna, F. L. Toninelli and N. Torri , \emph{Universality for the pinning model in the weak coupling regime}, 	Ann. Probab. {\bf 45},   (2017), 2154-2209. 
 

\bibitem{cf:DGLT09}
B. Derrida, G. Giacomin, H. Lacoin and F. L. Toninelli, \emph{Fractional moment bounds and disorder relevance for pinning models}, Commun. Math. Phys. {\bf 287} (2009), 867-887.

\bibitem{cf:DR}
B. Derrida and M. Retaux, \emph{The depinning transition in presence of disorder: a toy model}, J. Stat. Phys. {\bf 156} (2014), 268-290.


\bibitem{cf:EON11}
T.~R.~Einert, H.~Orland and R.~R.~Netz,
\emph{Secondary structure formation of homopolymeric single-stranded nucleic acids including force and loop entropy: implications for DNA hybridization}, Eur. Phys. J. E {\bf 34} (2011), 55 (15 pages).

\bibitem{cf:Fisher}
M.~E.~Fisher, \emph{Walks, walls, wetting, and melting},
J. Statist. Phys. {\bf 34} (1984), 667-729.

\bibitem{cf:GOO}
T.~Garel and H.~Orland, \emph{On the role of mismatches in DNA denaturation}, arXiv:cond-mat/0304080

\bibitem{cf:GO04} T.~Garel and H.~Orland, \emph{Generalized Poland-Scheraga model for DNA hybridization},  Biopolymers {\bf 75} (2004), 453-467.



\bibitem{cf:GB} G. Giacomin, {\sl Random polymer models}, 
Imperial College Press, World Scientific, 2007. 

\bibitem{cf:G08} G. Giacomin, \emph{Renewal convergence rates and correlation decay for homogeneous pinning models}
Elec. Jour. Probab. {\bf 13}, 2008,  513--529.



\bibitem{cf:G} G. Giacomin, \emph{Disorder and critical phenomena through basic probability models}, \'Ecole d'\'et\'e de probablit\'es de Saint-Flour XL-2010, Lecture Notes in Mathematics {\bf 2025}, Springer, 2011.


\bibitem{cf:GK}
G. Giacomin and M. Khatib, \emph{Generalized Poland Sheraga denaturation model and two dimensional renewal processes}, Stoch. Proc. Appl. {\bf 127} (2017), 526-573. 

\bibitem{cf:GLT10}
G. Giacomin, H. Lacoin and F. L. Toninelli, \emph{ Hierarchical pinning models, quadratic maps and quenched
disorder}, Probab. Theor. Rel. Fields {\bf 147} (2010), 185-216.


\bibitem{cf:1GLT10}
G. Giacomin, H. Lacoin and F. L. Toninelli, \emph{ Marginal relevance of disorder for pinning models}, Comm. Pure
Appl. Math. {\bf 63} (2010), 233-265.

\bibitem{cf:GLT11} G. Giacomin, H. Lacoin and F. L. Toninelli, \emph{ Disorder relevance at marginality and critical point shift}, Ann. Inst. H. Poincar\'e  {\bf 47} (2011), 148-175.

 \bibitem{cf:GT06}
G. Giacomin and F. L. Toninelli, \emph{ Smoothing effect of quenched disorder on polymer depinning transitions}, Commun. Math. Phys.  {\bf 266} (2006), 1-16.


 %\bibitem{cf:GT09} G. Giacomin and F. L. Toninelli, \emph{ On the irrelevant disorder regime of pinning models}, Ann. Probab.  {\bf 37} (2009), 1841-1875.

\bibitem{cf:GTalea}
G. Giacomin and F. L. Toninelli,
\emph{The localized phase of disordered copolymers with adsorption},
ALEA-Latin American Journal of Probability and Mathematical Statistics {\bf 1} (2006), 149-180.

\bibitem{cf:dH}
F.~den Hollander, \emph{Random polymers}, Lectures from the 37th Probability Summer School held in Saint-Flour, 2007. Lecture Notes in Mathematics {\bf 1974},  Springer-Verlag, 2009. 


\bibitem{cf:Hun}
J. Hunter, \emph{Renewal theory in two-dimensions: asymptotic results}, Advances in Applied Probability {\bf 6} (1974), 546–562 .


 
\bibitem{cf:KM03}
 Y. Kafri, D. Mukamel, \emph{Griffiths singularities in unbinding of strongly disordered polymers},  Phys. Rev. Lett. {\bf 91} (2003), 038103. 


%\bibitem{cf:King73}
% J. F. C. Kingman, \emph{Subadditive Ergodic Theory}, Ann. Probab. {\bf 1} (1973), 882-909.
 
\bibitem{cf:Khatib} M. Khatib, \emph{Le mod\`ele de Poland-Scheraga g\'en\'eralis\'e/une approche de renouvellement bidimensionel pour la d\'enaturation de l'ADN}, Ph.D.\ manuscript, 2016.
 
\bibitem{cf:KL12} H. Kunz, R. Livi, \emph{DNA denaturation and wetting in the presence of disorder}, Eur. Phys. Lett.  {\bf 99 } (2012), 30001.


\bibitem{cf:La10}
H. Lacoin, \emph{New bounds for the free energy of directed polymers in dimension 1 + 1 and 1 + 2}, Commun. Math. Phys. {\bf 294} (2010), 471-503. 

\bibitem{cf:Lac}
H. Lacoin, \emph{The martingale approach to disorder irrelevance for pinning models}, Elec. Comm. Probab. {\bf 15} (2010), 418-427.
 
\bibitem{cf:Martin02} J. B. Martin \textit{Linear growth for greedy lattice animals}, Stochastic Process. Appl. \textbf{98} (2002) 43-66. 
 
%\bibitem{cf:Martin04} J. B. Martin, \textit{ Limiting shape for directed percolation models},  Ann. Probab. \textbf{32} (2004), 2908-2937. 
 
\bibitem{cf:NG06}
 R. A. Neher and U. Gerland, \emph{Intermediate phase in DNA melting},  Phys. Rev. E {\bf 73 } (2006), 030902R.

\bibitem{cf:Peng} C.-K. Peng, S. V. Buldyrev, A. L. Goldberger, S. Havlin, F. Sciortino, M. Simons and H. E. Stanley
    \emph{Long-range correlations in nucleotide sequences} Nature {\bf356} (1992), 168-170.




\bibitem{cf:Pi81}
I. F. Pinelis, \emph{A problem on large deviations in a space of trajectories}, Theory Probab. Appl. {\bf 26} (1981), 69-84.
 
\bibitem{cf:PSbook} D. Poland and H. A. Scheraga, \emph{Theory of helix-coil transitions in biopolymers;: Statistical mechanical theory of order-disorder transitions in biological macromolecules}, Academic Press, 1970.

\bibitem{cf:SW11}
S. Shneer and V. Wachtel, \emph{A unified approach to the heavy-traffic analysis of the maximum of random
walks}, Theory Probab. Appl. {\bf 55} (2011), 332-341.

\bibitem{cf:TN08}
M.V. Tamm and  S.K. Nechaev, \emph{Unzipping of two random heteropolymers: Ground-state energy and finite-size effects}, Phys. Rev. E {\bf 78} (2008), 011903.

%\bibitem{cf:T08}
%F. L. Toninelli, \emph{Disordered pinning models and copolymers: beyond annealed bounds}, Ann. Appl. Probab. {\bf 18} (2008), 1569-1587.

\bibitem{cf:Ton08}
 F. L. Toninelli, \emph{A replica-coupling approach to disordered pinning models}, Commun. Math. Phys. {\bf 280} (2008), 389-401. 
 


\bibitem{cf:Wa12}
F. Watbled, \emph{Sharp asymptotics for the free energy of 1+1 dimensional directed polymers in an infinitely divisible
environment}, Elec. Commun. Probab. {\bf 17} (2012), 1-9. 

\bibitem{cf:We15}
R. Wei, \emph{On the Long-range directed polymer model}, J. Stat. Phys., \textbf{165} Issue 2 (2016), 320-350.

%\bibitem{cf:Wi68}
%J. A. Williamson, \emph{Random walks and Riesz kernels}, Pacific J. Math. {\bf 25} (1968), 393-415 .


\end{thebibliography}
\end{document}